\documentclass[11pt]{amsart}

\usepackage{amsfonts, amssymb, amscd}
\numberwithin{equation}{section}

\usepackage[symbol]{footmisc}

\usepackage{bm}
\usepackage{verbatim}
\usepackage{amssymb}
\usepackage{mathrsfs}
\usepackage{graphicx}
\usepackage[nospace,noadjust]{cite}
\usepackage{tikz-cd}
\usepackage{subcaption}
\usepackage{subfiles}
\usepackage[toc,page]{appendix}
\usepackage{mathtools}
\usepackage{comment}
\usepackage{enumerate}
\usepackage{enumitem}

\usepackage{graphicx}
\graphicspath{{images/}}

\usepackage{appendix}
\usepackage{hyperref}

\newcommand{\id}{\mathrm{id}}

\newcommand{\Cc}{\mathbb{C}}

\newcommand{\Pp}{\mathbb{P}}
\newcommand{\Qq}{\mathbb{Q}}

\newcommand{\Rr}{\mathbb{R}}

\newcommand{\Zz}{\mathbb{Z}}

\newcommand{\Span}{\operatorname{Span}}
\newcommand{\alct}{a\text{-}\operatorname{lct}}

\newcommand{\eRcomp}{\epsilon\text{-}\operatorname{Rcomp}}

\newcommand{\pld}{\operatorname{pld}}
\newcommand{\Center}{\operatorname{center}}

\newcommand{\mld}{{\rm{mld}}}

\newcommand{\lct}{\operatorname{lct}}

\newcommand{\Supp}{\operatorname{Supp}}

\newcommand{\mult}{\operatorname{mult}}

\newcommand{\Bb}{\mathcal{B}}

\newcommand{\Ff}{\mathcal{F}}
\newcommand{\Oo}{\mathcal{O}}
\newcommand{\Ii}{{\Gamma}}

\newcommand{\Ee}{\mathcal{E}}

\newcommand{\Ll}{\mathcal{L}}

 \usepackage{todonotes}

\newcommand\han[1]{\todo[color=yellow!40]{#1}} %Han

\newcommand\chen[1]{\todo[color=blue!40]{#1}} %Li

\newtheorem{thm}{Theorem}[section]
\newtheorem{conj}[thm]{Conjecture}
\newtheorem{cor}[thm]{Corollary}
\newtheorem{lem}[thm]{Lemma}

\newtheorem{prop}[thm]{Proposition}

\newtheorem{problem}[thm]{Problem}
\newtheorem{claim}[thm]{Claim}
\theoremstyle{definition}
\newtheorem{defn}[thm]{Definition}
\newtheorem{rem}[thm]{Remark}

\newtheorem{ex}[thm]{Example}

\theoremstyle{definition}

\begin{document}

\title{Boundedness of ($\epsilon, n$)-Complements for Surfaces}

\dedicatory{Dedicated to Vyacheslav Shokurov on the occasion of his seventieth birthday}
%\date{\today}
%\setcounter{footnote}{-1}
%Boundedness of Index of Exceptional Singularities}
% Dedicated to Gang Tian on the occasion of this sixtieth birthday ’s Sixtieth Birthday with admirationyachesla

\author{Guodu Chen}
\address{Guodu Chen, Beijing International Center for Mathematical Research, Peking University, Beijing 100871, China}
\email{1601110003@pku.edu.cn}

\author{Jingjun Han}
\address{Jingjun Han, Department of Mathematics, Johns Hopkins University, Baltimore, MD 21218, USA}
\email{jhan@math.jhu.edu}

\begin{abstract}
%	We show the existence of $(\epsilon,n)$-complements for $(\epsilon,\Rr)$-complementary curve pairs for boundaries with coefficients in $[0,1]$ and $(\epsilon,\Rr)$-complementary surface pairs for boundaries with DCC coefficients.
We show the existence of $(\epsilon,n)$-complements for $(\epsilon,\Rr)$-complementary surface pairs when the coefficients of boundaries belong to a DCC set. 
\end{abstract}

\date{\today}

\maketitle
\pagestyle{myheadings}\markboth{\hfill  G.Chen and J.Han \hfill}{\hfill Boundedness of ($\epsilon, n$)-Complements for Surfaces\hfill}

% admitting an $\epsilon$-plt blow up
\tableofcontents

\section{Introduction}

    We work over the field of complex numbers $\Cc$.
     
     The theory of complements was introduced by Shokurov when he proved the existence of log flips for threefolds \cite{Sho92}. It originates
from his earlier work on anti-canonical systems on Fano threefolds \cite{Sho79}. The theory is further developed in \cite{Sho00,PS01,PS09,Bir16a,HLS19,Sho19}, see \cite{FM18,Xuyanning19-1,Xuyanning19-2,FMX19} for recent works. 

At the first sight, the definition of complements seems to be technical. However, the boundedness of log canonical complements for Fano type varieties with hyperstandard sets $\Ii\subseteq [0,1]\cap \Qq$ proved by Birkar \cite{Bir16a} is a breakthrough in the study of Fano varieties. Birkar's result plays an important role in the proof of Birkar-Borisov-Alexeev-Borisov Theorem \cite{Bir16a,Bir16b}, the boundedness of log Calabi-Yau fibrations \cite{Bir18}, Jonsson--Musta\c{t}\v{a}'s conjecture on graded sequence of ideals which implies an algebraic proof of Demailly--Koll\'{a}r's openness conjecture \cite{Xu19}, Chi Li's conjecture on minimizers of the normalized volumes \cite{Xu19}, and the openness of K-semistability in families of log Fano pairs \cite{Xu19,BLX19}.

  The previous researches on complements were mainly focused on the existence of log canonical complements for Fano type varieties with sets $\Gamma=\bar{\Gamma}$ of rational numbers, where $\bar{\Gamma}$ is the closure of $\Gamma$. In this paper, we will study the existence of $\epsilon$-log canonical complements for Fano type and Calabi--Yau varieties with sets $\Gamma$ of real numbers. Recall that the DCC stands for the descending chain condition, and the ACC stands for the ascending chain condition. The following is a generalization of Shokurov's conjecture on the existence of $(\epsilon,n)$-complements (see \cite{Sho04comp}, \cite[Conjecture 1.3, Conjecture 1.4]{Bir04}). 
     
\begin{conj}[Existence of $(\epsilon,n)$-complements]\label{conj a complement}
	Let $\epsilon$ be a non-negative real number, $d,p$ two positive integers, and $\Ii \subseteq [0,1]$ a DCC set. Then there exists a positive integer $p|n$ depending only on $d, p, \epsilon$ and $\Gamma$ satisfying the following.
	
	Assume that $(X,B)$ is a pair, $X\to Z$ is a contraction and $z\in Z$ is a $($not necessarily closed$)$ point such that
	\begin{enumerate}     
		\item $\dim X=d$,
		\item either $\epsilon=0$ or $-K_X$ is big over $Z$,
		\item $B\in \Ii$, that is, the coefficients of $B$ belong to $\Ii$, and    
		\item $(X/Z\ni z,B)$ is $(\epsilon,\Rr)$-complementary.     
	\end{enumerate} 
	Then there is an $(\epsilon,n)$-complement $(X/Z\ni z,B^{+})$ of $(X/Z\ni z,B).$ Moreover, if $\Span_{\Qq_{\ge0}}(\bar{\Ii}\cup\{\epsilon\}\backslash\Qq)\cap (\Qq\backslash\{0\})=\emptyset$, then we may pick $B^+\ge B$.
\end{conj}

\begin{rem}\label{rem: conditions for complements fano type}	
	We weaken conditions ``(2)' $-(K_X+B)$ is big and nef'' and ``(4)' $(X,B)$ is $\epsilon$-lc'' in conjectures proposed in \cite{Sho04comp}, \cite[Conjecture 1.3, Conjecture 1.4]{Bir04} to conditions (2) and (4), strengthen the conclusion ``$n$ belongs to a finite set'' to ``a positive integer $p|n$'', and additionally have the ``Moreover'' part. The ``Moreover'' part is about monotonicity property of complements which is useful in applications especially when $\bar{\Ii}$ is a set of rational numbers and does not hold in general. Shokurov informed us that condition (4) is an analog of the existence of $\Rr$-complements in \cite{Sho19} which he believes to be the correct assumption for the boundedness of log canonical complements after a decade of research.
\end{rem}     %4) is an analog of existence of an $\R$-complement in \cite{HLS19,Sho19}. We point out that 4) is an analog of the existence of $\Rr$-complements in \cite{HLS19,Sho19} which is believed to be the correct assumption for the boundedness of log canonical complements after a long time study. 

\begin{rem}

According to the minimal model program, varieties of general type, Fano varieties and Calabi--Yau varieties form three fundamental classes in birational geometry as building blocks of algebraic varieties, and the boundedness of complements for Calabi--Yau varieties implies the boundedness of global index (see Proposition \ref{prop:global index}) which is related to the boundedness of Calabi--Yau varieties, one of the central open problems in biratioanl geometry. Thus it is also very natural to study the theory of complements for non Fano type varieties. 

Philosophically, the boundedness of complements should hold not only for Fano type varieties. However, it is not completed clear what are the most general settings. % for higher dimension varieties.	
\end{rem}

\begin{rem}     
For the purpose of induction in birational geometry, we need the theory of complements for infinite sets $\Ii$ when we apply the adjunction formula. In order to prove the existence of complements for infinite sets $\Ii$, a frequently-used strategy is to replace $\Ii$ by some finite set which is contained in the closure of $\Ii$, see \cite{PS09,Bir16a,FM18,HLS19,FMX19}. Thus we need to study the theory of complements when $\Ii$ is a subset of real numbers. Moreover, the theory of complements for sets $\Ii$ of real numbers also appears naturally in the study of the ACC for minimal log discrepancies even when the coefficients of boundaries belong to a finite rational set as the accumulation points of minimal log discrepancies might be irrational numbers \cite{LJH18}. The boundedness of log canonical complements with some Diophantine approximation properties for Fano type varieties with any DCC set $\Ii\subseteq[0,1]$ was established in \cite{HLS19}. The theory of complements in this case could be applied to prove the ACC for minimal log discrepancies of exceptional singularities \cite{HLS19}, and it implies some known but important results, such as the Global ACC \cite{Sho19}, the ACC for log canonical thresholds \cite{Sho19}, the ACC for Fano indices \cite{Sho19}. 
\end{rem}
     
     \begin{rem}
     	When $\epsilon>0$, Conjecture \ref{conj a complement} can be regarded as a relative version of Birkar-Borisov-Alexeev-Borisov Theorem. It implies Birkar-Borisov-Alexeev-Borisov Theorem (Proposition \ref{Prop: comp conj implies bbab}), Shokurov's index conjecture (Proposition \ref{prop: epsiloncomp conj implies index}) and Shokurov-M\textsuperscript{c}Kernan's conjecture on Fano type fibrations (Proposition \ref{prop: epsilon comp conj implies MS conj}). We refer readers to \cite{HLS19,Sho19} for other applications of Conjecture \ref{conj a complement}. These applications also provide positive evidences for Conjecture \ref{conj a complement}.
     	\end{rem}
     
     In this paper, we show the following result which implies that Conjecture \ref{conj a complement} holds for surfaces. 
     
\begin{thm} \label{thm epsilonlccompforsurfaces}    
     Let $\epsilon$ be a non-negative real number, $p,M$ two positive integers and $\Ii\subseteq[0,1]$ a DCC set. Then there exists a positive integer $p|n$ depending only on $\epsilon,p,M$ and $\Ii$ satisfying the following.

     Assume that $(X/Z\ni z,B)$ is a surface pair such that
     \begin{enumerate}     
       \item $B\in \Ii$, 
       \item $(X/Z\ni z,B)$ is $(\epsilon,\Rr)$-complementary, and
       \item either $\epsilon=0$ or the multiplicity of any fiber of any minimal elliptic fibration of the minimal resolution of $X$ over $Z$ is bounded from above by $M$.
      \end{enumerate}
      Then there is an $(\epsilon,n)$-complement $(X/Z\ni z,B^{+})$ of $(X/Z\ni z,B).$ Moreover, if $\Span_{\Qq_{\ge0}}(\bar{\Ii}\cup\{\epsilon\}\backslash\Qq)\cap (\Qq\backslash\{0\})=\emptyset$, then we may pick $B^+\ge B.$ In particular, Conjecture \ref{conj a complement} holds for surfaces. 
\end{thm}
     %The main difficulties in the proof Theorem \ref{thm epsilonlccompforsurfaces} cause from 
     
         We remark that in Theorem \ref{thm epsilonlccompforsurfaces}, $-K_X$ is not assumed to be big over $Z$. Birkar proved the existence of $(\delta,n)$-complements of $(X/Z\ni z,B)$ for some $\delta\le \epsilon$ depending only on $\epsilon,\Ii$ in Theorem \ref{thm epsilonlccompforsurfaces} when $p=1$, $\epsilon$ is a positive real number, $\Ii$ is the standard set and conditions $(2)$ and $(3)$ are replaced by stronger conditions (2)' and (4)' in Remark \ref{rem: conditions for complements fano type}. Conjecture \ref{conj a complement} is still widely open when $\epsilon>0$ and $\dim X\ge 3$, and just a few special cases are known. Kawatika proved Conjecture \ref{conj a complement} for the case when $\dim X=3$, $\Gamma=\{0\}$, $\epsilon=1$, $X=Z$, $z$ is a closed point, and $\mld (X\ni z,B)=1$ \cite{Kawakita15index}. When $\dim Z=0$ and $X$ is of Fano type over $Z$, one could prove Conjecture \ref{conj a complement} by Birkar-Borisov-Alexeev-Borisov Theorem and by using Diophantine approximation techniques for complements developed in \cite{HLS19}, see \cite[Theorem 1.3]{FM18} for the case when $\bar{\Ii}\subseteq\Qq$. 
     
 Without condition (3) on the boundedness of multiple fibers, that is Shokurov's conjecture on the boundedness of $\epsilon$-complements under condition $(c')$ in \cite{Sho04comp}, Theorem \ref{thm epsilonlccompforsurfaces} no longer holds if $\epsilon>0$ even when $\dim Z=0$, see Example \ref{ex: counterexample without bdd multiple fiber} below. %Thus  Theorem \ref{thm epsilonlccompforsurfaces} is strong in the theory of complements for surfaces.
     
     \begin{ex}\label{ex: counterexample without bdd multiple fiber}
     	Let $f_m:X_m\to \Pp^1$ be a relatively minimal rational elliptic surface with only one multiple fiber $E_m$ of multiplicity $m\ge 2$ (\cite[Proposition 1.1]{Fujimoto90}). The canonical bundle formula of Kodaria \cite{Kodaria60,Kodaira63} implies that
     	%$$K_{X}+(1-\epsilon)E_m+\frac{\epsilon}{m} F_m\sim f_m^{*}\mathcal{O}_{\Pp^1}(-1)+(m-\epsilon)E_m+\frac{\epsilon}{m} F_m\sim_{\Rr} 0,$$
     	$$K_{X_m}+\frac{1}{km}\sum_{i=1}^k F_{m,i}\sim f_m^{*}\mathcal{O}_{\Pp^1}(-1)+(m-1)E_m+\frac{1}{km}\sum_{i=1}^k F_{m,i}\sim_{\Rr} 0,$$
     	where $F_{m,i}$ are general fibers of $f_m$. Thus $(X_m,0)$ is $(\epsilon,\Rr)$-complementary for any $0\le \epsilon<1$. However, $(X_m,0)$ is not $(\epsilon,n)$-complementary for any $n<m$ and $0<\epsilon<1$. 
     \end{ex}
     
     When $\epsilon=0$ and $\bar{\Ii}\subseteq\Qq$, there are two cases in the proof of the existence of log canonical complements \cite{PS01,PS09,Bir16a}, that is 1) non-exceptional pairs, and 2) exceptional pairs. For non-exceptional pairs, the strategy is to construct lc centers, apply the induction on lc centers, and lift complements from lc centers by using Kawamata--Viehweg vanishing theorem. However when $\epsilon>0$, it is possible that there is no lc center, and when $\bar{\Ii}$ contains some irrational numbers, there are some technique difficulties on lifting complements. Instead of applying the induction and lifting complements, we construct some ``good'' birational model, show the existence of decomposed $(\epsilon,\Rr)$-complements for surfaces (Theorem \ref{thm: decompsable epsilon complement for surfaces}) and the existence of uniform rational polytopes for linearity of minimal log discrepancies for surface germs (Theorem \ref{thm:uniformpolytopeforsurfacemlds}). It is expected that both Theorem \ref{thm: decompsable epsilon complement for surfaces} and Theorem \ref{thm:uniformpolytopeforsurfacemlds} should hold in higher dimensions, and will be used in the proof of Conjecture \ref{conj a complement}. When $\epsilon$ is an irrational number, we need to overcome another difficulty. In this case, if $\mld(X/Z\ni z,B)=\epsilon$, then $\mld(X/Z\ni z,B^{+})$ is always strictly larger than $\mld(X/Z\ni z,B)$ as $n\mld(X/Z\ni z,B^+)\ge\lceil n\epsilon \rceil$. Such a phenomenon will never happen when either $\epsilon=0$ or $\Ii\subseteq \Qq$, and thus it is out of the scope of conjectures proposed in \cite{Sho04comp}, \cite[Conjecture 1.3, Conjecture 1.4]{Bir04}. Here we need to modify Diophantine approximation techniques for complements developed in \cite{HLS19}. The other difficulty causes from that $X$ is not necessarily of Fano type over $Z$ and it is not always the case that we may run a $D$-MMP for any $\Rr$-Cartier divisor $D$. In this paper, we will use \cite{Fuj12} to run MMPs for non-lc pairs.

\medskip

     As an easy corollary of Theorem \ref{thm epsilonlccompforsurfaces}, if the number of components of $B$ is bounded from above, then we may show the boundedness of complements without assuming $\Ii$ is a DCC set. %for any $\Ii\subseteq[0,1]$, if $0$ is not an accumulation point of $\Ii$, then we may show the boundedness of complements.  
\begin{cor} \label{cor main}   
     Let $\epsilon$ be a non-negative real number and $p,s,M$ positive integers. Then there exists a positive integer $N$ depending only on $\epsilon,p,s$ and $M$ satisfying the following.

     Assume that $(X/Z\ni z,B)$ is a surface pair such that
     \begin{enumerate} 
       \item the number of components of $B$ is at most $s$,
       \item $(X/Z\ni z,B)$ is $(\epsilon,\Rr)$-complementary, and
       \item either $\epsilon=0$ or the multiplicity of any fiber of any minimal elliptic fibration of the minimal resolution of $X$ over $Z$ is bounded from above by $M$.
      \end{enumerate}
      Then $(X/Z\ni z,B)$ is $(\epsilon,n)$-complementary for some positive integer $n$ satisfying $n\le N$ and $p|n.$ 
      %Moreover, if $\Span_{\Qq\ge0}(\bar{\Ii}\cup\{\epsilon\}\backslash\Qq)\cap (\Qq\backslash\{0\})=\emptyset$, we may pick $B^+\ge B.$
\end{cor}
\begin{comment}
\begin{cor} \label{cor main}    
     Let $\epsilon$ be a positive real number, $p,M$ two positive integers and $\Ii\subseteq[0,1]$ such that $0$ is not an accumulation point of $\Ii$. Then there exists a positive integer $N$ which only depends on $\epsilon,p,M$ and $\Ii$ satisfying the following.

     Assume $(X/Z\ni z,B)$ is a surface pair, such that
     \begin{enumerate}     
       \item $B\in \Ii$, 
       \item $(X/Z\ni z,B)$ is $(\epsilon,\Rr)$-complementary, and
       \item the multiplicity of any fiber of any minimal elliptic fibration of the minimal resolution of $X$ is bounded above by $M$.
      \end{enumerate}
      Then $(X/Z\ni z,B)$ is $(\epsilon,n)$-complementary for some $n$ satisfying $n\le N$ and $p|n$.
\end{cor}     
\end{comment}
       % there is an  $(\epsilon,n)$-complementary      $(X/Z\ni z,B^{+})$ of $(X/Z\ni z,B)$
As we mentioned before, it is expected that Conjecture \ref{conj a complement} should hold in a more general setting, that is, $-K_X$ is not necessarily big over $Z$ and $\Gamma$ is not necessarily a DCC set. In dimension one, Birkar \cite[Theorem 3.1]{Bir04} showed $(X/Z\ni z,B)$ is $(\frac{1}{m+1},n)$-complementary if $\frac{1}{m+1}<\epsilon\le \frac{1}{m}$ for some positive integer $m$. In fact we can show the following stronger result.

\begin{thm}\label{thm: complementdim1}
     Let $\epsilon$ be a non-negative real number and $p$ a positive integer. Then there exists a positive integer $N$ depending only on $\epsilon$ and $p$ satisfying the following.
     
     If $\dim X=1$ and $(X/Z\ni z,B)$ is $(\epsilon,\Rr)$-complementary, then
     $(X/Z\ni z,B)$ is $(\epsilon,n)$-complementary for some positive integer $n$ satisfying $n\le N$ and $p|n$. 
\end{thm}     
     
\medskip

\textit{Structure of the paper.}
We outline the organization of the paper. In Section
2, we give a sketch of the proofs of our main theorems. In Section 3, we introduce some notation and tools which will be used in this paper, and prove certain basic results. In Section 4, we prove Theorem \ref{thm:uniformpolytopeforsurfacemlds}. In Section 5, we prove Theorem \ref{thm: decompsable epsilon complement for surfaces}. In Section 6, we prove Theorem \ref{thm epsilonlccompforsurfaces}, Corollary \ref{cor main} and Theorem \ref{thm: complementdim1}. In Section 7, we will show some applications of Conjecture \ref{conj a complement} and give a list of open problems. In Appendix A, we collect some known or folklore results for surfaces. In Appendix B, we show Nakamura's conjecture \cite[Conjecture 1.1]{mustata-nakamura18} holds for surface pairs with DCC coefficients.  

\medskip

\textbf{Acknowledgements}. This work was partially supported by China Scholarship Council (File No. 201806010039) and AMS Simons Travel Grant. This work began when the first author visited the second author at Johns Hopkins University in April of 2019. Part of this work was done while the first author visited the MIT Mathematics Department during 2018--2020. The first author would like to thank their hospitality. The authors would like to thank Caucher Birkar whose preprint \cite{Bir04} inspired them a lot, Vyacheslav V. Shokurov who teaches them the theory of complements, and Chenyang Xu, their advisor, for his constant support, valuable discussions and suggestions. The authors would also like to thank Qianyu Chen, Stefano Filipazzi, Yang He, Kai Huang, Chen Jiang, Jihao Liu, Wenfei Liu, Yujie Luo, Joaqu\'{i}n Moraga, Yusuke Nakamura and Lu Qi for very useful discussions.

\section{Sketch of the proof of Theorem \ref{thm epsilonlccompforsurfaces}}
%, that is we need to reduce Theorem \ref{thm main} to the case when $\Ii$ is a finite set.
For simplicity, we assume that $X$ is of Fano type over $Z$, that is, there is a boundary $\Delta$ such that $(X,\Delta)$ is a klt pair and $-(K_X+\Delta)$ is ample over $Z$. In this case, we may run a $D$-MMP for any $\Rr$-Cartier divisor $D$. By lower semi-continuity for relative minimal log discrepancies $\mld(X/Z\ni z,B)$ for surfaces (Proposition \ref{prop:equivrelmld}), we may assume that $z$ is a closed point (Proposition \ref{prop:dec eps com for closed point imply gene}). 

The next step in the proof of Theorem \ref{thm epsilonlccompforsurfaces} is similar to \cite{HLS19}. Since the ACC for $a$-log canonical thresholds is known for surfaces, by using the same arguments as in \cite{HLS19}, it suffices to show the case when $\Ii\subseteq [0,1]$ is a finite set, see Theorem \ref{thm: dcc limit epsilon-lc complementary}.

Then we want to reduce Theorem \ref{thm epsilonlccompforsurfaces} to the case when $\Ii\subseteq [0,1]\cap \Qq$. If $\epsilon$ is a rational number, then by using ideas in \cite{HLS19}, it is possible to show the existence of rational  $(\epsilon,\Rr)$-complementary polytopes. Thus by using Diophantine approximation \cite{HLS19}, we may assume that $\Ii\subseteq [0,1]\cap \Qq$. When $\epsilon$ is irrational, the rational $(\epsilon,\Rr)$-complementary polytopes may not exist any more, even for the global case, that is, $\dim Z=0$. For example, suppose that $X=\Pp^1\times \Pp^1$, $c\in (0,1)$ is an irrational number, $\epsilon=1-c$, $\Ii=\{0,c\}$. Then $\mld(X,c\Pp^1)=1-c$, and $\mld(X,c'\Pp^1)=1-c'<1-c$ for any $1>c'>c$. This is one of the main difficulties for us to generalize the method in \cite{HLS19} to our setting. 

As one of main ingredients in the proof of Theorem \ref{thm epsilonlccompforsurfaces}, we will show the following theorem on the existence of decomposed $(\epsilon,\Rr)$-complements for surfaces.
\begin{thm}\label{thm: decompsable epsilon complement for surfaces}
	Let $\epsilon$ be a non-negative real number and $\Ii\subseteq[0,1]$ a DCC set. Then there exist finite sets $\Ii_1\subseteq(0,1],\Ii_2\subseteq[0,1]\cap\Qq$ and $\mathcal{M}\subseteq\Qq_{\ge0}$ depending only on $\Ii$ and $\epsilon$ satisfying the following.
	
	Assume that $(X/Z\ni z,B)$ is a surface pair such that
	\begin{itemize}
		\item $B\in\Ii$, and
		%\item $z\in Z$ is a closed point, 
		\item $(X/Z\ni z,B)$ is $(\epsilon,\Rr)$-complementary.
	\end{itemize}
	Then there exist $a_i\in\Ii_1$, $B^i\in\Ii_2$ and $\epsilon_i\in\mathcal{M}$ with the following properties:
	\begin{enumerate}
		\item $\sum a_i=1,$
		\item $\sum a_i\epsilon_i\ge\epsilon$,
		\item $\sum a_i(K_X+B^i)\ge K_X+B$,
		\item $(X/Z\ni z,B^i)$ is $(\epsilon_i,\Rr)$-complementary for any $i$, and
		\item if $\epsilon\in\Qq$, then $\epsilon=\epsilon_i$ for any $i$, and if $\epsilon\neq 1$, then $\epsilon_i\neq1$ for any $i$.
	\end{enumerate}
	%Moreover, if $\epsilon\in\Qq$, then we may let $\mathcal{M}\subseteq\Qq$, and if $\epsilon\neq 1$, then we may let $1\notin \mathcal{M}$. \han{check!!!}
\end{thm}     

Note that we don't require $-K_X$ to be big over $Z$ in Theorem \ref{thm: decompsable epsilon complement for surfaces}. Theorem \ref{thm: decompsable epsilon complement for surfaces} is conjectured to hold in higher dimensions (Problem \ref{problem: decompsable epsilon complements}). Although Theorem \ref{thm: decompsable epsilon complement for surfaces} and Problem \ref{problem: decompsable epsilon complements} seem to be technical, we expect that Problem \ref{problem: decompsable epsilon complements} should be useful in birational geometry. Indeed Problem \ref{problem: decompsable epsilon complements} implies the Global ACC and the ACC for $a$-log canonical thresholds, see Proposition \ref{prop: decomposed epsilon comp imply acc for alct}.

When $Z=X$ and $z\in X$ is a closed point, Theorem \ref{thm: decompsable epsilon complement for surfaces} is an easy corollary of the existence of uniform rational polytopes for linearity of minimal log discrepancies for surface germs.

\begin{thm}\label{thm:uniformpolytopeforsurfacemlds}
	Let $\epsilon$ be a positive real number, $m$ a positive integer and $\bm{v}=(v_1^0,\ldots,v_m^0)\in\Rr^m$ a point. Then there exist a rational polytope $\bm{v}\in P\subseteq\Rr^m$ with vertices $\bm{v}_j=(v_1^j,\ldots,v_m^j)$, positive real numbers $a_j$ and positive rational numbers $\epsilon_j$ depending only on $\epsilon,m$ and $\bm{v}$ satisfying the following.  
	\begin{enumerate}
		\item $\sum a_j=1,\sum a_j\bm{v}_j=\bm{v}$ and $\sum a_j\epsilon_j\ge\epsilon$.
		\item Assume that $(X\ni x,\sum_{i=1}^m v_i^0B_{i})$ is a surface germ which is $\epsilon$-lc at $x$, where $B_1,\ldots,B_m\ge0$ are Weil divisors on $X$. Then the function $P\to \Rr$ defined by
		$$(v_1,\ldots,v_m)\mapsto\mld(X\ni x,\sum_{i=1}^m v_iB_{i})$$
		is a linear function, and
		$$\mld(X\ni x,\sum_{i=1}^m v_i^jB_{i})\ge \epsilon_j$$
		for any $j$.	
	\end{enumerate}
	Moreover, if $\epsilon\in \Qq$, then we may pick $\epsilon_j=\epsilon$ for any $j$.
\end{thm}

We also expect that Theorem \ref{thm:uniformpolytopeforsurfacemlds} should hold in higher dimensions, see Problem \ref{conj: linearmld}. 

Theorem \ref{thm:uniformpolytopeforsurfacemlds} follows from the classification of surface singularities. In order to prove Theorem \ref{thm:uniformpolytopeforsurfacemlds}, we need to generalize Alexeev's result on minimal log discrepancies for surfaces \cite{Ale93} to linear functional divisors, we also need to prove a generalization of Nakamura's conjecture for surfaces, that is, when the coefficients belong to either a finite set of linear functions or a DCC set.

\medskip

For the general case, Theorem \ref{thm: decompsable epsilon complement for surfaces} does not follow from Theorem \ref{thm:uniformpolytopeforsurfacemlds} directly since being $\epsilon_j$-lc over $z$ does not imply that $(X/Z\ni z,\sum_{i=1}^m v_i^jB_{i})$ itself is $(\epsilon_j,\Rr)$-complementary. %\han{could you give any example?}$(X/Z\ni z,\sum_{i=1}^m v_i^jB_{i})$ is

The idea to prove Theorem \ref{thm: decompsable epsilon complement for surfaces} is to construct a birational model $(\hat{X},\hat{B})$ of $(X,B)$ by using the minimal model program, such that $\hat{f}:\hat{X}\to X$ is a contraction, $-(K_{\hat{X}}+\hat{B})$ is semiample over $Z$, the coefficients of $\hat{B}$ belong to a finite set $\hat{\Gamma}$, the PLDs of $(\hat{X},\hat{B})$ belong to a finite set, and $\hat{f}^{*}(K_X+B)\le K_{\hat{X}}+\hat{B}$. Thus we may replace $(X,B)$ by $(\hat{X},\hat{B})$, and $\Ii$ by $\hat{\Ii}$. Indeed we will show that we may take $\hat{X}$ to be the minimal resolution of $X$, see Theorem \ref{thm reducetofiniteandfixpld}. Here we use some ideas from \cite{Bir04}. Now since $-(K_X+B)$ is semiample over $Z$ and the PLDs of $(X,B)$ belong to a finite set, we may finish the proof of Theorem \ref{thm: decompsable epsilon complement for surfaces}, see Lemma \ref{lemma antinefpolytope2}.

\medskip

Since $X$ is of Fano type over $Z$, we may additionally assume that $-N(K_X+\sum_{i=1}^m v_i^jB_{i})$ is base point free for some positive integer $N$ which only depends on $\epsilon$ and $\Ii$. Hence $(X/Z\ni z,\sum_{i=1}^m v_i^jB_{i})$ is $(\epsilon_j,N)$-complementary. Our final step is to show that
$$K_X+B':=\sum_j a_j'(K_X+\sum_{i=1}^m v_i^jB_{i})$$
is $(\epsilon,n)$-complementary for some positive rational numbers $a_j'$ and positive integer $n$ which only depend on $\Ii$ and $\epsilon$, and any $(\epsilon,n)$-complement of $(X/Z\ni z,B')$ is also an $(\epsilon,n)$-complement of $(X/Z\ni z,B)$ by Diophantine approximation, see Lemma \ref{lemma rational direction2} and Lemma \ref{lem:finitesetmonotonic}. Thus we finish the proof of Theorem \ref{thm epsilonlccompforsurfaces}. For this step, we use ideas from \cite{HLS19}. 

\section{Preliminaries}
 
\subsection{Pairs and singularities }

We adopt the standard notation and definitions in \cite{KM98} and \cite{BCHM10}, and will freely use them.

\begin{defn}   
     We say $\pi: X \to Z$ is a \emph{contraction} if 
     $X$ and $Z$ are normal quasi-projective varieties, $\pi$ is a projective morphism, and $\pi_*\Oo_X = \Oo_Z$ $(\pi$ is not necessarily birational$)$. %In particular, $\pi$ is surjective and has connected fibers.
\end{defn}

\begin{defn} \label{defn sing}
 A \emph{sub-pair} $(X/Z\ni z, B)$ consists of normal quasi-projective varieties $X,Z$, a contraction $\pi: X\to Z$, a point $z\in Z$, and an $\Rr$-divisor $B$ on $X$ with coefficients in $(-\infty,1]$ such that $K_X + B$ is $\Rr$-Cartier and $\dim z<\dim X$. A sub-pair $(X/Z\ni z,B)$ is called a \emph{pair} if $B\ge0$. We say that a pair $(X/Z\ni z, B)$ is a \emph{germ} (near $z$), if $z$ is a closed point. %\han{Please check which one is better.} 
     %When $z$ is the generic point of $Z$, for convenience we will omit $z$. 
     When $Z=X$ and $\pi$ is the identity map, we will omit $Z$, let $x=z$, and say $(X\ni x, B)$ is a pair. When $\dim Z=\dim z=0$, we will omit both $Z$ and $z$, and say $(X,B)$ is a pair, in this case, our definition of pairs coincides with the usual definition of pairs \cite{BCHM10}. 
               
     Let $E$ be a prime divisor on $X$, $B:=\sum b_iB_i$ an $\Rr$-divisor on $X$, where $B_i$ are distinct prime divisors. We define $\mult_E B$ to be the multiplicity of $E$ along $B$, $\lfloor B\rfloor:=\sum \lfloor b_i\rfloor B_i, \{B\}:=\sum \{b_i\}B_i$, $||B||:=\max_i\{|b_i|\}$, $\mult_E B:=\sum b_i\mult_E B_i$ and $\mult_E\Supp B:=\sum \mult_E B_i.$ For any point $\bm{v}=(v_1,\dots,v_m)\in \Rr^m,$ we define $||\bm{v}||_{\infty}:=\max_{1\le i\le m}\{|v_i|\}$.
     
     Let $\phi :W\to X$ be any log resolution of $(X, B)$ and write $K_W +B_W:=\phi^{*}(K_X +B)$. The \emph{log discrepancy} of a prime divisor $D$ on $W$ with  respect to $(X, B)$ is $1-\mult_DB_W$ and denoted by $a(D, X, B)$. 

     Assume that $(X/Z\ni z,B)$ is a pair, let 
     $$e(z):=\{E\mid E\text{ is a prime divisor over }X, \pi(\Center_{X}(E))=\bar{z}\}.$$
     We define
     $$\mld(X/Z\ni z,B):=\inf_{E\in e(z)}\{a(E,X,B)\}.$$
     %otherwise, we define $\mld(X/Z\ni z,B):=0.$
     By Lemma \ref{lem:mldisattained}, the infimum can be replaced by the minimum. 
     
     We say that $(X/Z\ni z, B)$ is $\epsilon$-lc $($respectively klt, lc$)$ over $z$ for some non-negative real number $\epsilon$ if $\mld(X/Z\ni z,B)\ge \epsilon$ $($respectively $>0,\ge0)$. 
     In particular, when $\dim Z=\dim z=0$ the definition of singularities coincides with the usual definition. For simplicity, when $X=Z$, and $\pi$ is the identity map, we let $x=z$, and use $\mld(X\ni x,B)$ instead of $\mld(X/Z\ni z,B)$. In this case, we will say that $(X\ni x, B)$ is $\epsilon$-lc $($respectively klt, lc$)$ at $x$ instead of $(X\ni x, B)$ is $\epsilon$-lc $($respectively klt, lc$)$ over $x.$
\end{defn}

     It is well-known that if $(X/Z\ni z, B)$ is lc over $z$, then $(X,B)$ is lc near $\pi^{-1}(z)$.
     
\medskip

The following lemma is well-known to experts (c.f. \cite[Example 1]{Sho04}). For the reader's convenience, we give a proof here.
\begin{lem}\label{lem:logsmoothmld}
    Suppose that $(X\ni x,B)$ is a log smooth sub-pair. Then $\mld(X\ni x,B)={\rm codim}\ x-\mult_x B$.
\end{lem}
\begin{proof}
    If ${\rm codim}\ x=1$, then $\mld(X\ni x,B)=a(\bar{x},X,B)={\rm codim}\ x-\mult_x B$. If ${\rm codim}\ x>1$, then by \cite[Lemma 2.45]{KM98}, there exists a sequence of blow-ups:
     $$X_n\to \cdots\to X_1\to X_0=:X,$$
     such that
     \begin{enumerate}
      \item $f_i:X_i\to X_{i-1}$ is the blow-up of $\bar{x_{i-1}}\subseteq X_{i-1}$ with the exceptional divisor $E_i$ for any $1\le i\le n,$ where $x_0:=x$,
      \item $x_i\in E_i$ for any $1\le i\le n-1$, and
      \item $a(E_i,X,B)>a(E_n,X,B)=\mld(X\ni x,B)$ for any $1\le i\le n-1$.
     \end{enumerate}
     We may write 
     $$K_{X_i}+\tilde{B_{X_i}}:=K_{X_i}+B_{X_i}+(1-a_i)E_i=f_i^*(K_{X_{i-1}}+\tilde{B_{X_{i-1}}}),$$
     where $a_i:=a(E_i,X,B),B_{X_i}$ is the strict transform of $\tilde{B_{X_{i-1}}}$ for any $1\le i\le n$ and $\tilde{B_{X_0}}:=B$. We will show by induction on $k$ that $a_k\ge a_1$ for any $1\le k\le n$ and hence we conclude that $\mld(X\ni x,B)=a_1={\rm codim }\ x-\mult_x B.$ Assume that for some $1\le k_0\le n-1,a_i\ge a_1$ for any $1\le i\le k_0.$ By induction, we have
     \begin{align*}
      a_{k_0+1}&={\rm codim }\ x_{k_0}-\mult_{x_{k_0}}\tilde{B_{X_{k_0}}}\\&={\rm codim}\ x_{k_0}-\mult_{x_{k_0}}{B_{X_{k_0}}}-(1-a_{k_0})
       \\&\ge a_{k_0}\ge a_1.
     \end{align*}
     The first inequality holds since $(X_{k_0},B_{X_{k_0}}+(1-a_{k_0})E_{k_0})$ is a log smooth sub-pair and $x_{k_0}\in E_{k_0}$. We finish the induction.
\end{proof}

\begin{defn}
     A \emph{minimal log discrepancy place} or \emph{mld place} of a pair $(X/Z\ni z,B)$ is a divisor $E\in e(z)$ with $\mld(X/Z\ni z,B)=a(E,X,B)$, and a \emph{minimal log discrepancy center} or \emph{mld center} is the image of a mld place. 
\end{defn}

\begin{lem}\label{lem:mldisattained}
      Suppose that $(X/Z\ni z,B)$ is lc over $z$. Then there exists a prime divisor $E\in e(z)$ such that $\mld(X/Z\ni z,B)=a(E,X,B)$.
\end{lem}

\begin{proof}
     Let $f:Y\to X$ be a log resolution of $(X,B),$ and we may write
     $K_Y+B_Y:=f^*(K_X+B).$
     %where $a_i=a(B_i,X,B)$ for any $i\in \mathcal{I}.$
     Shrinking $Z$ near $z,$ we may assume that $(Y,B_Y)$ is a log smooth sub-pair. Let $\mu:=\pi\circ f$, where $\pi$ is the morphism $X\to Z$. Now
     \begin{align*}\mld(X/Z\ni z,B)=\inf_{E\in e(z)}\{a(E,X,B)\}=\inf_{y\in \mu^{-1}(z)}\{\mld(Y\ni y,B_Y)\}
     \end{align*}
     which belongs to the finite set 
     $\{{\rm codim}\ y-\mult_yB_Y\mid y\in\mu^{-1}(z)\}$ by Lemma \ref{lem:logsmoothmld}. Hence there exist a point $y\in\pi^{-1}(z)$ and a prime divisor $E\in e(z)$ such that $\mld(X/Z\ni z,B)=\mld(Y\ni y,B_Y)=a(E,X,B)$.
\end{proof}

\begin{lem}\label{lemma mldcenterfinite}
      Suppose that $(X/Z\ni z,B)$ is lc over $z$, and either
      \begin{enumerate}
          \item $\dim Z>\dim z$, or
          \item $\dim Z=\dim z$, and $\mld(X/Z\ni z,B)<1.$
      \end{enumerate} 
      Then $(X/Z\ni z,B)$ has only finitely many mld centers. %In particular, we have $\mld(X/Z\ni z,B)=\mld(X\ni x,B)$ for some $x\in \pi^{-1}(z)$, where $\pi$ is the morphsim $X\to Z$.
\end{lem}

\begin{proof}
     Let $f:Y\to X$ be a log resolution of $(X,B)$, and we may write
     $$K_Y+B_Y:=K_Y+\sum_{i\in \mathcal{I}}(1-a_i)B_i=f^{*}(K_X+B),$$
     where $a_i:=a(B_i,X,B)$ for any $i\in \mathcal{I}$. Let $\mu:=\pi\circ f$, where $\pi$ is the morphsim $X\to Z$. Shrinking $Z$ near $z$, we may assume that $a_i\ge0$ for any $i\in \mathcal{I}$. 
    
     We claim that the set
     $$\mathcal{S}:=\{y\mid \mld(Y\ni y,B_Y)=\mld(X/Z\ni z,B),y\in\mu^{-1}(z)\}$$
     is finite. Assume that the claim holds. Then $\mathcal{S}':=\{f(\bar{y})\mid y\in\mathcal{S}\}$ is a finite set. We will show that $\mathcal{S}'$ is the set of mld centers of $(X/Z\ni z,B)$. Pick any mld place $E\in e(z)$, let $y'$ be the generic point of $\Center_{Y}(E).$ Then $y'\in\mu^{-1}(z)$ and $\mld(X/Z\ni z,B)=a(E,X,B)=a(E,Y,B_Y)\ge\mld(Y\ni y',B_Y)\ge\mld(X/Z\ni z,B).$ Hence $\mld(X/Z\ni z,B)=\mld(Y\ni y',B_Y)$ and $y'\in \mathcal{S}$ . Thus $\Center_X(E)\in\mathcal{S}'$.
     
     It suffices to show the claim. Let $\mathcal{I}(y):=\{i\in \mathcal{I}\mid y\in B_i\}$ and $|\mathcal{I}(y)|$ the cardinality of $\mathcal{I}(y)$. It is enough to show that $|\mathcal{I}(y)|={\rm codim}\ y$ for any $y\in\mathcal{S}$, since then $y$ would be the generic point of a connected component of $\cap_{i\in \mathcal{J}}B_i$ for some non-empty subset $\mathcal{J}\subseteq \mathcal{I}$, and the total number of subsets of $\mathcal{I}$ is finite.
    
     Suppose that $\dim Z=\dim z$ and $\mld(X/Z\ni z,B)<1$. By Lemma \ref{lem:logsmoothmld},
     \begin{align*}
      \mld(X/Z\ni z,B)&=\mld(Y\ni y,B_Y)={\rm codim}\ y-\sum_{i\in \mathcal{I}(y)}(1-a_i)\\&= {\rm codim}\ y-|\mathcal{I}(y)|+\sum_{i\in \mathcal{I}(y)} a_i<1,
     \end{align*}
    which implies that
    ${\rm codim}\ y=|\mathcal{I}(y)|$ and we are done.
    
    If $\dim Z>\dim z$, then possibly replacing $Y$ by a higher model and shrinking $Z$ near $z,$ we may assume that $\mu^{-1}(\bar{z})=\cup_{i\in \mathcal{I}_z} B_i$ for some non-empty subset $\mathcal{I}_z\subseteq \mathcal{I}$ and $\mu(B_i)=\bar{z}$ for any $i\in \mathcal{I}_z$. Hence $\mathcal{I}_z\cap \mathcal{I}(y)\neq \emptyset$, and by Lemma \ref{lem:logsmoothmld},
    \begin{align*}
      &\mld(X/Z\ni z,B)=\mld(Y\ni y,B_Y)={\rm codim}\ y-\sum_{i\in \mathcal{I}(y)}(1-a_i)\\&= {\rm codim}\ y-|\mathcal{I}(y)|+\sum_{i\in \mathcal{I}(y)} a_i\ge \sum_{i\in \mathcal{I}(y)\cap \mathcal{I}_z}a_{i}\ge  \mld(X/Z\ni z,B),
    \end{align*}
    which implies that ${\rm codim }\ y=|\mathcal{I}(y)|$ and we are done.
\end{proof}

\begin{rem}
If $\dim Z=0$, $X=\mathbb{P}^1$ and $B=0$, then any closed point of $X$ is an mld center of $(X,B)$, thus there are infinitely many mld centers.  
\end{rem}

\subsection{Arithmetic of sets}

\begin{defn}
     Let $\Ii\subseteq[0,1]$ be a set of real numbers, we define $|\Ii|$ to be the cardinality of $\Ii$, $\lim\Ii$ to be the set of accumulation points of $\Ii$, and
     $$\sum \Ii:=\{0\}\cup\{\sum_{p=1}^l i_p\mid i_p\in\Ii \text{ for }1\le p\le l,l\in\Zz_{>0}\},$$ 
     $$\Ii_+:=\{0\}\cup\{\sum_{p=1}^l i_p\le1\mid i_p\in\Ii \text{ for }1\le p\le l,l\in\Zz_{>0}\},$$
     $$D(\Ii):=\{a\le1\mid a=\frac{m-1+f}{m},m\in\Zz_{>0},f\in \Ii_+\}.$$   
\end{defn}    

\begin{defn}[DCC and ACC sets]\label{def: DCC and ACC}        
    We say that $\Ii\subseteq\Rr$ satisfies the \emph{descending chain condition} $($DCC$)$ if any decreasing sequence $a_{1} \ge a_{2} \ge \cdots$ in $\Ii$ stabilizes. We say that $\Ii$ satisfies the \emph{ascending chain condition} $($ACC$)$ if any increasing sequence in $\Ii$ stabilizes.
\end{defn}

We need the following well-known lemma.
\begin{lem}[{\cite[Proposition 3.4.1]{HMX14}}]\label{lemma DI}
     Let $\Ii\subseteq[0,1],$ we have $D(D(\Ii))=D(\Ii)\cup\{1\}.$
\end{lem}

\subsection{Complements}

\begin{defn} \label{defn complement}
     We say that $(X/Z\ni z,B^+)$ is an \emph{$\Rr$-complement} of $(X/Z\ni z,B)$ if we have $B^+\ge B,$ $(X/Z\ni z,B^+)$ is lc over $z$, and $K_X + B^+ \sim_{\Rr} 0$ over some neighborhood of $z.$ In addition, if $(X/Z\ni z,B^+)$ is $\epsilon$-lc over $z,$ we say that $(X/Z\ni z,B^+)$ is an \emph{$(\epsilon,\Rr)$-complement} of $(X/Z\ni z,B).$ 
     
     Let $n$ be a positive integer. An \emph{$n$-complement} of $(X/Z\ni z,B)$ is a pair $(X/Z\ni z,B^+),$ such that over some neighborhood of $z,$ we have 
     \begin{enumerate} 
        \item $(X,B^+)$ is lc,          
        \item $n(K_X+B^+)\sim 0,$ and          
        \item $nB^+\ge n\lfloor B\rfloor+\lfloor(n+1)\{B\}\rfloor.$         
        \end{enumerate} 
        We say that $(X/Z\ni z, B^{+})$ is a \emph{monotonic $n$-complement} of $(X/Z\ni z, B)$ if we additionally have $B^{+}\ge B$. We say that $(X/Z\ni z,B^+)$ is an \emph{$(\epsilon,n)$-complement} of $(X/Z\ni z,B)$ if $(X/Z\ni z,B^+)$ is $\epsilon$-lc over $z$. 
        
        We say that $(X/Z\ni z,B)$ is {\emph{$(\epsilon,n)$-complementary}} $($respectively \emph{$(\epsilon,\Rr)$-complementary}$)$ if there exists an $(\epsilon,n)$-complement $($respectively \emph{$(\epsilon,\Rr)$-complement}$)$ $(X/Z\ni z,B^+)$ of $(X/Z\ni z,B).$ 
\end{defn} 

\begin{lem}\label{lem semiamplecomplement}
      Let $\epsilon$ be a non-negative real number. Suppose that $(X/Z\ni z,B)$ is $\epsilon$-lc over $z$, $-(K_X+B)$ is semiample over $Z$, and either
      \begin{enumerate}
          \item $\dim Z>\dim z$, or 
          \item $\dim Z=\dim z$, and $\epsilon<1$.
      \end{enumerate}
    Then there exists an $\Rr$-Cartier divisor $G\ge0$ such that $(X/Z\ni z,B+G)$ is an $(\epsilon,\Rr)$-complement of $(X/Z\ni z,B)$ and $\lfloor B+G\rfloor=\lfloor B\rfloor$.
\end{lem}
\begin{proof}
     The result follows from Lemma \ref{lemma mldcenterfinite}.
\end{proof}

\begin{comment}
\begin{proof}
    We may assume that $\dim z<\dim X$. 
    %We first consider the case when $\dim Z=0$ and $\epsilon\ge1$. In this case, we have $\epsilon=1$, $B=B'=0$, $K_X\equiv0$ and hence $(X,B')$ is $(\epsilon,\Rr)$-complementary. Now we may assume either $\dim Z>0$ or $\epsilon<1.$
    Since $-(K_X+B)$ is semiample over $Z$, by Lemma \ref{lemma mldcenterfinite}, there exists $0\le G\sim_{\Rr,Z}-(K_X+B)$ such that $\Supp G$ does not contain any mld center of $(X/Z\ni z,B)$, $\lfloor B+G\rfloor=\lfloor B\rfloor$ and $(X/Z\ni z,B+G)$ is an $(\epsilon,\Rr)$-complement of $(X/Z\ni z,B)$.
\end{proof}

   Let $\epsilon$ and $\epsilon'$ be two non-negative real numbers such that $\epsilon'<1$ ,$\epsilon\ge1$ and $\epsilon'\ge1.$

      Let $(X/Z\ni z,B)$ and $(X/Z\ni z,B')$ be two pairs such that
      \begin{enumerate}
         \item $(X/Z\ni z,B)$ is $(\epsilon,\Rr)$-complementary,
         %\item $\Supp B=\Supp B',$
         \item $(X/Z\ni z,B')$ is $\epsilon'$-lc over $z,$ and
         \item $-(K_X+B')$ is semiample over Z.
      \end{enumerate}
      
      Let $\epsilon'$ be a non-negative real numbers, and $(X/Z\ni z,B')$ a pair satisfying the following.
      \begin{enumerate}
         \item $(X/Z\ni z,B')$ is $\epsilon'$-lc over $z,$
         \item $-(K_X+B')$ is semiample over $Z$, and
         \item either $\dim Z>0$ or $\dim Z=0$ 
      \end{enumerate}

\end{comment}

The following lemma is well-known to experts. We will use the lemma frequently without citing it in this paper. 
\begin{lem}\label{lemma pullbaclcomplement}
     Let $\epsilon$ be a non-negative real number and $(X/Z\ni z,B)$ a pair. Assume that $f:X\dashrightarrow X'$ is a birational contraction and $B'$ is the strict transform of $B$ on $X'.$
     \begin{enumerate}
        \item If $(X/Z\ni z,B)$ is $(\epsilon,\Rr)$-complementary, then $(X'/Z\ni z,B')$ is $(\epsilon,\Rr)$-complementary.
        \item If $f$ is $-(K_X+B)$-non-positive and $(X'/Z\ni z,B')$ is $(\epsilon,\Rr)$-complementary, then $(X/Z\ni z,B)$ is $(\epsilon,\Rr)$-complementary.
     \end{enumerate}
\end{lem}

\begin{proof}
     Let $p:W\to X$ and $q:W\to X'$ be a common resolution of $f.$
     \medskip
     
     \noindent $(1)$ Let $(X/Z\ni z,B+G)$ be an $(\epsilon,\Rr)$-complement of $(X/Z\ni z,B)$ and $G'$ the strict transform of $G$ on $X'.$ We claim that $(X'/Z\ni z,B'+G')$ is an $(\epsilon,\Rr)$-complement of $(X'/Z\ni z,B').$ It suffices to show that $(X'/Z\ni z,B'+G')$ is $\epsilon$-lc over $z.$ By the negativity lemma, $p^*(K_X+B+G)=q^*(K_{X'}+B'+G')$ over a neighborhood of $z.$ Hence for any $E\in e(z),$ $a(E,X',B'+G')=a(E,X,B+G)\ge\epsilon$, that is, $(X'/Z\ni z,B'+G')$ is $\epsilon$-lc over $z.$
     \medskip
     
     \noindent $(2)$ Let $(X'/Z\ni z,B'+G')$ be an $(\epsilon,\Rr)$-complement of $(X'/Z\ni z,B')$. We have
     $$q^*(K_{X'}+B'+G')\ge q^*(K_{X'}+B')\ge p^*(K_X+B).$$
     Let $K_X+B+G:=p_*q^*(K_{X'}+B'+G').$ We claim that $(X/Z\ni z,B+G)$ is an $(\epsilon,\Rr)$-complement of $(X/Z\ni z,B).$ It suffices to show that $(X/Z\ni z,B+G)$ is $\epsilon$-lc over $z.$ By the negativity lemma, $p^*(K_{X}+B+G)=q^*(K_{X'}+B'+G')$ over a neighborhood of $z.$ Hence for any $E\in e(z),$ $a(E,X,B+G)=a(E,X',B'+G')\ge\epsilon$, that is, $(X/Z\ni z,B+G)$ is $\epsilon$-lc over $z.$
\end{proof}

\section{Uniform rational polytopes for linearity of MLDs}
The main goal of this section is to show Theorem \ref{thm:uniformpolytopeforsurfacemlds}.
\subsection{PLDs and MLDs in dimension two}
     
   We say that $(X\ni x,B)$ is an lc surface germ if $(X\ni x,B)$ is a surface germ which is lc at $x$.
     
\begin{defn}\label{defn pld}
     Let $(X\ni x, B:=\sum b_iB_i)$ be an lc surface germ, $f: Y \to X$ the minimal resolution of $X,$ and we may write
     $$K_{Y}+B_Y+\sum_{i\in \mathcal{I}} (1-a_i)E_i=f^{*}(K_{X}+B)$$
     where $B_Y$ is the strict transform of $B$, $E_i$ are $f$-exceptional divisors and $a_i:=a(E_i,X,B)$. The \emph{partial log discrepancy} $\pld(X\ni x,B)$ of $( X\ni x,B)$ is defined as follows. If $x$ is a singular point, we define
     $$\pld(X\ni x,B):=\min_{i\in \mathcal{I}}\{a_i\}.$$
     If $x$ is a smooth point, we set $\pld(X\ni x,B)=+\infty.$ For convenience, we may omit $x.$
     
     We define the $\emph{dual graph of the minimal resolution of}$ $X$ as the extended dual graph of $\cup E_i$ by adding $B_{Y,i}$ \cite[Chapter 4]{KM98}, where $B_{Y,i}$ are the strict transforms of $B_i$.
     
     In this paper, when we talk about the partial log discrepancy of a surface germ $(X\ni x,B)$, we always assume that $(X\ni x,B)$ is lc.

\end{defn}
   
     The following lemma is well-known to experts. For the reader's convenience, we include the proof here.

\begin{lem} \label{lemma concave}
     Let $(X,B)$ be an lc surface pair, $f:Y\to X$ a proper birational morphism from a smooth surface with exceptional divisors $E_i$ $(i\in\mathcal{I})$ and $B_Y$ the strict transform of $B$. Assume that the dual graph $\mathcal{DG}$ of $f$ is a tree and $0\le a_i:=a(E_i,X,B)\le 1$ for any $i\in\mathcal{I}$. Let $v_i$ be the vertex corresponds to the exceptional curve $E_i$, and $w_i:=-(E_i\cdot E_i)$ for any $i\in\mathcal{I}$.
     \begin{enumerate}
         \item If $a_k\ge \epsilon$ for some positive real number $\epsilon$ and positive integer $k$, then $w_k\le \lfloor\frac{2}{\epsilon}\rfloor$.
          
         \item If $w_k\ge2$ for some positive integer $k$, then $1-a_k\ge\frac{1-a_{k'}}{2}$ for any vertex $v_{k'}$ which is adjacent to $v_{k}$, and $2a_k\le a_{k_1}+a_{k_2}$ for any two vertices $v_{k_1}$ and $v_{k_2}\ (k_1<k_2)$ which are adjacent to $v_{k}.$ 
         
         \item If $v_{k_0}$ is a fork, $v_{k_1},v_{k_2}$ and $v_{k_3}$ are adjacent to $v_{k_0}$, and $w_{k_i}\ge2$ for any $0\le i\le 2,$ then $a_{k_0}\le a_{k_3}.$ If the equality holds, then $w_{k_0}=w_{k_1}=w_{k_2}=2$ and $B_Y\cdot E_{k_0}=B_Y\cdot E_{k_1}=B_Y\cdot E_{k_2}=0.$
 
         \item If $w_k=1,$ $w_i\ge2$ for any $i\neq k$, $a_k<\min_{i\neq k}\{a_i\}$ and $v_k$ is not a fork, then $\mathcal{DG}$ is a chain.
         
         \item If $v_0,v_1,\dots,v_n$ are vertices such that $v_i$ is adjacent to $v_{i+1}$ for any $0\le i\le n-1,$ $w_i\ge2$ for any $2\le i\le n-1$, $w_1\ge3,a_1\ge a_0$ and $a_1\ge\epsilon$ for some $\epsilon>0,$ then $n\le \lfloor\frac{1}{\epsilon}\rfloor.$
     \end{enumerate}    
     
\end{lem}

\begin{proof}    
 
     Since $0\le a_i\le 1$ for any $i\in\mathcal{I},$ we have
     \begin{align}\label{4.1}
        0 &=(K_Y+B_Y+\sum_{i\in \mathcal{I}} (1-a_i)E_i)\cdot E_k
        \\&\ge-2+w_ka_{k}+\sum_{E_i\cdot E_k\ge 1}(1-a_i)+B_Y\cdot E_k.\nonumber%\ \ \ \ \ \ \ \ (*)
     \end{align}
     
\noindent     (1) By $(\ref{4.1})$, $0\ge -2+w_ka_k$, and $w_k\le \frac{2}{a_k}\le\frac{2}{\epsilon}$.
     
 \noindent    (2) By $(\ref{4.1})$, $0\ge -2+2a_k+1-a_{k'}$ and $0\ge-2+2a_k+1-a_{k_1}+1-a_{k_2}$. Thus $1-a_k\ge\frac{1-a_{k'}}{2}$, and $2a_k\le a_{k_1}+a_{k_2}$. 
     
   \noindent  (3) By (2), we have $1-a_{k_i}\ge\frac{1-a_{k_0}}{2}$ for $i=1,2$, and if the equalities hold, then $B_Y\cdot E_{k_1}=B_Y\cdot E_{k_2}=0$. In particular, $1-a_{k_1}+1-a_{k_2}\ge 1-a_{k_0}$, and 
   $a_{k_1}+a_{k_2}\le a_{k_0}+1$. Moreover, by $(\ref{4.1})$,
   \begin{align*}
   a_{k_3}&\ge 1+2a_{k_0}-a_{k_1}-a_{k_2}+(w_k-2)a_k+B_Y\cdot E_{k_0}\\
   &\ge 1+2a_{k_0}-a_{k_1}-a_{k_2}\ge a_{k_0},
   \end{align*}
   and if $a_{k_3}=a_{k_0}$, then $w_{k_0}=w_{k_1}=w_{k_2}=2$, and $B_Y\cdot E_{k_0}=B_Y\cdot E_{k_1}=B_Y\cdot E_{k_2}=0.$
     
  \noindent (4) Suppose that the statement does not hold. We may find a sequence of vertices $v_{m_0},\dots,v_{m_l}:=v_k$ such that $v_{m_0}$ is a fork and $v_{m_i}$ is adjacent to $v_{m_{i+1}}$ for any $0\le i\le l-1$. By $(3)$, we have $a_{m_0}\le a_{m_1}$. By (2), $a_{m_0}\le a_{m_1}\le\cdots\le a_{m_l}=a_k<\min_{i\neq k}\{a_i\}$, a contradiction.
  
  \noindent (5) We may assume that $n\ge2$. By $(\ref{4.1}),$ $0\ge-2+w_1a_1+1-a_{0}+1-a_{2}$. Thus $a_{2}-a_1\ge (a_1-a_{0})+(w_1-2)a_1\ge\epsilon.$ By $(2),$ $a_{i+1}-a_i\ge\epsilon$ for any $1\le i\le n-1$ and hence $a_{i}\ge i\epsilon$ for any $1\le i\le n.$ Thus $1\ge a_n\ge n\epsilon$ and $n\le \lfloor\frac{1}{\epsilon}\rfloor.$
\end{proof}

\begin{lem}\label{lemmma extractmld}
     Suppose that $(X\ni x,B)$ is an lc surface germ and $X$ is smooth. Then there exists a sequence of blow-ups:
     $$X_n\to X_{n-1}\to\cdots\to X_1\to X_0:=X$$
     with the following properties.
     \begin{enumerate}
         \item $f_i:X_{i}\to X_{i-1}$ is the blow-up of $X_{i-1}$ at a point $x_{i-1}\in X_{i-1}$ with the exceptional divisor $E_i$ for any $1\le i\le n$, where $x_0:=x$.
         \item $x_i\in E_i$ for any $1\le i\le n-1$.
         \item $a(E_i,X,B)>a(E_n,X,B)=\mld(X\ni x,B)$ for any $1\le i\le n-1$.
         \item If $a(E_1,X,B)\ge 1$, then $\mld(X\ni x,B)=a(E_1,X,B).$
         \item If $\mld(X\ni x,B)<1,$ then $0\le a(E_i,X,B)<1$ for any $1\le i\le n.$
     \end{enumerate}
\end{lem}

\begin{proof}
     The existence of such a sequence of blow-ups with properties $(1)$--$(3)$ follows from \cite[Lemma 2.45]{KM98}. It suffices to show $(4)$ and $(5)$. For convenience, we denote the strict transform of $E_i$ on $X_j$ by $E_i$ for any $n\ge j\ge i.$ Let $B_{X_i}$ be the strict transform of $B$ on $X_i,a_i:=a(E_i,X,B)$ and $e_i:=1-a_i$ for any $1\le i\le n$. Since $\sum E_i$ is snc, there exists at most one $i'\neq i,$ such that 
     $x_i\in E_{i'}$ on $X_i$. If $x_{i}\in E_{i'}$ for some $i'\neq i$, then we set $c_{i}=e_{i'}$, otherwise we let $c_{i}=0$. Let $e_0=c_0=0$, it is obvious that $c_{i+1}\in\{e_{i},c_i,0\}$ for $i\ge 1$.
     
     If $a_1\ge1$ then $e_1=-1+\mult_{x}B=1-a_1\le0.$
     We will show by induction on $k$ that $e_k\le e_1$ for any $k\ge1.$ Assume that for some $n-1\ge k_0\ge 1$, $e_j\le e_1$ for any $1\le j\le {k_0}$. Then by induction, we have $c_{k_0}\le \max\{e_{k_0-1},c_{k_0-1},0\}\le 0$, and
     $$e_{k_{0}+1}=-1+\mult_{x_{k_0}}B_{X_{k_0}}+e_{k_0}+c_{k_0}\le -1+\mult_{x}B+e_{k_0}\le 
     e_1.$$
     We finish the induction. Hence $e_n\le e_1$ and $a_n\ge a_1$. Thus $a_n=a_1$, $n=1$, and $\mld(X\ni x,B)=a(E_1,X,B).$ This proves $(4).$
     
     We now show $(5)$. If $\mld(X\ni x,B)<1,$ then $a_1<1$ and $e_1>0$ by $(4).$ It suffices to show that $e_i>0$ for any $n\ge i\ge2$.
     Suppose that there exists $i\ge 2$, such that $e_i\le0$ and $e_j>0$ for any $1\le j<i$. Then 
     $$e_{i}=-1+\mult_{x_{i-1}}B_{X_{i-1}}+e_{i-1}+c_{i-1}\le0,$$
     $c_{i-1}\ge \min\{e_{i-2},c_{i-2},0\}\ge0$ and $e_{i-1}>0.$ We will show by induction on $k$ that $e_k\le0$ and $c_{k-1}\le \max\{c_{i-1},e_{i-1}\}$ for any $n\ge k\ge i$. Assume that for some $n-1\ge k_0\ge i$, $e_{j}\le0$ and $c_{j-1}\le \max\{c_{i-1},e_{i-1}\}$ for any $i\le j\le k_0$. Then $c_{k_0}\le \max\{e_{k_0-1},c_{k_0-1},0\}\le \max\{c_{i-1},e_{i-1}\}$, and
     \begin{align*}
     e_{k_0+1}&=-1+\mult_{x_{k_0}}B_{X_{k_0}}+e_{k_0}+c_{k_0}\\
       &\le -1+\mult_{x_{i-1}}B_{X_{i-1}}+\max\{c_{i-1},e_{i-1}\}\le 0.
     \end{align*}
     We finish the induction, and $e_n=1-\mld(X\ni x,B)\le 0$, a contradiction.
\end{proof}

\noindent{\bf Notation} $(\star).$ Fix a positive real number $\epsilon$ and a DCC set $1\in\Ii\subseteq[0,1]$. Let $\epsilon_0:=\inf_{b\in\Ii}\{b>0\}>0$. We define
     $$\mathcal{P}ld(\Ii):=\{\pld(X\ni x,B)\mid (X\ni x,B)\text{ is an lc surface germ},B\in \Ii\},$$
     $$\mathcal{I}_0:=\{\frac{l_1}{l_2}\mid l_1\in\Zz_{\ge0},l_2\in\Zz_{>0},l_1,l_2\le \lfloor\frac{2}{\min\{\epsilon,\epsilon_0\}}\rfloor^{\lfloor\frac{2}{\epsilon}\rfloor}\},$$
     $$\Ll(\epsilon,\Ii):=\{\frac{1-\sum l_ib_i}{l}\ge\epsilon\mid l\in\Zz_{>0},l_i\in\Zz_{\ge0},b_i\in\Ii\},$$
     \begin{align*}
     \mathcal{P}(\epsilon,\Ii):=\{\frac{(A+p_1)\beta_1+p_2\beta_2}{A+p_1+p_2}\ge\epsilon&\mid A\in\Zz_{>0},p_1,p_2\in \mathcal{I}_0,\\&\beta_1,\beta_2\in\Ll(\epsilon,\Ii),\beta_2\ge \beta_1\}.
     \end{align*}
     For any positive integer $N,$ we define
     $$\Ff(N,\epsilon,\Ii):=\{\frac{l_0-\sum l_ib_i}{l}\ge\epsilon\mid l\in\Zz_{>0},l_i\in\Zz_{\ge0},l_0\le N,b_i\in\Ii\}.$$
     
     Note that by Lemma \ref{lemma 3.3}, for any lc surface germ $(X\ni x,B)$ such that $\pld(X,B)\ge\epsilon,B\in\Ii,$ and the dual graph $\mathcal{DG}\{\mathcal{DG}_{ch}^{(m_1,q_1),u_1},\mathcal{DG}_{ch}^{2^A},\mathcal{DG}_{ch}^{(m_2,q_2),u_2}\}$ $\in\mathfrak{C}_{\epsilon,\Ii}\cup\mathfrak{T}_{\epsilon,\Ii},$ then $\pld(X,B)\in\mathcal{P}(\epsilon,\Ii),$ since
     $\frac{m_i}{m_i-q_i},\frac{q_i}{m_i-q_i}\in\mathcal{I}_0,$ $\frac{\alpha_i}{m_i-q_i}\in\Ll(\epsilon,\Ii)$ $(i=1,2)$. Moreover, we have the following result.
     
\begin{lem} \label{lemma pld set}
     With notation $(\star)$. Then there is a positive integer $N_0$ depending only on $\epsilon$ and $\Ii$ such that
     $$\mathcal{P}ld(\Ii)\cap[\epsilon,+\infty)\subseteq\Ff(N_0,\epsilon,\Ii)\cup\mathcal{P}(\epsilon,\Ii).$$
     Moreover, if $\Ii$ is a finite set, then  $\lim\mathcal{P}ld(\Ii)\subseteq\{\frac{1-b^+}{l}\mid b^+\in\Ii_+,l\in\Zz_{>0}\}$.
     
\end{lem}

\begin{proof}
    If $\epsilon>1,$ then $\mathcal{P}ld(\Ii)\cap[\epsilon,+\infty)=\emptyset.$ We may let $N_0=1.$
    
    Suppose that $\epsilon\le1$. Let $N_1$ be the number defined in Lemma \ref{lemma 3.3} depending only on $\epsilon$ and $\Ii,$ that is, if the dual graph $\mathcal{DG}$ of the minimal resolution of $X$ has at least $N_1$ vertices then $\mathcal{DG}\in\mathfrak{C}_{\epsilon,\Ii}\cup\mathfrak{T}_{\epsilon,\Ii}$. Let $N_0$ be the number defined in Lemma \ref{lemma minimalreslwithbddvertices} depending only on $N_1$ and $\epsilon$. We will show that $N_0$ has the required properties. 
    
    Assume that $(X\ni x,B)$ is an lc surface germ such that $B\in\Ii$ and $+\infty>\pld(X\ni x,B)\ge\epsilon.$ Let $\mathcal{DG}$ be the dual graph of the minimal resolution of $X.$ If $\mathcal{DG}$ has at least $ N_1$ vertices, then $\pld(X\ni x,B)\in\mathcal{P}(\epsilon,\Ii)$ by Lemma \ref{lemma 3.3}. If $\mathcal{DG}$ has at most $N_1-1$ vertices, then $\pld(X\ni x,B)\in\Ff(N_0,\epsilon,\Ii)$ by Lemma \ref{lemma minimalreslwithbddvertices}. Hence $N_0$ has the required properties. 
    
    If $\Ii$ is a finite set, then
    \begin{align*}
   &(\lim\mathcal{P}ld(\Ii))\cap [\epsilon,+\infty)\subseteq\lim\Ff(N_0,\epsilon,\Ii)\cup\lim\mathcal{P}(\epsilon,\Ii)\\
   \subseteq &\lim \mathcal{P}(\epsilon,\Ii)\subseteq \{\frac{1-b^+}{l}\mid b^+\in\Ii_+,l\in\Zz_{>0}\},
    \end{align*}
    and 
     $$\lim\mathcal{P}ld(\Ii)\subseteq \{\frac{1-b^+}{l}\mid b^+\in\Ii_+,l\in\Zz_{>0}\}$$
      as $\Ff(N_0,\epsilon,\Ii)$ is a finite set and $0\in \{\frac{1-b^+}{l}\mid b^+\in\Ii_+,l\in\Zz_{>0}\}$.
\end{proof}
     The following result is stated in the proof of \cite[Theorem 3.8]{Ale93} without a proof. For the reader's convenience, we give a proof here.
     
\begin{lem} \label{lemma pld=mld}     
     Let $\epsilon$ be a positive real number. Then there exists a positive integer $N$ depending only on $\epsilon$ satisfying the following. 
     
     Assume that $(X\ni x,B)$ is an lc surface germ such that 
     \begin{enumerate}
       \item  $(X\ni x,B)$ is $\epsilon$-lc at $x$, and 
       \item the dual graph of the minimal resolution of $X$ has at least $N+1$ vertices.
     \end{enumerate}
     Then $\mld(X\ni x,B)=\pld(X\ni x,B)$.    
\end{lem}

\begin{proof}
     Let $N_0:=\lfloor\frac{1}{\epsilon}\rfloor$, we will show that $N:=2N_0$ has the required properties.
       
     Suppose on the contrary that there exists an lc surface germ $(X\ni x,B)$ satisfying the conditions and $\mld(X\ni x,B)<\pld(X\ni x,B)\le1.$ Let $f:Y\to X$ be the minimal resolution of $X$ with exceptional divisors $E_1,\dots,E_{n_0},n_0\ge N+1$. By Lemma \ref{lemmma extractmld}, there exist a positive integer $n>n_0$ and a sequence of blow-ups $$W:=Y_{n-n_0}\to\cdots\to Y_0:=Y$$
     with the following properties.
     \begin{itemize}
         \item $h_i:Y_{i}\to Y_{i-1}$ is the blow-up of $Y_{i-1}$ at a closed point $y_{i-1}\in Y_{i-1}$ with the exceptional divisor $E_{n_0+i}$ for any $1\le i\le n-n_0.$
         \item $y_{i}\in E_{n_0+i},$ for any $1\le i\le n-n_0-1,$ and $y_0\in E_{m_0}$ for some $1\le m_0\le n_0.$
         \item There is only one $(-1)$-curve on $Y_i$ for any $1\le i\le n-n_0.$
         \item $a(E_{n_0+i},X,B)>a(E_n,X,B)=\mld(X\ni x,B)$ for any $1\le i\le n-n_0-1.$
         \item $a(E_{n_0+i},X,B)<1$ for any $1\le i\le n-n_0$.
     \end{itemize}
     For convenience, we may denote the strict transform of $E_{i}$ on $Y_j$ by $E_i$ for any $0\le j\le n-n_0$ and any $i\le j+n_0.$
     Let $\mathcal{DG}$ be the dual graph of $g=f\circ h_1\circ\cdots\circ h_{n-n_0}:W\to X$, $v_i$ the vertex corresponds to the exceptional curve $E_i$, $w_i:=-(E_i\cdot E_i)$ for $1\le i\le n$ and let $w_0:=3,w_{n+1}:=3$.
     Since $\sum_{i=1}^{n-1} E_i$ is snc, $v_{n}$ is not a fork.
     Thus by Lemma \ref{lemma concave}$(4)$, $\mathcal{DG}$ is a chain.
     
     Possibly reindexing $\{E_i\}_{i=1}^n$,
     we may assume that $v_i$ is adjacent to $v_{i+1}$ for any $1\le i\le n-1$, $\mld(X\ni x,B)=a(E_m,X,B)$ for some $1\le m\le n-1$ and $2= w_{m+1}\le w_{m-1}.$ Let $a_i:=a(E_i,X,B)\ge \epsilon$ for any $1\le i\le n.$ By Lemma \ref{lemma concave}$(2)$, $a_1>\cdots>a_m$ and $a_m<\cdots<a_n$ .
     
     We show that $m\le N_0+1$ and $w_{m-1}\ge3$. We may assume that $m\ge2$. If $w_{m-1}=2$, then we can contract $E_{m}$ and get $Y_{n-n_0-1}$. There are two $(-1)$-curves on $Y_{n-n_0-1}$ as $w_{m-1}=w_{m+1}=2$, a contradiction. Hence $w_{m-1}\ge3.$ By Lemma \ref{lemma concave}$(5)$, $m\le N_0+1.$
     
     %\han{check $l$, it is possible $w_{m+1}\ge 3$ or $w_n=2$.}
     Let $l:=\min\{k\mid w_k\ge3,k>m\}.$ By Lemma \ref{lemma concave}$(5)$ again, $n-l\le N_0-1.$ We contract $(-1)$-curves which correspond to vertices $v_{m},v_{m+1},\dots,v_{l-1}$ one by one, and get a smooth surface $W'$. In particular, $f':W'\to X$ factors through $f$ and there are at most $2N_0$ $f'$-exceptional divisors as $m\le N_0+1$ and $n-l\le N_0-1.$ Hence the dual graph of $f$ has at most $2N_0$ vertices, a contradiction.
\end{proof}

     Let $\Ii\subseteq[0,1]$ be a set of real numbers, we define
     $$\mathcal{M}ld(2,\Ii):=\{\mld(X\ni x,B)\mid (X\ni x,B)\text{ is an lc surface germ},B\in \Ii\}.$$

\begin{lem}\label{lemma mld set}
     Let $\epsilon$ be a positive real number and $\Ii\subseteq[0,1]$ a DCC set. Then there is a positive integer $N_0$ depending only on $\epsilon$ and $\Ii$ such that
     $$\mathcal{M}ld(2,\Ii)\cap[\epsilon,+\infty)\subseteq\Ff(N_0,\epsilon,\Ii)\cup\mathcal{P}(\epsilon,\Ii).$$
\end{lem}  

\begin{proof}
     By Lemma \ref{lemma 3.3} and Lemma \ref{lemma pld=mld}, there exists a positive integer $N_1$ depending only on $\epsilon$ and $\Gamma$ such that for any lc surface germ $(X\ni x,B)$ satisfying $B\in\Ii,\mld(X\ni x,B)\ge\epsilon$ and the dual graph of the minimal resolution of $X$ has at least $ N_1$ vertices, then
     $$\mld(X\ni x,B)=\pld(X\ni x,B)\in\mathcal{P}(\epsilon,\Ii).$$
     Let $N_2$ be a positive integer satisfying the properties in Lemma \ref{lemma minimalreslwithbddvertices} and Lemma \ref{lemma pld set} depending only on $\Ii,N_1$ and $\epsilon.$ 
     Let $N_3$ be the number defined in Theorem \ref{lemma dccsmoothmld} depending only on $\Ii_1$, where $\Ii_1:=\Ii\cup\{1-\alpha\ge0\mid \alpha\in\mathcal{F}(N_2,\epsilon,\Ii)\}.$ 
     
     We may find a positive integer $N_0\ge N_2$ depending only on $N_3,N_2,\epsilon$ and $\Ii_1$ such that if
     $\alpha=l_0-\sum l_ib_i\ge\epsilon$
     for some $b_i\in\Ii_1$ and non-negative integers $l_0\le N_3,l_i$, then $\alpha\in\mathcal{F}(N_0,\epsilon,\Ii).$ We will show that $N_0$ has the required properties. 
     
     Assume that $(X\ni x,B)$ is a surface germ such that $B\in\Ii$ and $\mld(X\ni x,B)\ge\epsilon$. Let $f:Y\to X$ be the minimal resolution of $X$, and $\mathcal{DG}$ the dual graph of $f.$ If $\mathcal{DG}$ has at least $ N_1$ vertices, then by our choice of $N_1$,
     $$\mld(X\ni x,B)=\pld(X\ni x,B)\in\mathcal{P}(\epsilon,\Ii).$$
     If $\mathcal{DG}$ has at most $ N_1-1$ vertices, then by Theorem \ref{lemma minimalreslwithbddvertices}, $B_Y\in\Ii_1$, where $K_Y+B_Y:=f^*(K_X+B)$. We may write $B_Y:=\sum b_{Y,i}B_{Y,i},$ where $B_{Y,i}$ are distinct prime divisors and $b_{Y,i}\in\Ii_1.$ By Theorem \ref{lemma dccsmoothmld}, 
     $$\mld(X\ni x,B)=\mld(Y\ni y,B_Y)=l_0-\sum l_ib_{Y,i}\ge\epsilon,$$
     for some $y\in Y$ and non-negative integers $l_0\le N_3,l_i$. Hence $\mld(X\ni x,B)\in\Ff(N_0,\epsilon,\Ii)$ by our choice of $N_0$.
\end{proof}

\subsection{Proof of Theorem \ref{thm:uniformpolytopeforsurfacemlds}}
In this subsection, we show Theorem \ref{thm:uniformpolytopeforsurfacemlds}.

    \textbf{Notation} $(\star\star).$
     Let $\epsilon$ be a positive real number, $c,m$ two positive integers and $\bm{r}_0=(r_1,\dots,r_c)\in\Rr^c$ a point such that $r_0=1,r_1,\dots,r_c$ are linearly independent over $\Qq$. Let $s_1,\ldots, s_m: \Rr^{c}\to\Rr$ be $\Qq$-linear functions. 
     
     By a pair $(X/Z\ni z,B(\bm{r}):=\sum_{i=1}^ms_i(\bm{r})B_i),$ we mean that $X$ is a variety, $B_1,\dots,B_m\ge0$ are Weil divisors on $X$ such that $(X/Z\ni z,B(\bm{r}_0))$ is a pair. By \cite[Lemma 5.3]{HLS19}, $K_X+B(\bm{r})$ is $\Rr$-Cartier for any $\bm{r}\in\Rr^c$.
     
     Let $\epsilon_0:=\frac{1}{2}\min_{1\le i\le m}\{s_{i}(\bm{r}_0)\}$ and 
     $\mathcal{C}:=\{s_i(\bm{r})\mid 1\le i\le m\},$
     $$\mathcal{I}'_0:=\{\frac{m_1}{m_2}\mid m_1\in\Zz_{\ge0},m_2\in\Zz_{>0},m_1,m_2\le \lfloor\frac{4}{\min\{\epsilon,\epsilon_0\}}\rfloor^{\lfloor\frac{4}{\epsilon}\rfloor}\},$$
     $$\Ll(\epsilon,\mathcal{C}):=\{\beta(\bm{r}):=\frac{1-\sum_{i=1}^m l_is_i(\bm{r})}{l}\mid l\in\Zz_{>0},l_i\in\Zz_{\ge0},1\le i\le m,\beta(\bm{r}_0)\ge\epsilon\},$$
     \begin{align*}
     \mathcal{P}(\epsilon,\mathcal{C}):=\{p_A(\bm{r}):=&\frac{(A+p_1)\beta_1(\bm{r})+p_2\beta_2(\bm{r})}{A+p_1+p_2}\mid A\in\Zz_{>0},\\&p_1,p_2\in \mathcal{I}'_0,\beta_1(\bm{r}),\beta_2(\bm{r})\in\Ll(\epsilon,\mathcal{C}),\beta_2(\bm{r}_0)\ge \beta_1(\bm{r}_0)\}.
     \end{align*}
      For any positive integer $N$, we define
      \begin{align*}
     \Ff(N,\epsilon,\mathcal{C}):=\{f(\bm{r}):=\frac{l_0-\sum_{i=1}^m l_is_i(\bm{r})}{l}\mid l&\in\Zz_{>0},l_i\in\Zz_{\ge0},1\le i\le m,\\&l_0\le N,f(\bm{r}_0)\ge\epsilon\}\cup(\Qq\cap\{\epsilon\}).
     \end{align*}
     For any set $\mathcal{S}$ of linear functions of $c$ variables, we define $\mathcal{S}|_{\bm{r}'}:=\{s(\bm{r}')\mid s(\bm{r})\in\mathcal{S}\}$ for any point $\bm{r}'\in\Rr^c$.
     
    It is clear that $\Ll(\epsilon,\mathcal{C})\subseteq\Ff(N,\epsilon,\mathcal{C})$ are finite sets and for any $\Qq$-linear functions $h_1(\bm{r}),h_2(\bm{r}):\Rr^c\to\Rr,$ if $h_1(\bm{r}_0)=h_2(\bm{r}_0),$ then $h_1(\bm{r})=h_2(\bm{r})$ for any $\bm{r}\in\Rr^c$ as $r_0,r_1,\ldots,r_c$ are linearly independent over $\Qq$. 
    
    Note that $\Ll(\epsilon,\mathcal{C}|_{\bm{r}_0})=\Ll(\epsilon,\mathcal{C})|_{\bm{r}_0},\Ff(N,\epsilon,\mathcal{C}|_{\bm{r}_0})=\Ff(N,\epsilon,\mathcal{C})|_{\bm{r}_0}$ and $\mathcal{P}(\epsilon,\mathcal{C}|_{\bm{r}_0})$ $\subseteq\mathcal{P}(\epsilon,\mathcal{C})|_{\bm{r}_0},$ since $\mathcal{I}_0\subseteq\mathcal{I}'_0$, where $\mathcal{I}_0,\Ll(\epsilon,\mathcal{C}|_{\bm{r}_0}),\Ff(N,\epsilon,\mathcal{C}|_{\bm{r}_0}),\mathcal{P}(\epsilon,\mathcal{C}|_{\bm{r}_0})$ are defined as in \textbf{Notation} $(\star).$

\begin{lem}\label{lemma bddcimldlinear}
     With notation $(\star\star)$ and assume that $\epsilon$ is non-negative. Let $d,I_0,M_0$ be positive integers. Then there exist a positive real number $\delta$ and a $\Qq$-linear function $f(\bm{r}):\Rr^c\to\Rr$ depending only on $\epsilon,d,c,m,I_0,M_0,\bm{r}_0,\mathcal{C}$ satisfying the following.
     
     \begin{enumerate}
      \item $f(\bm{r}_0)\ge\epsilon$, and if $\epsilon\in\Qq$, then $f(\bm{r})=\epsilon$ for any $\bm{r}\in\Rr^c$.
      \item Assume that $(X/Z\ni z,B(\bm{r}_0))$ is a pair of dimension $d$, such that
      \begin{itemize} 
        \item $(X/Z\ni z,B(\bm{r}_0))$ is $\epsilon$-lc over $z,$
        \item $\mld(X/Z\ni z,C)\le M_0$ for any boundary $C,$ and
        \item $I_0D$ is Cartier for any $\Qq$-Cartier Weil divisor $D$ on $X$.
      \end{itemize}
      Then there exists a prime divisor $E\in e(z)$ such that
      $$\mld(X/Z\ni z,B(\bm{r}))=a(E,X,B(\bm{r}))\ge f(\bm{r})$$
      for any $\bm{r}\in\Rr^c$ satisfying $||\bm{r}-\bm{r}_0||_\infty\le\delta$. 
     \end{enumerate}
\end{lem}

\begin{proof}
     By \cite[Corollary 5.5]{HLS19}, there exists a positive real number $\delta_0$ depending only on $d,c,m,\bm{r}_0,\mathcal{C}$ such that $(X/Z\ni z,B(\bm{r}))$ is lc over $z$ for any $\bm{r}\in\Rr^c$ satisfying $||\bm{r}-\bm{r}_0||_{\infty}\le\delta_0.$ In particular, there exist a finite set $\Gamma_1=\{a_i\}_{i=1}^{k}$ of positive real numbers and a finite set $\Gamma_2$ of non-negative rational numbers depending only on $c,m,\bm{r}_0, \mathcal{C}$ and $\delta_0$ such that
     $$K_X+B(\bm{r}_0)=\sum_{i=1}^{k} a_i(K_X+B^i),$$ 
     and $(X/Z\ni z,B^i)$ is lc over $z$ for some $B^i\in\Gamma_2$.
     There exist a positive integer $c'\ge c$, real numbers $r_{c+1},\dots,r_{c'}$ and $\Qq$-linear functions $a_i(x_1,\dots,x_{c'}):\Rr^{c'}\to \Rr$ such that $r_0,r_1,\dots,r_{c'}$ are linearly independent over $\Qq,$ and $a_i(r_1,\dots,r_{c'})=a_i$ for any $1\le i\le k.$ Since $\sum_{i=1}^{k} a_i=1,\sum_{i=1}^{k} a_i(x_1,\dots,x_{c'})=1$ for any $(x_1,\dots,x_{c'})\in\Rr^{c'}.$ In particular,  we have $\sum_{i=1}^{k} a_i(\bm{r})=1,$ and $B(\bm{r})=\sum_{i=1}^{k} a_i(\bm{r})B^i$ for any $\bm{r}\in\Rr^c,$ where $a_i(\bm{r})=a_i(\bm{r},r_{c+1},\dots,r_{c'})$ for any $1\le i\le k$.
     
     Since $\Gamma_2$ is a finite set of rational numbers, we may find a positive integer $I_1$ such that $I_1\Ii_2\subseteq\Zz$. Let $I:=I_0I_1.$ By assumption, possibly shrinking $Z$ near $z$, we may assume that $I(K_X+B^i)$ is Cartier for any $1\le i\le k.$  
     
     Let $a_0:=\frac{1}{2}\min_{1\le i\le k}\{a_i\},M:=\frac{M_0}{a_0},$ and
     $$\mathcal{F}:=\{a(\bm{r}):=\sum_{i=1}^{k} a_i(\bm{r})b_i\mid a(\bm{r}_0)\ge\epsilon, b_i\in \frac{1}{I}\Zz_{\ge0},b_i\le M,1\le i\le k\}\cup(\Qq\cap\{\epsilon\})$$
     which is a finite set of $\Qq$-linear functions. 
     Let $f(\bm{r})\in\mathcal{F}$, such that $f(\bm{r}_0)=\min\mathcal{F}|_{\bm{r}_0}$. Note that if $\epsilon\in\Qq$, then $f(\bm{r}_0)=\epsilon$ and hence $f(\bm{r})=\epsilon$ for any $\bm{r}\in\Rr^c$ as $r_0,\dots,r_c$ are linearly independent over $\Qq$.
     %where $\mathcal{F}|_{\bm{r}_0}:=\{\mu(\bm{r}_0)\mid \mu(\bm{r})\in\mathcal{F}\}$.
     
     We may find a positive real number $\delta\le\delta_0$  depending only on $\mathcal{F},\delta_0,\bm{r}_0$ such that $a_i({\bm{r}})\ge a_0$ and $g(\bm{r})>h(\bm{r})$ for any $||\bm{r}-\bm{r}_0||_\infty\le\delta$, $1\le i\le k$, and $g(\bm{r}),h(\bm{r})\in\mathcal{F}$ satisfying $g(\bm{r}_0)>h(\bm{r}_0)$.
     
     We claim that $\delta$ and $f(\bm{r})$ have the required properties. We first show that 
     $\mld(X/Z\ni z,B(\bm{r}'))\in\mathcal{F}|_{\bm{r}'}$ for any $\bm{r}'\in\Rr^c$ satisfying $||\bm{r}'-\bm{r}_0||_\infty\le\delta.$ We may assume that $\mld(X/Z\ni z,B(\bm{r}'))=a(E',X,B(\bm{r}'))$ for some $E'\in e(z)$ and $a(E',X,B^1)=\max_{1\le i\le k}\{a(E',X,B^i)\}$. By assumption, we have 
     \begin{align*}
       M_0&\ge\mld(X/Z\ni z,B(\bm{r}'))=\sum_{i=1}^{k}a_i(\bm{r}')a(E',X,B^i)\\&\ge a_1(\bm{r}')a(E',X,B^1)\ge a_0a(E',X,B^1).
     \end{align*}
     Thus $a(E',X,B^1)\le M$ and $a(E',X,B(\bm{r}'))\in\mathcal{F}|_{\bm{r}'}.$  
     Let $E\in e(z)$ such that $\mld(X/Z\ni z,B(\bm{r}_0))=a(E,X,B(\bm{r}_0))$, then $a(E,X,B(\bm{r}))\in\Ff$. We show that $E$ has the required properties. By our choice of $\delta$ and $f(\bm{r})$, $a(E,X,B(\bm{r}))\ge f(\bm{r})$ for any $\bm{r}\in\Rr^c$ satisfying $||\bm{r}-\bm{r}_0||_\infty\le\delta$. It suffices to show that 
     $$\mld(X/Z\ni z,B(\bm{r}))=a(E,X,B(\bm{r}))$$ 
     for any $\bm{r}\in\Rr^c$ satisfying $||\bm{r}-\bm{r}_0||_\infty\le\delta$. Otherwise there exist a divisor $E'\in e(z)$ and a point $\bm{r}'\in\Rr^c$ such that  $||\bm{r}'-\bm{r}_0||_\infty\le\delta,\mld(X/Z\ni z,B(\bm{r}'))=a(E',X,B(\bm{r}'))<a(E,X,B(\bm{r}'))$. Since $a(E,X,B(\bm{r})),a(E',X,B(\bm{r}))\in\mathcal{F}$ and $a(E',X,B(\bm{r}_0))>a(E,X,B(\bm{r}_0)),a(E',X,B(\bm{r}'))>a(E,X,B(\bm{r}'))$, a contradiction.
\end{proof}
     
\begin{thm}\label{thm locally linearity of mld}
     With notation $(\star\star)$. Then there exist a positive real number $\delta$ and a $\Qq$-linear function $f(\bm{r})$ depending only on $\epsilon,c,m,\bm{r_0},\mathcal{C}$ satisfying the following.      
     \begin{enumerate}
         \item $f(\bm{r}_0)\ge\epsilon,$ and if $\epsilon\in\Qq$, then $f(\bm{r})=\epsilon$ for any $\bm{r}\in\Rr^c$.
     
         \item Assume that $(X\ni x,B(\bm{r}))$ is a surface germ such that $+\infty>\pld(X\ni x,B(\bm{r}_0))\ge\epsilon,$ then there exists a prime divisor $E$ over $X$ such that
         $$\pld(X\ni x,B(\bm{r}))=a(E,X,B(\bm{r}))$$ 
         for any $\bm{r}\in\Rr^c$ satisfying $||\bm{r}-\bm{r}_0||_\infty\le\delta$. Moreover, the $\Qq$-linear function $a(E,X,B(\bm{r}))$ only depends on $\pld(X\ni x,B(\bm{r}_0))$ and $\bm{r}_0$.
         
         \item Assume that $(X\ni x,B(\bm{r}))$ is a surface germ such that $\mld(X\ni x,B(\bm{r}_0))\ge\epsilon,$ then there exists a prime divisor $E$ over $X$ such that
         $$\mld(X\ni x,B(\bm{r}))=a(E,X,B(\bm{r}))\ge f(\bm{r})$$ 
         for any $\bm{r}\in\Rr^c$ satisfying $||\bm{r}-\bm{r}_0||_\infty\le\delta$. Moreover, the $\Qq$-linear function $a(E,X,B(\bm{r}))$ only depends on $\mld(X\ni x,B(\bm{r}_0))$ and $\bm{r}_0$.       
     \end{enumerate}     
\end{thm}

\begin{proof}
      Let $N_0'$ be the number defined in Lemma \ref{lemma 3.3} depending only on $\frac{\epsilon}{2}$ and $[\epsilon_0,1]$, and $N_0\ge N_0'$ the number defined in Lemma \ref{lemma 3.3} depending only on $N_0',\epsilon$ and $\mathcal{C}|_{\bm{r}_0}$. Let $N_1\ge N_0$ be the number defined in Lemma \ref{lemma pld=mld} depending only on $\frac{\epsilon}{2}$ and $N_0.$ Let $N_2$ be the number defined in Lemma \ref{lemma minimalreslwithbddvertices} depending only on $\epsilon$ and $N_1$. Let $N_3\ge N_2$ be the number defined in Lemma \ref{lemma mld set} depending only on $\epsilon,s_i(\bm{r}_0)$ and $N_2.$
      
      By \cite[Corollary 5.5]{HLS19}, there exists a positive real number $\delta_0$ depending only on $c,m,\bm{r}_0,\mathcal{C},\epsilon$ such that for any surface germ $(X\ni x,B(\bm{r}))$ satisfying $\pld(X\ni x,B(\bm{r}_0))\ge\epsilon$ $($respectively $\mld(X\ni x,B(\bm{r}_0))\ge\epsilon)$, then $\pld(X\ni x,B(\bm{r}))\ge\frac{\epsilon}{2}$ $($respectively $\mld(X\ni x,B(\bm{r}))\ge\frac{\epsilon}{2})$ for any $\bm{r}\in\Rr^c$ satisfying $||\bm{r}-\bm{r}_0||_{\infty}\le\delta_0.$ 
      \begin{claim}\label{claim:existenceofdelta}
      There exists a positive real number $\delta\le\delta_0$ depending only on $\epsilon,\mathcal{C},\delta_0,\bm{r}_0,N_1,N_3,c$ and $m$ such that
      \begin{enumerate}[label=(\roman*)]
          \item for any $f_1(\bm{r}),f_2(\bm{r})\in\Ff(N_3,\epsilon,\mathcal{C})$, if $f_1(\bm{r}_0)<f_2(\bm{r}_0)$ then $f_1(\bm{r})<f_2(\bm{r})$, and $s_i(\bm{r})\ge \epsilon_0$ for any $1\le i\le m$ and any $\bm{r}\in\Rr^c$ satisfying $||\bm{r}-\bm{r}_0||_{\infty}\le\delta$, and
          \item for any surface germ $(X\ni x,B(\bm{r}))$, if $\mld(X\ni x,B(\bm{r}_0))\ge\epsilon$ and the dual graph of the minimal resolution of $X$ has at most $N_1$ vertices, then
      $$\mld(X\ni x,B(\bm{r}))=a(E,X,B(\bm{r}))$$
      for some prime divisor $E$ over $X$ and any $\bm{r}\in\Rr^c$ satisfying $||\bm{r}-\bm{r}_0||_\infty\le\delta$.
      \end{enumerate}
      \end{claim}
      \begin{proof}[Proof of Claim \ref{claim:existenceofdelta}]
      Since $\mathcal{C}$ and $\Ff(N_3,\epsilon,\mathcal{C})$ are finite sets of linear functions, we may find a positive real number $\delta_1\le\delta_0$ depending only on $\epsilon,\mathcal{C},\Ff(N_3,\epsilon,\mathcal{C})$, $\delta_0,\bm{r}_0$ such that $(i)$ holds. It suffices to find a positive real number $\delta<\delta_1$ such that $(ii)$ holds.
      
      Indeed let $f:Y\to X$ be the minimal resolution of $X$ with exceptional divisors $E_1,\dots,E_{m'}$ $(m'\le N_1)$. We may write
      $$K_Y+B_Y(\bm{r}_0)+\sum_{i=1}^{m'}(1-a_i)E_i:=f^*(K_X+B(\bm{r}_0)),$$
      where $B_Y(\bm{r}_0)$ is the strict transform of $B(\bm{r}_0)$ and $a_i:=a(E_i,X,B(\bm{r}_0))$ for any $i$. By Lemma \ref{lemma minimalreslwithbddvertices},  $a_i\in\Ff(N_2,\epsilon,\mathcal{C})|_{\bm{r}_0}$. Then there exist $\Qq$-linear functions $s_{m+1}(\bm{r}),\dots$, $s_{m+m'}(\bm{r})\in\{1-a(\bm{r})\mid a(\bm{r})\in\Ff(N_2,\epsilon,\mathcal{C})\}$ such that $s_{m+i}(\bm{r}_0)=1-a_i$ for any $i.$ Let $B_Y'(\bm{r}):=B_Y(\bm{r})+\sum_{i=1}^{m'}s_{m+i}(\bm{r})E_i$, then we have
      $K_Y+B_Y'(\bm{r})=f^*(K_X+B(\bm{r}))$
      for any $\bm{r}\in\Rr^c$ satisfying $||\bm{r}-\bm{r}_0||_{\infty}\le\delta_1.$
      Since $Y$ is smooth, $\mld(Y/X\ni x,C)\le 2$ for any boundary $C$ on $Y$. By Lemma \ref{lemma bddcimldlinear}, we may find a positive real number $\delta_{(m')}<\delta_1$ depending only on $\epsilon,c,m+m',s_i(\bm{r}),\bm{r}_0,\delta_1$ such that
      $$\mld(Y/X\ni x,B_Y'(\bm{r}))=a(E,Y,B_Y'(\bm{r}))=a(E,X,B(\bm{r}))$$
      for some prime divisor $E$ over $X$ and any $\bm{r}\in\Rr^c$ satisfying $||\bm{r}-\bm{r}_0||_\infty\le\delta_{(m')}$. Let $\delta:=\min_{0\le i\le N_1}\{\delta_{(i)}\}$, the claim holds.
      \end{proof}
      Let $f(\bm{r})\in\Ff(N_3,\epsilon,\mathcal{C})$ such that $f(\bm{r}_0)=\min \Ff(N_3,\epsilon,\mathcal{C})|_{\bm{r}_0}$. We will show that $\delta$ and $f(\bm{r})$ have the required properties. \medskip
      
     \noindent (1) By the choice of $f(\bm{r})$, $f(\bm{r}_0)\ge\epsilon$. Note that if $\epsilon\in\Qq$, then $f(\bm{r}_0)=\epsilon$, and $f(\bm{r})=\epsilon$ for any $\bm{r}\in\Rr^c$ as $r_0,\dots,r_c$ are linearly independent over $\Qq$.
     \medskip
      
      \noindent (2) Let $f:Y\to X$ be the minimal resolution of $X,\mathcal{DG}$ the dual graph of $f$ and $\Ee(f)$ the set of $f$-exceptional divisors. If $\mathcal{DG}$ has at least $N_1$ vertices, then by Lemma \ref{lemma 3.3}, we have
      $$\pld(X\ni x,B(\bm{r}_0))=\min\{p_A(\bm{r}_0),p'_A(\bm{r}_0)\}\in\mathcal{P}(\epsilon,\mathcal{C})|_{\bm{r}_0},$$
      where $p_A(\bm{r})=\frac{(A+p_1)\beta_1(\bm{r})+p_2\beta_2(\bm{r})}{A+p_1+p_2}, p'_A(\bm{r})=\frac{(A+p'_2)\beta_2(\bm{r})+p'_1\beta_1(\bm{r})}{A+p'_1+p'_2}$ for some $A\in\Zz_{>0}$, $p_1,p_2,p_1',p_2'\in\mathcal{I}_0',\beta_1(\bm{r}),\beta_2(\bm{r})\in\Ll(\epsilon,\mathcal{C})$ and $p_1-p_1'=p_2'-p_2=1$.
      We may assume that %$\beta_1(\bm{r}_0)\le\beta_2(\bm{r}_0)$ and
      $$\pld(X\ni x,B(\bm{r}_0))=\min\{p_A(\bm{r}_0),p'_A(\bm{r}_0)\}=p_A(\bm{r}_0)=a(E,X,B(\bm{r}_0))$$
      for some prime divisor $E$ over $X$. In particular, $\beta_1(\bm{r}_0)\le\beta_2(\bm{r}_0)$ and hence $\beta_1(\bm{r})\le\beta_2(\bm{r})$ for any $\bm{r}\in\Rr^c$ satisfying $||\bm{r}-\bm{r}_0||_\infty\le\delta$.
      Thus
      $$p_A(\bm{r})-p_A'(\bm{r})=\frac{(A+1)(\beta_1(\bm{r})-\beta_2(\bm{r}))}{A+p_1+p_2}\le0$$
      and $\min\{p_A(\bm{r}),p'_A(\bm{r})\}=p_A(\bm{r})$ for any $\bm{r}\in\Rr^c$ satisfying $||\bm{r}-\bm{r}_0||_\infty\le\delta$. By Lemma \ref{lemma 3.3} again, we have
      $$\pld(X\ni x,B(\bm{r}))=\min\{p_A(\bm{r}),p'_A(\bm{r})\}=p_A(\bm{r})=a(E,X,B(\bm{r})),$$
      for any $\bm{r}\in\Rr^c$ satisfying $||\bm{r}-\bm{r}_0||_\infty\le\delta$.

      If $\mathcal{DG}$ has at most $N_1-1$ vertices, then by Lemma \ref{lemma minimalreslwithbddvertices}, $a(E,X,B(\bm{r}_0))\in\Ff(N_2,\epsilon,\mathcal{C})|_{\bm{r}_0}$ for any divisor $E\in\Ee(f)$. We may assume that 
      $$\pld(X\ni x,B(\bm{r}_0))=\min_{E\in\Ee(f)}\{a(E,X,B(\bm{r}_0))\}=a(E_1,X,B(\bm{r}_0))$$ 
      for some $E_1\in\Ee(f).$ By Claim \ref{claim:existenceofdelta}$(i)$,
      $$\pld(X\ni x,B(\bm{r}))=\min_{E\in\Ee(f)}\{a(E,X,B(\bm{r}))\}=a(E_1,X,B(\bm{r})),$$
      for any $\bm{r}\in\Rr^c$ satisfying $||\bm{r}-\bm{r}_0||_\infty\le\delta$. Since $r_0,\dots,r_c$ are linearly independent over $\Qq,$ the $\Qq$-linear function $a(E,X,B(\bm{r}))$ only depends on $\pld(X\ni x,B(\bm{r}_0))$ and $\bm{r}_0$.
      \medskip
      
      \noindent (3) %Let $\mathcal{DG}$ be the dual graph of the minimal resolution of  $X.$ 
      By Claim \ref{claim:existenceofdelta}$(ii)$, Lemma \ref{lemma pld=mld} and $(2)$, there exists a prime divisor $E$ over $X$ such that
      $$\mld(X\ni x,B(\bm{r}))=a(E,X,B(\bm{r}))$$ 
      for any $\bm{r}\in\Rr^c$ satisfying $||\bm{r}-\bm{r}_0||_\infty\le\delta$ as $\mld(X\ni x,B(\bm{r}))\ge\frac{\epsilon}{2}$.
      
      By Lemma \ref{lemma mld set}, $a(E,X,B(\bm{r}_0))\in(\Ff(N_3,\epsilon,\mathcal{C})\cup\mathcal{P}(\epsilon,\mathcal{C}))|_{\bm{r}_0}$. Hence $a(E,X,B(\bm{r}))$ $\in\Ff(N_3,\epsilon,\mathcal{C})\cup\mathcal{P}(\epsilon,\mathcal{C})$ as $r_0,r_1,\ldots,r_c$ are linearly independent over $\Qq$. Note that for any $p_A(\bm{r})=\frac{(A+p_1)\beta_1(\bm{r})+p_2\beta_2(\bm{r})}{A+p_1+p_2}\in\mathcal{P}(\epsilon,\mathcal{C}),\beta_1(\bm{r}),\beta_2(\bm{r})\in\Ll(\epsilon,\mathcal{C})$ and $\beta_2(\bm{r}_0)\ge\beta_1(\bm{r}_0)\ge f(\bm{r}_0).$ By Claim \ref{claim:existenceofdelta}$(i)$, $\beta_2(\bm{r})\ge\beta_1(\bm{r})\ge f(\bm{r})$ and hence  $p_A(\bm{r})\ge\beta_1(\bm{r})\ge f(\bm{r})$ for any $\bm{r}\in\Rr^c$ satisfying $||\bm{r}-\bm{r}_0||_\infty\le\delta$. In particular, $a(E,X,B(\bm{r}))\ge f(\bm{r})$ for any $\bm{r}\in\Rr^c$ satisfying $||\bm{r}-\bm{r}_0||_\infty\le\delta$.
\end{proof}

\begin{cor}\label{cor verticesofpolytope}
     Let $m$ be a positive integer, $\epsilon\le\epsilon'$ two non-negative real numbers and $\bm{v}=(v_1^0,\dots,v_m^0)\in\Rr^m$ a point.  Then there exist a positive integer $I$ and finite sets $\Ii_1\subseteq(0,1],\Ii_2\subseteq[0,1]\cap\Qq,\mathcal{M}\subseteq\Qq_{\ge0}$ depending only on $m,\epsilon,\epsilon'$ and $\bm{v}$ satisfying the following.
     
     Assume that $(X\ni x,B:=\sum_{j=1}^m v_j^0B_j)$ is a surface germ such that
     \begin{itemize}
         \item $B_i\ge0$ is a reduced Weil divisor for any $i$,
         \item $\pld(X\ni x,B)=\epsilon',$ and
         \item $\mld(X\ni x,B)\ge\epsilon.$
     \end{itemize}
     Then there exist $a_i\in \Ii_1,$ $B^i\in\Ii_2,$ and $\epsilon_i,\epsilon_i'\in\mathcal{M}$ such that
     \begin{enumerate}
         \item $\sum a_i=1$,
         \item $K_X+B=\sum a_i(K_X+B^i),$
         \item $\sum a_i\epsilon_i\ge\epsilon$ and $\sum a_i\epsilon'_i=\epsilon',$
         \item $\pld(X\ni x,B^i)=\epsilon_i'$ for any $i,$
         \item $\mld(X\ni x,B^i)\ge\epsilon_i$ for any $i,$
         \item $I(K_X+B^i)$ is Cartier near $x$ for any $i$, and
         \item if $\epsilon\in\Qq$, then $\epsilon_i=\epsilon$ for any $i$, if $\epsilon\neq1,$ then $\epsilon_i\neq1$ for any $i$.
     \end{enumerate}
\end{cor}

\begin{proof}
    %The result follows from Theorem \ref{thm locally linearity of mld}.
    If $\bm{v}\in\Qq^m$, then $\epsilon'\in \Qq$. By Proposition \ref{prop pld}, there exists an integer $I$ depending only on $\epsilon'$ and $\bm{v}$ such that %for any surface germ $(X\ni x,B=\sum_{i=1}^mv_iB_i)$ satisfying $\pld(X\ni x,B)=\epsilon',$ then 
    $I(K_X+B)$ is Cartier near $x.$ Let $\epsilon_0:=\{\gamma\ge\epsilon\mid\gamma\in\frac{1}{I}\Zz_{\ge0}\}.$
    Let $\Ii_1:=\{1\},\Ii_2:=\{v_1^0,\dots,v_m^0\}$ and $\mathcal{M}:=\{\epsilon_0,\epsilon'\}$. It is clear that the integer $I,$ the sets $\Ii_1,\Ii_2$ and $\mathcal{M}$ have the required properties.
    
    Suppose that $\bm{v}\in\Rr^m\setminus\Qq^m.$
    There exist a point $\bm{r}_0=(r_1,\dots,r_c)\in\Rr^c$ and $\Qq$-linear functions $s_i(\bm{r}):\Rr^{c}\to \Rr$ such that $r_0=1,r_1,\ldots,r_c$ are linearly independent over $\Qq$ and $s_i(\bm{r}_0)=v_i^0$ for any $1\le i\le m$.
    Since $\epsilon'\in\Span_{\Qq}(\{1,v_1^0,\dots,v_m^0\})$, there exists a $\Qq$-linear function $h(\bm{r})$ such that $h(\bm{r}_0)=\epsilon'.$
    
    If $\epsilon=\epsilon'=0$, then let $\delta$ be the number defined in \cite[Corollary 5.5]{HLS19} depending only on $c,m,\bm{r}_0,s_i(\bm{r})$, and let $f(\bm{r})=0$. In this case, for any surface germ $(X\ni x,B(\bm{r}):=\sum_{i=1}^ms_i(\bm{r})B_i)$, if $\pld(X\ni x,B(\bm{r}_0))=0$ then $\pld(X\ni x,B(\bm{r}))=0$ for any $\bm{r}\in\Rr^c$ satisfying $||\bm{r}-\bm{r}_0||_\infty\le\delta$, as $a(E,X,B(\bm{r}))$ is a linear function for any prime divisor $E$ over $X$ and $r_0,r_1,\ldots,r_c$ are linearly independent over $\Qq$. If $\epsilon'>\epsilon=0$, then let $\delta$ be the number defined in Theorem \ref{thm locally linearity of mld} depending only on $c,m,\epsilon',\bm{r}_0,s_i(\bm{r})$, and let $f(\bm{r})=0$. 
    If $\epsilon>0,$ then let $\delta$ and $f(\bm{r})$ be the real number and $\Qq$-linear function defined in Theorem \ref{thm locally linearity of mld} depending only on $c,m,\epsilon,\bm{r}_0,s_i(\bm{r})$. Note that, if $\epsilon\in\Qq$, then $f(\bm{r})=\epsilon$ for any $\bm{r}\in\Rr^c$. Then there exist positive real numbers $a_1,\dots,a_k$ and points $\bm{r}_1,\dots,\bm{r}_k\in\Qq^c$ such that 
    \begin{itemize}
        \item $\sum_{i=1}^{k} a_i=1,$
        \item $\sum_{i=1}^{k} a_i\bm{r}_i=\bm{r}_0,$ and
        \item $||\bm{r}_i-\bm{r}_0||_\infty\le \delta$ for any $1\le i\le k.$
    \end{itemize}
    Let $\Ii_1:=\{a_i\mid 1\le i\le k\},\Ii_2:=\{s_j(\bm{r}_i)\mid 1\le j\le m,1\le i\le k\}$ and 
    $\mathcal{M}:=\{\epsilon_i:=f(\bm{r}_i),\epsilon_i':=h(\bm{r}_i)\mid 1\le i\le k\}$.
    
    Let $B(\bm{r}):=\sum_{j=1}^m s_j(\bm{r})B_j$, then $\mld(X\ni x,B(\bm{r}_0))=\mld(X\ni x,B)\ge\epsilon$ and $\pld(X\ni x,B(\bm{r}_0))=\pld(X\ni x,B)=\epsilon'$. By our choice of $\delta$ and $f(\bm{r})$,
    $$\pld(X\ni x,B(\bm{r}))=h(\bm{r}),\,\mld(X\ni x,B(\bm{r}))\ge f(\bm{r})$$ for any $\bm{r}\in\Rr^c$ satisfying $||\bm{r}-\bm{r}_0||_\infty\le\delta.$ Let $B^i:=\sum_{j=1}^m s_j(\bm{r}_i)B_j\in\Ii_2$.
    %\epsilon_i=f(\bm{r}_i)$ and $\epsilon_i'=h(\bm{r}_i)$ for any $1\le i\le k$. 
    Then $a_i,\epsilon_i,\epsilon_i'$ and $(X\ni x,B^i)$ satisfy $(1)$--$(5)$ and $(7)$. The existence of $I$ follows from Proposition \ref{prop pld}. 
\end{proof}
\begin{comment}
\begin{thm}\label{thm:uniformpolytopeforsurfacemlds}
Let $\epsilon$ be a positive real number, $m$ a positive integer and $\bm{v}=(v_1^0,\ldots,v_m^0)\in\Rr^m$ a point. There exist positive real numbers $a_j$, positive rational numbers $\epsilon_j$ and a rational polytope $P$ containing $\bm{v}$ depending only on $\epsilon,m$ and $\bm{v}$ satisfying the following.  
	\begin{enumerate}
	   \item $\sum a_j=1$ and $\sum a_j\epsilon_j\ge\epsilon$.
	   \item Assume that $(X\ni x,\sum_{i=1}^m v_i^0B_{i})$ is a surface germ which is $\epsilon$-lc at $x$, where $B_1,\ldots,B_m\ge0$ are Weil divisors on $X$. Then the function $P\to \Rr$ defined by
  	   $$(v_1,\ldots,v_m)\mapsto\mld(X\ni x,\sum_{i=1}^m v_iB_{i})$$
	   is a linear function, and
	   $$\mld(X\ni x,\sum_{i=1}^m v_i^jB_{i})\ge \epsilon_j$$
	   for any vertex $(v_1^j,\ldots,v_m^j)$ of $P$.	
	\end{enumerate}
\end{thm}
\end{comment}

  Recall that we say that $V\subseteq\Rr^m$ is the rational envelope of $\bm{v}\in\Rr^m$, if $V$ is the smallest affine subspace containing $\bm{v}$ which is defined over the rationals.

\begin{proof}[Proof of Theorem \ref{thm:uniformpolytopeforsurfacemlds}]
      There exist $\Qq$-linearly independent real numbers $r_0=1,r_1,\ldots,r_c$ for some $0\le c\le m$, and $\Qq$-linear functions $s_i(\bm{r}):\Rr^{c}\to \Rr$ such that $s_i(\bm{r}_0)=v_i^0$ for any $1\le i\le m$, where $\bm{r}_0=(r_1,\dots,r_c)$. Note that the map $\Rr^c\to V$ defined by
     $$\bm{r}\mapsto (s_1(\bm{r}),\dots,s_m(\bm{r}))$$    
     is one-to-one, where $V\subseteq\Rr^m$ is the rational envelope of $\bm{v}$.
      
      If $c=0$, $P=V=\{\bm{v}\}$ and there is nothing to prove. Suppose that $c\ge1$. Let $B(\bm{r}):=\sum_{i=1}^m s_i(\bm{r})B_i$, then $B(\bm{r}_0)=\sum_{i=1}^m v_i^0B_{i}$ and $(X\ni x,B(\bm{r}_0))$ is $\epsilon$-lc at $x$. By Theorem \ref{thm locally linearity of mld}, there exist a positive real number $\delta$ and a $\Qq$-linear function $f(\bm{r})$ depending only on $\epsilon,c,m,\bm{r}_0,s_i(\bm{r})$ such that $f(\bm{r}_0)\ge\epsilon$, and
      $$\mld(X\ni x,B(\bm{r}))=a(E,X,B(\bm{r}))\ge f(\bm{r})$$
      for some prime divisor $E$ over $X$ and any $\bm{r}\in\Rr^c$ satisfying $||\bm{r}-\bm{r}_0||_\infty\le\delta$.  
      %Moreover, if $\epsilon\in\Qq$, then we may additionally require that $f(\bm{r})\ge\epsilon$ for any $\bm{r}\in\Rr^c$ satisfying $||\bm{r}-\bm{r}_0||_\infty\le\delta$. 
      We may find $2^c$ positive rational numbers $r_{i,1},r_{i,2}$ such that $r_{i,1}<r_i<r_{i,2}$ and $\max\{r_i-r_{i,1},r_{i,2}-r_i\}\le\delta$ for any $1\le i\le c$. By our choice of $\delta$, the function $\Rr^c\to \Rr$ defined by
      $$\bm{r}\mapsto \mld(X\ni x,B(\bm{r}))$$
      is a linear function on $\bm{r}\in U_c:=[r_{1,1},r_{1,2}]\times\cdots\times[r_{c,1},r_{c,2}]$. 
      
      Let $\bm{r}_j$ be the vertices of $U_c,\epsilon_j:=f(\bm{r}_j),v_i^j:=s_i(\bm{r}_j)$ for any $i,j$. Note that if $\epsilon\in\Qq$, then $\epsilon_j=f(\bm{r}_j)=\epsilon$ for any $j$. Let $P:=\{(s_1(\bm{r}),\dots,s_m(\bm{r}))\mid \bm{r}\in U_c\}\subseteq V$. Then the function $P\to \Rr$ defined in the theorem is a linear function, $(v_1^j,\dots,v_m^j)$ are vertices of $P,$ and
      $$\mld(X\ni x,\sum_{i=1}^m v_i^jB_i)=\mld(X\ni x,B(\bm{r}_j))\ge f(\bm{r}_j)=\epsilon_j$$ 
      for any $j$. 
      
      Finally, we may find positive real numbers $a_j$ such that $\sum a_j=1$ and $\sum a_j\bm{r}_j=\bm{r}_0$. Then $\sum a_j\bm{v}_j=\bm{v}$ and $\sum a_j\epsilon_j\ge\epsilon$ as $\sum_j a_j v_i^j=\sum_j a_js_i(\bm{r}_j)=s_i(\sum_j a_j\bm{r}_j)=s_i(\bm{r}_0)=v_i^0$ for any $1\le i\le m,$ and $\sum a_j\epsilon_j=\sum a_jf(\bm{r}_j)=f(\sum a_j\bm{r}_j)=f(\bm{r}_0)\ge\epsilon.$
\end{proof}	

\section{Existence of decomposed complements}
The main goal of this section is to show 
Theorem \ref{thm: decompsable epsilon complement for surfaces}. 
\subsection{Inversion of stability} As the first step in the proofs of Theorem \ref{thm epsilonlccompforsurfaces} and Theorem \ref{thm: decompsable epsilon complement for surfaces}, we need to reduce two theorems to the case when $\Ii$ is a finite set. The proofs are similar to \cite{HLS19}.
\begin{defn}[Relative $a$-lc thresholds]\label{defn alct} 
     Suppose that $(X/Z\ni z, B)$ is $a$-lc over $z$ for some non-negative real number $a$. The \emph{$a$-lc threshold} of $(X/Z\ni z,B)$ with respect to an $\Rr$-Cartier divisor $M\ge0$ is defined as
     $$\alct(X/Z\ni z,B;M):=\sup\{t\ge0\mid (X/Z\ni z,B+tM)\text{ is $a$-lc over }z\}.$$
     In particular, if $a = 0$ and $X=Z$, then we obtain the \emph{lc threshold}.    
\end{defn}

\begin{conj}[ACC for $a$-lc thresholds]\label{conj: ACC for aLCTs}
     Fix a positive integer $d$ and a non-negative real number $a$. Let $\Ii_1$ and $\Ii_2$ be two DCC sets, then 
     \begin{align*}
     a\text{-}LCT(d,\Ii_1,\Ii_2)=&\{\alct(X\ni x,B;M)\mid (X\ni x,B)\text{ is $a$-lc at }x,\\
     &\dim X=d,B\in\Ii_1,M\in\Ii_2\}
     \end{align*}
     satisfies the ACC. 
\end{conj}
Hacon-M\textsuperscript{c}Kernan-Xu \cite{HMX14} proved Conjecture \ref{conj: ACC for aLCTs} for $a=0$.
It is natural to wonder if the relative version of Conjecture \ref{conj: ACC for aLCTs} holds.

\begin{conj}[ACC for relative $a$-lc thresholds]\label{conj: relative ACC for aLCTs}
     Fix a positive integer $d$ and a non-negative real number $a$. Let $\Ii_1$ and $\Ii_2$ be two DCC sets, then 
     \begin{align*}
     a\text{-}LCT_{rel}(d,\Ii_1,\Ii_2)=&\{\alct(X/Z\ni z,B;M)\mid (X/Z\ni z,B)\text{ is $a$-lc over }z,\\
     &\dim X=d,B\in \Ii_1,M\in\Ii_2\}
     \end{align*}
     satisfies the ACC. 
\end{conj}
     
\begin{comment}
  The following conjecture is due to Shokurov.
\begin{conj}[{\cite[Problem 5]{Sho88}}, ACC for MLDs]\label{conj: ACC for mlds}
      Let $d$ be a positive integer and $\Ii\subseteq[0,1]$ a DCC set. Then the set
      $$\mathcal{M}ld(d,\Ii):=\{\mld(X\ni x,B)\mid (X\ni x,B) \text{ is lc at }x, \dim X=d, B\in \Ii\}$$
      satisfies the ACC. 
\end{conj}
\end{comment}
In order to show Conjecture \ref{conj: relative ACC for aLCTs}, it is enough to show Conjecture \ref{conj: ACC for aLCTs}.
\begin{lem}\label{lem:localalctimplyrel}
     Let $d$ be a positive integer and $a$ a non-negative real number. Then the ACC for $a$-lc thresholds in dimension $d$ $($Conjecture \ref{conj: ACC for aLCTs}$)$ implies the ACC for relative $a$-lc thresholds in dimension $d$ $($Conjecture \ref{conj: relative ACC for aLCTs}$)$. In particular, Conjecture \ref{conj: relative ACC for aLCTs} holds in dimension $2$.
\end{lem}
\begin{proof}
     Let $t:=\alct(X/Z\ni z,B;M)$, then either $t$ is the lc threshold of $(X,B)$ with respect to $M$ over a neighborhood of $\pi^{-1}(z)$ or $\mld(X/Z\ni z,B+tM)=a$, where $\pi$ is the morphism $X\to Z$. %By Lemma \ref{lemma mldcenterfinite}, 
     In the former case, $t$ belongs to an ACC set by \cite[Theorem 1.1]{HMX14}. In the latter case, it suffices to show that $\mld(X\ni x,B+tM)=a$ for some point $x\in \pi^{-1}(z)$, and hence $t=\alct(X\ni x,B;M)$ belongs to an ACC set by assumption. Indeed by Lemma \ref{lem:mldisattained}, $\mld(X/Z\ni z,B+tM)=a(E,X,B+tM)$ for some $E\in e(z)$. Let $x$ be the generic point of $\Center_XE$, then $\mld(X/Z\ni z,B+tM)=\mld(X\ni x,B+tM).$
     
     By the ACC for minimal log discrepancies in dimension 2 (c.f. Lemma \ref{thm accformld}) and \cite[Main Theorem 1.8]{BirkarSho10}, Conjecture \ref{conj: ACC for aLCTs} holds in dimension $2$.
     %Hence the result follows from \cite[Main Theorem 1.8]{BirkarSho10}.
\end{proof}

\begin{lem}\label{lem: accdcc projection}
Let $\alpha$ be a positive real number, $\Ii\subseteq\Rr_{\ge0}$ a bounded DCC set and $\Ii''=\bar{\Ii''}$ an ACC set. Then there exist a finite set $\Ii'\subseteq \bar\Ii$ and a projection $g:\bar\Ii\to \Ii'$ (i.e., $g\circ g=g$) depending only on $\alpha$, $\Ii$ and $\Ii''$ satisfying the following properties:
 \begin{enumerate}
 	\item $\gamma+\alpha \ge g(\gamma)\ge \gamma$ for any $\gamma\in \Ii$, 
 	\item $g(\gamma')\ge g(\gamma)$ for any $\gamma'\ge \gamma$, $\gamma,\gamma'\in\Ii$, and
 	\item for any $\beta\in \Ii''$ and $\gamma\in \Ii$, if $\beta\ge \gamma$, then $\beta\ge g(\gamma)$.
 \end{enumerate}
\end{lem}
\begin{proof}	
	We may replace $\Ii$ by $\bar\Ii$ and therefore assume that $\Ii=\bar\Ii$. Since $\Ii$ is a bounded DCC set, there exists a positive real number $M\ge1$, such that $\Ii\subseteq [0,M]$. 
		
	Let $N:=\lceil\frac{M}{\alpha}\rceil,\Ii_0:=\Ii\cap [0,\frac{1}{N}]$, and
	$$\Ii_k:=\Ii\cap (\frac{k}{N},\frac{k+1}{N}]$$
	for any $1\leq k\leq N-1$.
	
	For any $\gamma\in\Ii$, there exists a unique $0\leq k\leq N-1$ such that $\gamma\in\Ii_{k}$. If $\gamma\in\Ii$ and $\gamma>\max_{\beta\in\Ii''}\{\beta\}$, then we let $g(\gamma)=\max_{\beta\in \Ii_k}\{\beta\}$ which has the required properties. In the following, we may assume that for any $\gamma\in\Ii$, $\gamma\le\max_{\beta\in\Ii''}\{\beta\}$.
	
	Since $\Ii=\bar\Ii$ and $\Ii''=\bar\Ii''$, we may define $f:\Ii\to\Ii'', g:\Ii\to\Ii$ in the following ways,
	$$f(\gamma):=\min\{\beta\in\Ii''\mid \beta\geq \gamma\},\text{and}\ g(\gamma):=\max\{\beta\in\Ii_k\mid\beta\leq f(\gamma),\gamma\in\Ii_k\}.$$
	For any $\gamma\in\Ii$, it is clear that 
	$$0\le g(\gamma)-\gamma\leq\frac{1}{N}\leq\alpha.$$
	Since $f(\gamma)\ge g(\gamma)\geq \gamma$, 
	$f(g(\gamma))=f(\gamma)$ and $g(g(\gamma))=g(\gamma)$. We will show that $\Ii':=\{g(\gamma)\mid\gamma\in\Ii\}$ has the required properties.
	
	It suffices to show that $\Ii'$ is a finite set. Since $\Ii'\subseteq\Ii$, $\Ii'$ satisfies the DCC. We only need to show that $\Ii'$ satisfies the ACC. Suppose that there exists a strictly increasing sequence $g(\gamma_1)<g(\gamma_2)<\ldots,$ where $\gamma_j\in \Ii$. Since $f(\gamma_j)$ belongs to the ACC set $\Ii''$, possibly passing to a subsequence, we may assume that $f(\gamma_j)$ is decreasing. Thus $g(\gamma_j)$ is decreasing, and we get a contradiction. Therefore $\Ii'$ satisfies the ACC. 	
\end{proof}

The proof of Theorem \ref{thm: dcc limit epsilon-lc divisor} is very similar to \cite[Theorem 5.21]{HLS19}.
\begin{thm}\label{thm: dcc limit epsilon-lc divisor}
Let $d$ be a positive integer, $\epsilon$ a non-negative real number, $\alpha$ a positive real number, and $\Ii\subseteq [0,1]$ a DCC set. Suppose that Conjecture \ref{conj: ACC for aLCTs} holds in dimension $d$ with $a=\epsilon$. Then there exist a finite set $\Ii'\subseteq \bar\Ii$ and a projection $g:\bar\Ii\to \Ii'$ depending only on $d,\epsilon,\alpha$ and $\Ii$ satisfying the following. 

Assume that $(X/Z\ni z,B:=\sum_{i=1}^sb_iB_i)$ is a pair of dimension $d$ such that
\begin{itemize}
    \item $b_i\in\Ii$ for any $1\le i\le s$,
    \item $B_i\ge0$ is a $\Qq$-Cartier Weil divisor for any $1\le i\le s$, and
    \item $(X/Z\ni z,B)$ is $\epsilon$-lc over $z$.
\end{itemize}
Then
\begin{enumerate}
	\item  $\gamma+\alpha \ge g(\gamma)\ge \gamma$ for any $\gamma\in\Ii$, 
\item $g(\gamma')\ge g(\gamma)$ for any $\gamma'\ge \gamma$, $\gamma,\gamma'\in\Ii$, and
	\item  $(X/Z\ni z,\sum_{i=1}^s g(b_i)B_i)$ is $\epsilon$-lc over $z$.
\end{enumerate}
\end{thm}

\begin{proof}
Possibly replacing $\Ii$ by $\bar\Ii$, we may assume that $\Ii=\bar\Ii$. Let $\Ii'':=\overline{a\text{-}LCT_{rel}(d,\Ii,\Zz_{\ge0})}$. Since we assume that Conjecture \ref{conj: ACC for aLCTs} holds in dimension $d$ with $a=\epsilon$, by Lemma \ref{lem:localalctimplyrel}, $\Ii''$ satisfies the ACC. By Lemma \ref{lem: accdcc projection}, there exists a projection $g:\bar\Ii\to \Ii'$ satisfying Lemma \ref{lem: accdcc projection}(1)--(3).

It suffices to show that $(X/Z\ni z,\sum_{i=1}^s g(b_i)B_i)$ is $\epsilon$-lc over $z$. Otherwise, there exists some $1\le j\le s$, such that $(X/Z\ni z,\sum_{i=1}^j g(b_i)B_i+\sum_{i=j+1}^s b_iB_i)$ is $\epsilon$-lc over $z$, and $(X/Z\ni z,\sum_{i=1}^{j+1} g(b_i)B_i+\sum_{i=j+2}^s b_iB_i)$ is not $\epsilon$-lc over $z$. Let 
$$\beta:=\alct(X/Z\ni z,\sum_{i=1}^{j} g(b_i)B_i+\sum_{i=j+2}^s b_iB_i;B_{j+1}).$$
Then $g(b_{j+1})> \beta\ge b_{j+1}$. Since $g(b_i),b_i\in\Ii$ for any $i$, we have $\beta\in \Ii''$ which contradicts Lemma \ref{lem: accdcc projection}(3).
\end{proof}

\begin{lem}\label{lem:run anti pair MMP}
Let $(X/Z\ni z,B)$, $(X/Z\ni z,B')$ be two pairs such that
\begin{enumerate}
    \item $\Supp B'\subseteq \Supp B,\Supp\lfloor B\rfloor\subseteq\Supp\lfloor B'\rfloor$, and
    \item $(X/Z\ni z,B)$ has an $\Rr$-complement $(X/Z\ni z,B+G)$ with $\lfloor B+G\rfloor=\lfloor B\rfloor$.
\end{enumerate}
Then there exist a positive real number $u$ and a pair $(X/Z\ni z,\Delta)$, such that $-(K_X+B')\sim_{\Rr} u(K_X+\Delta)$ over a neighborhood of $z$. 

Moreover,  if $(X/Z\ni z,B+G)$ is dlt over a neighborhood of $z$, then $(X/Z\ni z,\Delta)$ is dlt over a neighborhood of $z$.
% for some positive real number $u$ and $(X/Z\ni z,\Delta)$.
\end{lem}
\begin{proof}
 By assumption, there exists a positive real number $\epsilon_0<1$, such that $B-\epsilon_0B'\ge0$, and any coefficient of $\Delta:=\frac{B-\epsilon_0B'}{1-\epsilon_0}+\frac{G}{1-\epsilon_0}$ is at most 1. Then over a neighborhood of $z$, we have
   	$-\epsilon_0 (K_X+B')\sim_{\Rr}(1-\epsilon_0)(K_X+\Delta),$
 %where $\Delta:=\frac{B-\epsilon_0B'}{1-\epsilon_0}+\frac{G}{1-\epsilon_0}$. 
 Hence $(X/Z\ni z,\Delta)$ and $u:=\frac{1-\epsilon_0}{\epsilon_0}$ have the required properties.
   	
   	If $(X/Z\ni z,B+G)$ is dlt over a neighborhood of $z$, then we may choose $0<\epsilon_0\ll 1$, such that $(X/Z\ni z,\Delta)$ is dlt over a neighborhood of $z$.
\end{proof}

The proof of Theorem \ref{thm: dcc limit epsilon-lc complementary} is very similar to that of \cite[Theorem 5.22]{HLS19}.
\begin{thm}\label{thm: dcc limit epsilon-lc complementary}
	Let $d$ be a positive integer, $\epsilon$ a non-negative real number, and $\Ii\subseteq [0,1]$ a DCC set. Suppose that either $d\le 2$ or Conjecture \ref{conj: ACC for aLCTs} holds in dimension $d$ with $a=\epsilon$. Then there exist a finite set $\Ii'\subseteq \bar\Ii$, and a projection $g:\bar\Ii\to \Ii'$ depending only on $\epsilon,d$ and $\Ii$ satisfying the following. 
	
	Assume that $(X/Z\ni z,B:=\sum b_iB_i)$ is a pair such that
	\begin{itemize}
	    \item $X$ is $\Qq$-factorial of dimension $d$,
	    \item $B_i\ge0$ is a Weil divisor for any $i$,
	    \item $b_i\in\Ii$ for any $i$,
	    \item $(X/Z\ni z,B)$ has an $(\epsilon,\Rr)$-complement $(X/Z\ni z,B+G)$ with  $\lfloor B+G\rfloor=\lfloor B\rfloor$, and
	    \item either $(X/Z\ni z,B+G)$ is dlt over a neighborhood of $z$ or $d\le 2$.
	    
	\end{itemize}
	Then possibly shrinking $Z$ near $z,$ we have
	\begin{enumerate}
		\item  $g(\gamma)\ge \gamma$ for any $\gamma\in \Ii$,
	    \item $g(\gamma')\ge g(\gamma)$ for any $\gamma'\ge \gamma$, $\gamma,\gamma'\in\Ii$, 
		\item  $(X/Z\ni z,\sum g(b_i)B_i)$ is $\epsilon$-lc over $z$, and
		\item $-(K_X+\sum g(b_i)B_i)$ is pseudo-effective over $Z$.
	%	\item $-(K_X+\sum g(b_i)B_i)\sim_{\Rr,Z} u(K_X+\Delta)$ for some positive real number $u$ and lc pair $(Y/Z,\Delta)$, and
	\end{enumerate}

	Moreover, if one of the following holds:
\begin{enumerate}[label=(\roman*)]
	    \item $d\le 2$, or
	    \item $X$ is of Fano type over $Z$, or
	    \item the abundance conjecture holds in dimension $d$.
\end{enumerate}
	Then $-(K_X+\sum g(b_i)B_i)$ has a good minimal model over $Z$, and $(X/Z\ni z,\sum g(b_i)B_i)$ is $(\epsilon,\Rr)$-complementary. In particular, there exists a positive real number $\tau$ depending only on $d,\epsilon$ and $\Ii$ such that for any $B,B'\in \Ii$, if $B\le B'$ and $||B-B'||<\tau$, then $(X/Z\ni z,B)$ is $(\epsilon,\Rr)$-complementary if and only if $(X/Z\ni z,B')$ is $(\epsilon,\Rr)$-complementary.
\end{thm}
\begin{proof}
	Suppose that the theorem does not hold. By Theorem \ref{thm: dcc limit epsilon-lc divisor}, there exist a sequence of $\Qq$-factorial pairs of dimension $d$, $(X_{k}/Z_k\ni z_k,B_{(k)}:=\sum_{i} b_{k,i}B_{k,i})$, and a sequence of projections $g_{k}:\bar\Ii\to \bar{\Ii}$, such that
	\begin{itemize}
	    \item $b_{k,i}\in\Ii,b_{k,i}+\frac{1}{k}\ge g_k(b_{k,i})\ge b_{k,i}$ for any $k,i$,
	    \item $(X_{k}/Z_k\ni z_k,B_{(k)})$ has an $(\epsilon,\Rr)$-complement $(X_{k}/Z_{k}\ni z_k,B_{(k)}+G_{k})$ with $\lfloor B_{(k)}+G_k\rfloor=\lfloor B_{(k)}\rfloor$ for any $k,$
	    \item either $(X_{k}/Z_{k}\ni z_k,B_{(k)}+G_{k})$ is dlt over a neighborhood of $z_k$ for any $k$ or $d\le2$,
	    \item $(X_{k}/Z_k\ni z_k,B_{(k)}':=\sum_{i} g_k(b_{k,i})B_{k,i})$ is $\epsilon$-lc over $z_k$ for any $k,$ and
	    \item $-(K_{X_{k}}+B_{(k)}')$ is not pseudo-effective over $Z_k$ for any $k$. 
	\end{itemize}
 %, moreover, we may assume that $-(K_{X_{k}}+B_{(k)}')$ is not pseudo-effective over $Z_k$
	Possibly shrinking $Z_k$ near $z_k$, by Lemma \ref{lem:run anti pair MMP}, $-(K_{X_{k}}+B_{(k)}')\sim_{\Rr,Z_k}u_k(K_{X_k}+\Delta_k)$ for some positive real number $u_k$ and $\Qq$-factorial pair $(X_k,\Delta_k)$ for any $k$, and if $(X_{k}/Z_{k}\ni z_k,B_{(k)}+G_{k})$ is dlt over a neighborhood of $z_k$, then so is $(X_k,\Delta_k)$. By \cite[Theorem 1.1]{Fuj12} and \cite{BCHM10}, we may run a $-(K_{X_{k}}+B_{(k)}')$-MMP with scaling of an ample divisor over $Z_k$, and reach a Mori fiber space $Y_k\to Z_k'$ over $Z_k$, such that $-(K_{Y_k}+B_{(Y_{k})}')$ is antiample over $Z_k'$, where $B_{(Y_{k})}'$ is the strict transform of $B_{(k)}'$ on $Y_k$. Since $-(K_{X_k}+B_{(k)})$ is pseudo-effective over $Z_k$, $-(K_{Y_k}+B_{(Y_{k})})$ is nef over $Z_k'$, where $B_{(Y_{k})}$ is the strict transform of $B_{(k)}$ on $Y_k$. 
	
	For each $k$, there exist a positive integer $k_j$ and a positive real number $0\le b_k^{+}\le 1$, such that $b_{k,k_j}\le b_k^{+}< g_k(b_{k,k_j})$, and $K_{F_k}+B_{F_k}^{+}\equiv 0$, where
	$$K_{F_k}+B_{F_k}^{+}:=\left(K_{Y_{k}}+(\sum_{i<k_j}g_k(b_{k,i})B_{Y_k,i})+b_k^{+}B_{Y_k,k_j}+(\sum_{i>k_j}b_{k,i}B_{Y_k,i})\right)|_{F_k},$$
	$B_{Y_k,i}$ is the strict transform of $B_{k,i}$ on $Y_k$ for any $i$, and $F_k$ is a general fiber of $Y_k\to Z_k'$. Since $(X_{k}/Z_k\ni z_k,B_{(k)})$ is $(\epsilon,\Rr)$-complementary, $(Y_{k},B_{(Y_{k})})$ is lc over $z_k$. Thus $(Y_{k}/Z_k\ni z_k,B_{(Y_{k})}')$ is lc over $z_k$ by Theorem \ref{thm: dcc limit epsilon-lc divisor}, hence $(F_k,B_{F_k}^{+})$ is lc.
	
	Since $g_k(b_{k,k_j})$ belongs to the DCC set $\bar{\Ii}$ for any $k,k_j$, possibly passing to a subsequence, we may assume that $g_k(b_{k,k_j})$ is increasing. Since $g_k(b_{k,k_j})-b_k^{+}>0$ and 
	$$\lim_{k\to +\infty} (g_k(b_{k,k_j})-b_k^{+})=0,$$ 
	by Lemma \ref{lem: strictly increasing seq b_k}, possibly passing to a subsequence, we may assume that $b_{k}^{+}$ is strictly increasing. 
	
	Now $K_{F_k}+B_{F_k}^{+}\equiv0$, the coefficients of $B_{F_k}^{+}$ belong to the DCC set $\bar{\Ii}\cup\{b_k^{+}\}_{k=1}^{\infty}$, and $b_{k}^{+}$ is strictly increasing. This contradicts the global ACC \cite[Theorem 1.4]{HMX14}.
	
	If one of $(i)$-$(iii)$ holds, then there exists a good minimal model $-(K_{X'}+\sum g(b_i)B_{i}')$ of $-(K_X+\sum g(b_i)B_i)$ over a neighborhood of $z$, where $B_{i}'$ is the strict transform of $B_i$ on $X'$. Since $(X/Z\ni z,\sum b_iB_i)$ is $(\epsilon,\Rr)$-complementary, $(X'/Z\ni z,\sum b_iB_{i}')$ is $(\epsilon,\Rr)$-complementary. Hence $(X'/Z\ni z,\sum b_iB_{i}')$ is $\epsilon$-lc over $z$, and $(X'/Z\ni z,\sum g(b_i)B_{i}')$ is $\epsilon$-lc over $z$.
	%Since $-(K_Y+\sum g(b_i)B_{i}')$ is semiample over a neighborhood of $z$. 
	By Lemma \ref{lem semiamplecomplement}, $(X'/Z\ni z,\sum g(b_i)B_{i}')$ is $(\epsilon,\Rr)$-complementary, thus $(X/Z\ni z,\sum g(b_i)B_i)$ is also $(\epsilon,\Rr)$-complementary.
	
	Let $\tau:=\min\{\gamma'-\gamma>0\mid \gamma'\in\bar{\Ii},\gamma\in\Ii'\}$, we will show that $\tau$ has the required properties. We may write $B:=\sum b_iB_i$ and $B':=\sum b_i'B_i$, where $B_i\ge0$ are Weil divisors on $X$. It is enough to show that $g(b_i)\ge b_i'$ for any $i$. Indeed as $g(b_i)\in\Ii'$ and $b_i'-g(b_i)\le b_i'-b_i<\tau$, $g(b_i)\ge b_i'$ by the choice of $\tau$.
\end{proof}

\begin{lem}\label{lem: strictly increasing seq b_k}	
Let $\{a_k\}_{k=1}^{\infty}$ be an increasing sequence of real numbers, and $\{b_k\}_{k=1}^{\infty}$ a sequence of real numbers, such that $b_k<a_k$ for any $k$, and $\lim_{k\to+\infty}(a_k-b_k)=0$. Then possibly passing to a subsequence, we may assume that $b_k$ is strictly increasing.
\end{lem}
\begin{proof}
     It suffices to show that for any $k\ge 1,$ there exists an integer $k'>k$ such that $b_{k'}>b_k.$ Let $c_j=a_j-b_j$ for any $j,$ then $c_j>0$ and $\lim_{j\to +\infty}c_j=0$ by assumption. Thus there exists an integer $k'>k,$ such that $c_{k'}<c_k.$ Since $a_{k'}>a_k,b_{k'}=a_{k'}-c_{k'}>a_{k}-c_{k}=b_k.$
\end{proof}

The following theorem is a slight generalization of \cite[Main Proposition 2.1]{BirkarSho10}, that is, we assume Conjecture \ref{conj: ACC for aLCTs} instead of the ACC conjecture for MLDs, and we do not assume $(X\ni x,B')$ is lc near $x$.

\begin{thm}[Inversion of stability for $\epsilon$-lc pairs]\label{thm: inversion of stability e-lc pair}
	Let $d$ be a positive integer, $\epsilon$ a non-negative real number, and $\Ii\subseteq [0,1]$ a DCC set. Suppose that Conjecture \ref{conj: ACC for aLCTs} holds in dimension $d$ with $a=\epsilon$. There exists a positive real number $\tau$ depending only on $d,\epsilon$ and $\Ii$ satisfying the following. 
	
	Assume that $(X\ni x,B)$ and $(X\ni x,B')$ are two pairs of dimension $d$ such that
	\begin{enumerate}
	    \item $B\le B'$,  $||B-B'||<\tau,B'\in\Ii$, and
	    \item $(X\ni x,B)$ is $\epsilon$-lc at $x$.
	\end{enumerate}
	Then $(X\ni x,B')$ is also $\epsilon$-lc at $x$.
\end{thm}

\begin{proof}
Suppose that the theorem does not hold, then there exist two sequence of pairs $(X_k\ni x_k,B_k),(X_k\ni x_k,B_k')$, such that $B_k\le B'_k$,  $||B_k-B'_k||<\frac{1}{k},B_k'\in\Ii$, $(X_k\ni x_k,B_k)$ is $\epsilon$-lc at $x_k$, and $(X_k\ni x_k,B_k')$ is not $\epsilon$-lc at $x_k$. 

 We may write $B_{k}:=\sum_{i=1}^{m_{k}}b_{ki}B_{ki}$ and $B'_{k}:=\sum_{i=1}^{m_{k}}b'_{ki}B_{ki}$ for any $k$, where $B_{ki}$ are distinct prime divisors on $X_{k}$ for any $k$. We first show that the set $\tilde{\Ii}=\{b_{ki}\mid k,i\}$ satisfies the DCC. Otherwise there exists a strictly decrease sequence $b_{i_{l}j_{l}}\in\tilde{\Ii},b_{i_{1}j_{1}}>b_{i_{2}j_{2}}>\dots.$ Since $i_{l}\in\Zz_{>0}$ and $b'_{i_{l}j_{l}}$ belongs to the DCC set $\Ii$, possibly passing to a subsequence, we may assume that $i_{l}$ is strictly increasing and $b'_{i_{l}j_{l}}$ is increasing, and hence $\lim_{l\to +\infty} (b'_{i_{l}j_{l}}-b_{i_{l}j_{l}})=0$. By Lemma \ref{lem: strictly increasing seq b_k}, possibly passing to a subsequence we may assume that $b_{i_{l}j_{l}}$ is strictly increasing, a contradiction. Possibly replacing $\Ii$ by $\Ii\cup \tilde{\Ii}$, we may assume that $B_k\in \Ii$ for any $k$.

Possibly shrinking $X_k$ near $x_k$, we may assume that $(X_k,B_k)$ is lc. Let $f_k:Y_k\to X_k$ be a dlt modification of $(X_k,B_k)$, and write 
$$K_{Y_k}+B_{Y_k}+E_{Y_k}=f_k^{*}(K_{X_k}+B_k),$$
where $B_{Y_k}$ is the strict transform of $B_k$, and $E_{Y_k}$ is the sum of reduced $f_k$-exceptional divisors.

By Theorem \ref{thm: dcc limit epsilon-lc divisor}, possibly passing to a subsequence, we may assume that $({Y_k}/X_k\ni x_k,B_{Y_k}'+E_{Y_k})$ is $\epsilon$-lc over $x_k$ for any $k$, where $B_{Y_k}'$ is the strict transform of $B_k'$. Possibly shrinking $X_k$ near $x_k$, by Lemma \ref{lem:run anti pair MMP}, there exist a positive real number $u_k$ and a dlt pair $(Y_k,\Delta_k)$, such that
$-(K_{Y_k}+B_{Y_k}'+E_{Y_k})\sim_{\Rr,X_k} u_k(K_{Y_k}+\Delta_k)$. Let $E_{Y_k}'+B_{k}'-B_k=f_k^{*}(B_k'-B_k)$. Then
$$-(K_{Y_k}+B_{Y_k}'+E_{Y_k})\sim_{\Rr,X_k}B_{Y_k}-B'_{Y_k}\sim_{\Rr,X_k}E_{Y_k}'.$$
By \cite[Theorem 1.8]{Birkar12}, we may a run a $-(K_{Y_k}+B_{Y_k}'+E_{Y_k})$-MMP with scaling of an ample divisor over $X_k$, and terminates with a model $W_k$ over $X_k$ such that $-(K_{W_k}+B_{W_k}'+E_{W_k})\sim_{\Rr,X_k}E_{W_k}'=0$, where $B_{W_k}',E_{W_k}$ and $E_{W_k}'$ are the strict transforms of $B_{Y_k}',E_{Y_k}$ and $E_{Y_k}'$ respectively. Thus $K_{W_k}+B_{W_k}'+E_{W_k}$ is the pullback of $K_{X_k}+B_k'$. Since $(W_k/X_k\ni x_k,B_{W_k}+E_{W_k})$ is $(\epsilon,\Rr)$-complementary, by Theorem \ref{thm: dcc limit epsilon-lc complementary}, possibly passing to a subsequence, we may assume that $({W_k}/X_k\ni x_k,B_{W_k}'+E_{W_k})$ is $\epsilon$-lc over $x_k$, hence $({X_k}\ni x_k,B_{k}')$ is $\epsilon$-lc at $x_k$, a contradiction.
\end{proof}
\begin{rem}
If $X$ is $\Qq$-factorial, then we may show that Theorem \ref{thm: inversion of stability e-lc pair} holds for pairs $(X/Z\ni z,B)$ and $(X/Z\ni z, B')$ by assuming Conjecture \ref{conj: ACC for aLCTs} in dimension $d$ with $a=\epsilon$.
\end{rem}

\subsection{Proof of Theorem \ref{thm: decompsable epsilon complement for surfaces} for surface germs}

    Let $(X/Z\ni z,B)$ be a surface germ which is lc over $z$, we define
    $$\mathcal{P}(X/Z\ni z,B):=\{\pld(X\ni x,B)\mid x\in X \text{ is a closed point},\pi(x)=z\}.$$

\begin{lem}\label{lemma finitelogdis}
    Let $\Ii\subseteq[0,1]$ and $\mathcal{P}\subseteq[0,1]$ be two finite sets. Then there exists a discrete set $\mathcal{LD}_0\subseteq\Rr_{\ge0}$ depending only on $\Ii$ and $\mathcal{P}$ satisfying the following.
     
     Assume that $(X/Z\ni z,B)$ is a surface germ such that
     \begin{itemize}
       \item $B\in\Ii,$
       %\item $z\in Z$ is a closed point,
       \item $(X/Z\ni z,B)$ is lc over $z,$ and
       \item  $\mathcal{P}(X/Z\ni z,B)\subseteq{\mathcal{P}}\cup\{+\infty\}$.
     \end{itemize}
     Then $a(E,X,B)\in\mathcal{LD}_0$ for any prime divisor $E\in e(z)$. 
\end{lem}

\begin{proof}
    By Corollary \ref{cor verticesofpolytope}, there exist a finite set $\Ii_1$ of positive real numbers, a finite set $\Ii_2$ of non-negative rational numbers and a positive integer $I$ depending only on $\Ii$ and $\mathcal{P}$ such that 
    \begin{itemize}
        \item $\sum a_i=1,$
        \item $K_X+B=\sum a_i(K_X+B^i),$ and
        \item $I(K_X+B^i)$ is Cartier near $x$ for any $i$ and any closed point $x$ satisfying $\pi(x)=z$, where $\pi$ is the morphism $X\to Z$,
    \end{itemize}
    for some $a_i\in\Ii_1$ and $B^i\in\Ii_2$. In particular, $I(K_X+B^i)$ is Cartier over a neighborhood of $z$ for any $i$. Let 
    $$\mathcal{LD}_0:=\{\sum \alpha_i\beta_i\mid \alpha_i\in\Ii_1,\beta_i\in\frac{1}{I}\Zz_{\ge0}\}.$$
    Since $a(E,X,B)=\sum a_ia(E,X,B^i)$ and $a(E,X,B^i)\in\frac{1}{I}\Zz_{\ge0},a(E,X,B)\in \mathcal{LD}_0$ for any prime divisor $E\in e(z)$.
\end{proof}

\begin{defn}
     Let $(X,B)$ and $(X',B')$ be two pairs. We say that $(X',B')$ is a \emph{non-positive birational model} of $(X,B)$ if there exists a common birational resolution $p:W\to X'$ and $q:W\to X$, such that
     $$p^*(K_{X'}+B')\ge q^*(K_X+B).$$
\end{defn}%\han{D is a non-positive birational model or $(X,B)$?}

\begin{rem}
    By definition, if $(X',B')$ is a non-positive birational model of $(X,B),$ then any non-positive birational model of $(X',B')$ is a non-positive birational model of $(X,B).$
\end{rem}

\begin{comment}
\begin{thm}\label{thm reducetofiniteandfixpld}
     Let $\epsilon$ be a non-negative real number and $\Ii\subseteq[0,1]$ a DCC set. Then there exists a finite set $\hat{\Ii}\subseteq[0,1]$ depending only on $\Ii$ and $\epsilon$ satisfying the following.
     
     Assume that $(X/Z\ni z,B)$ is a surface germ such that $B\in\Ii,$ and $(X/Z\ni z,B)$ is $(\epsilon,\Rr)$-complementary. Then possibly shrinking $Z$ near $z$, there exists a non-positive birational model $(\hat{X}/Z\ni z,\hat{B})$ of $(X/Z\ni z,B)$ with the following properties:
     \begin{enumerate}
        \item  $\hat{X}$ is smooth,
        \item $\hat{X}\to X$ is a morphism,
        \item  $\hat{B}\in\hat{\Ii}$,
        \item $-(K_{\hat{X}}+\hat{B})$ is semiample over $Z$, and
        %\item $\mathcal{P}(Y/Z\ni z,B_Y)\subseteq\hat{\mathcal{P}}\cup\{+\infty\},$
        \item  $(\hat{X}/Z\ni z,\hat{B})$ is $(\epsilon,\Rr)$-complementary.
        \end{enumerate}
        Moreover, if $\epsilon>0$, then we may pick $\hat{X}=Y$, where $Y$ is the minimal resolution of $X$.
\end{thm}
\end{comment}
 The proof of Theorem \ref{thm reducetofiniteandfixpld} is similar to \cite[Main Theorem 1.6]{Bir04}.
\begin{thm}\label{thm reducetofiniteandfixpld}
     Let $\epsilon$ be a non-negative real number and $\Ii\subseteq[0,1]$ a DCC set. Then there exists a finite set $\hat{\Ii}\subseteq[0,1]\cap (\Qq+\Span_{\Qq_{\ge0}}(\bar{\Ii}))$ depending only on $\Ii$ and $\epsilon$ satisfying the following.
     
     Assume that $(X/Z\ni z,B)$ is a surface germ such that $B\in\Ii,$ and $(X/Z\ni z,B)$ is $(\epsilon,\Rr)$-complementary. Then possibly shrinking $Z$ near $z$, there exists a boundary $B_Y$ on the minimal resolution $Y$ of $X$ such that
     \begin{enumerate}
        \item $(Y,B_Y)$ is a non-positive birational model of $(X,B)$,
        \item  $B_Y\in\hat{\Ii}$,
        \item $-(K_{Y}+B_Y)$ is semiample over $Z$, and
        \item  $(Y/Z\ni z,B_Y)$ is $(\epsilon,\Rr)$-complementary.
        \end{enumerate}
        \end{thm}

\begin{proof}
    \textbf{Step 1.} In this step, we show that there exist two finite sets $\tilde{\Ii}\subseteq[0,1]$ and  $\tilde{\mathcal{P}}\subseteq[0,1]$ depending only on $\Ii$ and $\epsilon$ such that there exists a non-positive birational model $(\tilde{X}/Z\ni z,\tilde{B})$ of $(X/Z\ni z,B)$ with the following properties:
    \begin{itemize}
        \item $Y$ dominates $\tilde{X}$,
        \item  $\tilde{B}\in\tilde{\Ii}$,
        \item $-(K_{\tilde{X}}+\tilde{B})$ is semiample over $Z$, 
        \item $\mathcal{P}(\tilde{X}/Z\ni z,\tilde{B})\subseteq\tilde{\mathcal{P}}\cup\{+\infty\}$, and
        \item  $(\tilde{X}/Z\ni z,\tilde{B})$ is $(\epsilon,\Rr)$-complementary.
    \end{itemize}
    
    Let $\pi$ be the morphsim $X\to Z$ and $\mathfrak{P}:=\{x\in\pi^{-1}(z)\text{ is a closed point}\mid \pld(X\ni x,0)=0\}.$
    Note that for any $x\in\mathfrak{P}$, $\pld(X\ni x,B)=\pld(X\ni x,0)=0$ and $12(K_X+B)$ is Cartier near $x$ by the classification of lc surface singularities. Let $f':X'\to X$ be the morphism such that $f'$ is the minimal resolution near $x$ for any $x\in \mathfrak{P}$, and is an isomophism over $X\backslash{\mathfrak{P}}$. In particular, $Y$ dominates $X'$. Note that if $\mathfrak{P}=\emptyset$, then $f'$ is the identity map.
    We may write 
    $K_{X'}+B':=(f')^*(K_X+B),$ then $B'\in\Ii\cup(\frac{1}{12}\Zz_{\ge0}\cap[0,1])$
    and $(X'/Z\ni z,B')$ is $(\epsilon,\Rr)$-complementary. Replacing $\Ii$ and $(X/Z\ni z,B)$ by $\Ii\cup(\frac{1}{12}\Zz_{\ge0}\cap[0,1])$ and $(X'/Z\ni z,B')$ respectively, we may assume that $\mathfrak{P}=\emptyset$. In particular, $X$ is $\Qq$-factorial klt over a neighborhood of $z$. 
    
    Let $(X/Z\ni z,B+G)$ be an $(\epsilon,\Rr)$-complement of $(X/Z\ni z,B)$. Possibly replacing $B,G$ by $B-B\wedge \lfloor B+G\rfloor+\lfloor B+G\rfloor,G-G\wedge \lfloor B+G\rfloor$ respectively, we may assume that $\lfloor B\rfloor=\lfloor B+G\rfloor$.
    
    By Theorem \ref{thm: dcc limit epsilon-lc complementary}, we may assume that $\Ii$ is a finite set, and $-(K_X+B)$ has a unique good minimal model $X''$ over $Z$ such that $-(K_{X''}+B'')$ is semiample over $Z$, where $B''$ is the strict transform of $B$. Possibly replacing $(X,B)$ by $(X'',B'')$, we may assume that $-(K_X+B)$ is semiample over $Z.$
     
     By Theorem \ref{cor accforpld}, $\mathcal{P}ld(\Ii)$ is an ACC set and hence $\Ii':=D(\Ii)\cup\{1-\gamma\ge0\mid\gamma\in\mathcal{P}ld(\Ii)\}$ satisfies the DCC. Let $\tau$ be the number defined in Theorem \ref{thm: dcc limit epsilon-lc complementary} depending only on $\Ii'$ and $\epsilon$. If $\epsilon=0$, let $\epsilon'=\tau$, otherwise let $\epsilon'=\epsilon.$
     Let 
     $$\tilde{\Ii}=\Ii\cup(D(\Ii)\cap[0,1-\epsilon'])\subseteq D(\Ii)\subseteq \Span_{\Qq\ge0}(\Ii)$$ 
     be the union of two finite sets. 
     Note that when $\epsilon>1,$ $[0,1-\epsilon']=\emptyset.$
     By Theorem \ref{cor accforpld}, $\mathcal{P}ld(\tilde{\Ii})$ is an ACC set, thus
     $$\tilde{\mathcal{P}}:=\mathcal{P}ld(\tilde{\Ii})\backslash (\lim\mathcal{P}ld(\tilde{\Ii}))^{\tau},$$
     is a finite set, 
     where $(\lim\mathcal{P}ld(\tilde{\Ii}))^{\tau}:=\cup_{b\in\lim\mathcal{P}ld(\tilde{\Ii)}}(b,b+\tau)$. By construction,
    $$(\lim\mathcal{P}ld(\tilde{\Ii}))^{\tau}\subseteq \cup_{b\in\lim\mathcal{P}ld(\epsilon',\tilde{\Ii)}}(b,b+\tau)\cup (0,\tau),$$
    where $\mathcal{P}ld(\epsilon',\tilde{\Ii}):=\mathcal{P}ld(\tilde{\Ii})\cap[\epsilon',+\infty)$. Let $\mathcal{P}_0:=\{1\}\cup\{1-b\mid b\in \lim \mathcal{P}ld(\epsilon',\tilde{\Ii})\}$. By Lemma \ref{lemma pld set} and Lemma \ref{lemma DI}, we have
     \begin{align*}
         \mathcal{P}_0\subseteq D(\tilde{\Ii})\cap[0,1-\epsilon']
         \subseteq D(D(\Ii))\cap[0,1-\epsilon']
         =D(\Ii)\cap[0,1-\epsilon']\subseteq \tilde{\Ii}.
     \end{align*}
     
     From now on, we construct a sequence of surface germs $(X_i/Z\ni z,B_{X_i})$,
     $$X:=X_1\dashrightarrow X_2\dashrightarrow\cdots\dashrightarrow X_l=\tilde{X},$$ 
     with the following properties:
     \begin{itemize}
        \item $B_{X_i}\in\tilde{\Ii}$,
        \item $-(K_{X_i}+B_{X_i})$ is semiample over $Z$,
        \item $(X_i/Z\ni z,B_{X_i})$ is $(\epsilon,\Rr)$-complementary, 
      %  \item $X\dashrightarrow X_{i}$ only contracts divisors $E$ with $a(E,X_i,B_i)< 1$, \han{check <1 or $\le 1$}
        \item $(X_i,B_{X_i})$ is a non-positive birational model of $(X,B)$,
        \item $\mathcal{P}(\tilde{X}/Z\ni z,\tilde{B})\subseteq\tilde{\mathcal{P}}\cup\{+\infty\}$, where $\tilde{B}:=B_{X_l}$, and
        \item $Y$ dominates $X_i$.
     \end{itemize}     
      Assume that the germ $(X_i/Z\ni z,B_{X_i})$ is constructed for some $i\ge1$. If $\mathcal{P}(X_i/Z\ni z,B_{X_i})\nsubseteq\tilde{\mathcal{P}}\cup\{+\infty\}$, then exists a closed point $x_i\in \pi_i^{-1}(z)$, such that $0<a_i:=\pld(X_i\ni x_i,B_{X_i})\in(\lim\mathcal{P}ld(\tilde{\Ii}))^{\tau}$ and $(X_i,B_{X_i})$ is dlt near $x_i$, where $\pi_i$ is the morphism $X_i\to Z$. There exist $ b_i\in \lim\mathcal{P}ld(\epsilon',\tilde{\Ii})\cup\{0\}$, and a prime divisor $E_i$ over $X_i$, such that ${\rm Center}_{X_i}E_i=\{x_i\}$, $b_i<a(E_i,X_i,B_{X_i})=a_i<b_i+\tau$, and the minimal resolution $f_i:Y_i\to X_i$ factors through $p_i: W_i\to X_i$, where $p_i$ is the extraction of $E_i$.
      Let
     $$K_{W_i}+B_{W_i}+(1-a_i)E_i=p_i^{*}(K_{X_i}+B_{X_i}),$$
     where $B_{W_i}$ is the strict transform of $B_{X_i}$. Since $(W_i/Z\ni z,B_{W_i}+(1-a_i)E_i)$ is $(\epsilon,\Rr)$-complementary and $-(K_{W_i}+B_{W_i}+(1-a_i)E_i)$ is semiample over $Z,$ by Lemma \ref{lem semiamplecomplement}, there exists an $\Rr$-Cartier divisor $G_{W_i}\ge0$ such that $(W_i/Z\ni z,B_{W_i}+(1-a_i)E_i+G_{W_i})$ is an $(\epsilon,\Rr)$-complement of $(W_i/Z\ni z,B_{W_i}+(1-a_i)E_i)$ and $\lfloor B_{W_i}+(1-a_i)E_i+G_{W_i}\rfloor=\lfloor B_{W_i}+(1-a_i)E_i\rfloor$. By Theorem \ref{thm: dcc limit epsilon-lc complementary}, $(W_i/Z\ni z,B_{W_i}+(1-b_i)E_i)$ is $(\epsilon,\Rr)$-complementary and $-(K_{W_i}+B_{W_i}+(1-b_i)E_i)$ has a unique good minimal model $X_{i+1}$ over $Z$, such that $-(K_{X_{i+1}}+B_{X_{i+1}})$ is semiample over $Z,$ where $B_{X_{i+1}}:=(q_i)_*(B_{W_i}+(1-b_i)E_i)$ and $q_i:W_i\to X_{i+1}$. 
     %The MMP only contracts divisors $E$ with $a(E,X_{i+1},B_{i+1})< a(E,X_i,B_i)< 1$. 
     It follows that $(X_{i+1}/Z\ni z,B_{X_{i+1}})$ is $(\epsilon,\Rr)$-complementary. As $1-b_i\in \mathcal{P}_0\subseteq\tilde{\Ii}$, $B_{X_{i+1}}\in\tilde{\Ii}$. Since $Y$ dominates $X_i$ and $Y_i$ dominates $W_i$ and $X_{i+1}$, $Y$ dominates $X_{i+1}$.

     Let %$$d(X_i,B_i):=\sum_{r\in \tilde{\Ii}}\sum_{\frac{r}{\epsilon}\ge m\ge1}\#\{E\mid a(E,X_i,B_i)>\frac{r}{m},E\subseteq\Bb_i\},$$
     $$d(X_i,B_{X_{i}}):=\sum_{r\in \tilde{\Ii}}|\{E\mid a(E,X_i,B_{X_{i}})>1-r,E\in\Bb_i\}|,$$
     where $\Bb_i:=\Ee(f_i)\cup\Supp (f_i)_*^{-1}B_{X_i},$ and $\Ee(f_i)$ is the set of exceptional divisors of $f_i$. It is clear that $d(X_i,B_{X_{i}})<+\infty$ since both $\Bb_i$ and $\tilde{\Ii}$ are finite sets. 

     Let $\rho_i:=\rho(Y_i)$. Then $\rho_{i+1}\le\rho_i$ as $Y_i$ is a resolution of $X_{i+1}$. We will prove the following claim.  
     \begin{claim} 
          If $\rho_i=\rho_{i+1}$, then $d(X_i,B_{X_{i}})-1\ge d(X_{i+1},B_{X_{i+1}}).$
     \end{claim} 
     Assume that the claim holds, then $\mathcal{P}(X_l/Z\ni z,B_{X_l})\subseteq\tilde{\mathcal{P}}\cup\{+\infty\}$ for some positive integer $l$ and we are done.
     
     \begin{proof}[Proof of the Claim]
     If $\rho_i=\rho_{i+1}$, then $Y_i$ is the minimal resolution of $X_{i+1}$ and $Y_i=Y_{i+1}$. Since $1-b_i\in\mathcal{P}_0\subseteq\tilde{\Ii}$, $b_i=a(E_i,X_{i+1},B_{X_{i+1}})<a(E_i,X_i,B_i)$, and $(X_{i+1},B_{X_{i+1}})$ is a non-positive birational model of $(X_i,B_{X_{i}})$, it suffices to show that $\Bb_i= \Bb_{i+1}$.
     
     Since $Y_i$ is the minimal resolution of $X_{i+1}$, any exceptional divisor of $W_i\to X_{i+1}$ belongs to $\Ee(f_{i+1})$, and $\Ee(f_i)\subseteq \Ee(f_{i+1})\cup \{E_i\}$. It suffices to show that any exceptional divisor $F_i$ of $W_i\to X_{i+1}$ belongs to $\Supp B_{W_i}$. Assume that $F_i$ is contracted in some step of $-(K_{W_i}+B_{W_i}+(1-b_i)E_i)$-MMP, $\phi: W_i'\to W_i''$ (here we still denote the divisor by $F_i$). Since $Y_i$ is the minimal resolution of $W_i''$ and $F_i$ is exceptional over $W_i''$, there exists $c\ge0$ such that 
     $$ K_{W_i'}+cF_i=\phi^{*}K_{W_i''}.$$ 
     By the negativity lemma, $F_i^2<0$, and $K_{W_i'}\cdot F_i\ge0$. Since $-(K_{W_i}+B_{W_i}+(1-a_i)E_i)=-p_i^{*}(K_{X_i}+B_{X_{i}})$ is nef over $Z$ and $-(K_{W_i'}+B_{W_i'}+(1-b_i)E_i')\cdot F_i<0$, $E_i\cdot F_i>0$ and
     $$(K_{W_i'}+B_{W_i'}+(1-t) E_{W_i'})\cdot F_i=0,$$
     for some $t\in(b_i,a_i]$, where $B_{W_i'}$ and $E_{W_i'}$ are the strict transforms of $B_{W_i}$ and $E_i$ on $W_i'$ respectively. Therefore, $B_{W_i'}\cdot F_i<0$ which implies that $F_i\subseteq\Supp B_{W_i'}$. 
     \end{proof}
     \medskip
     
     \textbf{Step 2.} According to Lemma \ref{lemma finitelogdis}, there exists a finite set $\hat{\mathcal{P}}$ depending only on $\tilde{\Ii}$ and $\tilde{P}$ and such that $a(E,\tilde{X},\tilde{B})\in\hat{\mathcal{P}}$ for any prime divisor $E\in e(z)$ with $a(E,\tilde{X},\tilde{B})\le1$. Let
     \begin{align*}
     \hat{\Ii}:=&\tilde{\Ii}\cup(\{1-\alpha\ge0\mid \alpha\in\hat{\mathcal{P}}\}\cap (\Qq+\Span_{\Qq_{\ge0}}\tilde{\Ii}))\\
     \subseteq &[0,1]\cap (\tilde{\Ii}\cup(\Qq+\Span_{\Qq_{\ge0}}\tilde{\Ii})))\\
     \subseteq &[0,1]\cap (\Qq+\Span_{\Qq_{\ge0}}(\bar{\Ii})).
     \end{align*}
     We will show that $\hat{\Ii}$ has the required properties.
     
    We may write 
     $K_{Y}+B_Y:=h^*(K_{\tilde{X}}+\tilde{B})$, where $h$ is the morphism $Y\to\tilde{X}$. Since $Y\to X$ is the minimal resolution of $X$ and $(\tilde{X},\tilde{B})$ is a non-positive birational model of $(X,B)$, $B_Y$ is a boundary and $B_Y\in\hat{\Ii}$. Since $K_{Y}+B_Y$ is the crepant pullback of $K_{\tilde{X}}+\tilde{B}$, it satisfies $(1),(3),(4)$, and we are done.
     \end{proof}

\begin{defn}
     A curve $C$ on $X/Z$ is called \emph{extremal} if it generates an extremal ray $R$ of $\overline{\rm NE}(X)$ which defines a contraction morphism $X\to Y/Z$ and if for some ample divisor $H$, we have $H\cdot C=\min\{H\cdot \Sigma\mid  \Sigma\in [R]\}$, where $\Sigma$ ranges over curves generating $R$.
\end{defn}

     We will need the following lemma. The proof is very similar to that of \cite[Proposition 2.9]{HanLiu18}, see also \cite{Sho96lcflip,Sho09,Birkar11}.

\begin{lem}\label{lemma uniformnefpolytope}
     Let $I,d,m$ be three positive integers, $\bm{v}:=(v_1^0,\ldots,v^0_m)\in\Rr^m$ a point, $V\subseteq\Rr^m$ the rational envelope of $\bm{v}$, $\Ll\subseteq V$ a rational polytope, $\bm{v}_j:=(v_{1}^j,\ldots,v_{m}^j)$ the vertices of $\Ll$. Then there exists an open set $U\ni \bm{v}$ of $V$ depending only on $I,d,m,\bm{v}$ and $\bm{v}_j$ satisfying the following.

     Assume that $(X,B:=\sum_{i=1}^m v_i^0B_i)$ is an lc pair of dimension $d$, and $X\to Z$ is a contraction such that
     \begin{itemize}
       \item $B_i\ge0$ is a reduced Weill divisor for any $i$,
       \item $(X,B^j:=\sum_{i=1}^m v_{i}^jB_i)$ is lc for any $j$,
       \item $I(K_X+B^j)$ is Cartier for any $j$, and
       \item $-(K_X+B)$ is nef over $Z$.
     \end{itemize}
     Then $-(K_X+\sum_{i=1}^m u_iB_i)$ is nef over $Z$ for any $(u_1,\ldots,u_m)\in U$.
\end{lem}

\begin{proof}
     For any curve $C_0$ such that $(K_X+B)\cdot C_0=0$, we have 
     $-(K_X+\sum_{i=1}^m u_iB_i)\cdot C_0=0$ for any $(u_1,\ldots,u_m)\in V$ as $V$ is the rational envelope of $\bm{v}$. Thus it suffices to show that there exists an open set $U$ of $V$, such that if $(K_X+B)\cdot R<0$ for some extremal ray $R$, then $(K_X+\sum_{i=1}^m u_iB_i)\cdot R<0$ for any $(u_1,\ldots,u_m)\in U$. 

     If $\bm{v}\in\Qq^{m},$ then $V=\{\bm{v}\}$ and there is nothing to prove. Thus we may assume that $\bm{v}\notin\Qq^{m}$.
     We may find positive real numbers $a_j$ and lc pairs $(X,B^{j})$ with such that $\sum a_j=1$, and
     $$K_X+B=\sum a_j(K_X+B^j).$$
     By the cone theorem, there is a curve $\Sigma_j$ generating $R$ such that $(K_X+B^j)\cdot\Sigma_j\ge-2d$ for any $j$. Let $C$ be an extremal curve of $R$, and $H$ an ample Cartier divisor, such that $H\cdot C=\min\{H\cdot \Sigma\mid \Sigma\in [R]\}$.

     Since both $C$ and $\Sigma_j$ generate $R$, we have
     $$(K_X+B^j)\cdot C=((K_X+B^j)\cdot\Sigma_j)\left(\frac{H\cdot C}{H\cdot \Sigma_j}\right)\ge-2d,$$
     for any $j$. Thus
     \begin{align*}
        0&>(K_X+B)\cdot C=\sum a_j(K_X+B^j)\cdot C
        \\&\ge -\sum_{j\neq j_0}a_j\cdot 2d+a_{j_0}(K_X+B^{j_0})\cdot C
        \ge -2d+a_{j_0}(K_X+B^{j_0})\cdot C,
     \end{align*}
     for any $j_0$. Since $I(K_X+B^{j_0})$ is Cartier, there are only finitely many possibilities for the numbers $a_{j_0}(K_X+B^{j_0})\cdot C$, and hence for $(K_X+B)\cdot C$. Thus there is a positive real number $\alpha$ which only depends on $I,d,m,\bm{v}$ and $\bm{v}_j$, such that $(K_X+B)\cdot C<-\alpha$.

     Let $||.||$ be a norm on $V$, $\Ff_{\bm{v}}$ the set of faces of $\Ll$ which does not contain $\bm{v}$, and $\beta:=dist(\Ff_{\bm{v}},\bm{v})$. Let $\delta:=\frac{\alpha \beta}{2d}$. We claim that 
     $$U:=\{\bm{u}\mid\bm{u}\in V,||\bm{v}-\bm{u}||<\delta\}$$ 
     has the required properties. Suppose on the contrary that there exist a pair $(X/Z,B:=\sum_{i=1}^m v_i^0B_i)$ satisfying the conditions, an extremal curve $C_{1}$ and $D:=\sum_{i=1}^m c_iB_i$ with $(c_1,\dots,c_m)\in U$ such that $(K_X+B)\cdot C_{1}<0$ while $(K_X+D)\cdot C_{1}\ge0$. Let $R_{1}'$ be the ray going from $D$ to $B$, and let $D'$ be the intersection point of $R_{1}'$ and $\partial\Ll$. We have
     $$(K_X+D')\cdot C_{1}\leq \frac{\beta+\delta}{\delta} (K_X+B)\cdot C_{1} <-(1+\frac{2d}{\alpha}) \alpha < -2d,$$
     which contradicts the cone theorem. 
\end{proof}

\begin{lem}\label{lemma antinefpolytope2}
     Let $\epsilon$ be a non-negative real number, $\Ii\subseteq[0,1]$ and $\mathcal{P}\subseteq[0,1]$ two finite sets. Then there exist finite sets $\Ii_1\subseteq(0,1],\Ii_2\subseteq[0,1]\cap\Qq,\mathcal{M}\subseteq\Qq_{\ge0}$ and a positive integer $I$ depending only on $\Ii,\mathcal{P}$ and $\epsilon$ satisfying the following.
     
     Assume that $(X/Z\ni z,B)$ is a $\Qq$-factorial surface germ such that
     \begin{itemize}
       \item $B\in\Ii,$
       \item $-(K_X+B)$ is semiample over $Z,$
       \item $\mathcal{P}(X/Z\ni z,B)\subseteq\mathcal{P}\cup\{+\infty\},$ and
       \item  $(X/Z\ni z,B)$ is $(\epsilon,\Rr)$-complementary.
     \end{itemize}
     Possibly shrinking $Z$ near $z$, there exist $a_i\in\Ii_1$, $B^i\in\Ii_2$ and $\epsilon_i\in\mathcal{M}$ such that
     \begin{enumerate}
       \item $\sum a_i=1,$
       \item $\sum a_i\epsilon_i     \ge\epsilon$,
       \item  $K_X+B=\sum a_i(K_X+B^i),$
       \item $(X/Z\ni z,B^i)$ is $(\epsilon_i,\Rr)$-complementary for any $i$,
       \item $-I(K_X+B^i)$ is Cartier and semiample over $Z$ for any $i$, and
       
       \item if $\epsilon\in\Qq$, then  $\epsilon=\epsilon_i$ for any $i$, and if $\epsilon\neq 1$, then $\epsilon_i\neq1$ for any $i$.
     \end{enumerate}
\end{lem}

\begin{proof}
   Possibly shrinking $Z$ near $z$, by Corollary \ref{cor verticesofpolytope} and Lemma \ref{lemma uniformnefpolytope}, there exist a positive integer $I$, a finite set $\Ii_1$ of positive real numbers and two finite sets $\Ii_2,\mathcal{M}$ of non-negative rational numbers depending only on $\Ii,\mathcal{P}$ and $\epsilon$ such that there exist $a_i\in\Ii_1,B^{i}\in\Ii_2$ and $\epsilon_i\in\mathcal{M}$ satisfying (1)--(3), (5)--(6). Moreover, $(X/Z\ni z,B^{i})$ is $\epsilon_i$-lc over $z$ for any $i$. It suffices to show that $(X/Z\ni z,B^{i})$ is $(\epsilon_i,\Rr)$-complementary over $z$ for any $i$.
 
     By Lemma \ref{lem semiamplecomplement}, $(X/Z\ni z,B)$ has an $(\epsilon,\Rr)$-complement $(X/Z\ni z,B+G)$ such that $\lfloor B+G\rfloor=\lfloor B\rfloor$. Possibly shrinking $Z$ near $z$, by Lemma \ref{lem:run anti pair MMP}, $-(K_X+B^i)\sim_{\Rr,Z}u_i(K_X+\Delta_i)$ for some positive real number $u_i$ and surface pair $(X,\Delta_i)$.
     
     By the relative abundance theorem for $\Qq$-factorial log surfaces \cite[Theorem 8.1]{Fuj12}, $-(K_X+B^{i})$ is semiample over $Z$ for any $i$. The lemma follows from Lemma \ref{lem semiamplecomplement}.
\end{proof}
 \begin{thm}\label{thm:decomp comp for closed points}
     Theorem \ref{thm: decompsable epsilon complement for surfaces} holds for the case when $\dim z=0.$
\end{thm}
\begin{proof}
     According to Theorem \ref{thm reducetofiniteandfixpld}, there exists a finite set $\hat{\Ii}\subseteq[0,1]$ depending only on $\Ii$ and $\epsilon,$ such that there exists a boundary $B_Y$ on the minimal resolution $Y$ of $X$ satisfying properties (1)--(4) in Theorem \ref{thm reducetofiniteandfixpld}.

     By Lemma \ref{lemma antinefpolytope2}, there exist a finite set $\Ii_1\subseteq(0,1]$ and two finite sets $\Ii_2,\mathcal{M}$ of non-negative rational numbers depending only on $\hat{\Ii}$ and $\hat{\mathcal{P}}$ such that 
     \begin{itemize}
       \item $\sum a_i=1,$
       \item $\sum a_i\epsilon_i\ge\epsilon$,
       \item  $K_{Y}+B_Y=\sum a_i(K_{Y}+B_Y^i)$,
       \item $(Y/Z\ni z,B_Y^i)$ is $(\epsilon_i,\Rr)$-complementary for any $i$, and
       \item if $\epsilon\in\Qq$, then  $\epsilon=\epsilon_i$ for any $i$, and if $\epsilon\neq 1$, then $\epsilon_i\neq1$ for any $i$,
     \end{itemize}
     for some $a_i\in\Ii_1,B_Y^i\in\Ii_2$ and $\epsilon_i\in\mathcal{M}.$ 
     
     Let $B^i$ be the strict transform of $B_Y^i$ on $X$, then $B^i\in\Ii_2$ and $(X/Z\ni z,B^i)$ is $(\epsilon_i,\Rr)$-complementary for any $i$. Since $(Y/Z\ni z,B_Y)$ is a non-positive birational model of $(X/Z\ni z,B)$, $\sum a_i(K_X+B^i)\ge K_X+B$. 
\end{proof}

\subsection{Proof of Theorem \ref{thm: decompsable epsilon complement for surfaces}}
\begin{lem}\label{lemma Amb99rel}

    Suppose that $(X/Z\ni z,B)$ is lc over $z$. Then there exists an open subset $U$ of $Z$ such that $U\cap \bar{z}\neq \emptyset$ and
    $$\mld(X/Z\ni z_{cp},B)=\mld(X/Z\ni z,B)+\dim {z}$$
    for any closed point $z_{cp}\in \bar{z}\cap U.$ 

\end{lem}
	
When $X=Z$ and $X\to Z$ is the identity map, we obtain \cite[Proposition 2.3]{Ambro99}. The proof of Lemma \ref{lemma Amb99rel} is similar to that of \cite[Proposition 2.3]{Ambro99}.

\begin{proof}
    
    We may assume that $(X,B)$ is lc and $\dim{z}>0$. 
    Let $g:Y\to X$ be a log resolution of $(X,B)$,     $$K_Y+B_Y:=K_Y+\sum_{i\in \mathcal{I}}(1-a_i)B_i=g^{*}(K_X+B),$$
    such that $\sum_{i\in\mathcal{I}} B_i$ is snc, $B_i$ are distinct prime divisors, $1\in \mathcal{I}$, $\mu(B_1)=\bar{z}$, $a_1=\mld(X/Z\ni z,B_1)$, and for each $i\in\mathcal{I}\backslash\{1\}$, $B_i$ is either the strict transform of an irreducible component of $\Supp B$ or a $g$-exceptional divisor, where $a_i:=a(B_i,X,B)$. Let $\mu=\pi\circ g$, where $\pi$ is the morphism $X\to Z$. Possibly replacing $Y$ by a higher model, and shrinking $Z$ near ${z}$, we may further assume that
	\begin{itemize}
		\item there exists a non-empty subset $\mathcal{I}_z\subseteq\mathcal{I}$ such that $\mu(B_i)=\bar{z}$ for any $i\in \mathcal{I}_z$, 
		\item either $\dim z=\dim Z$ and $\mathcal{I}=\mathcal{I}_z$ or $\dim z<\dim Z$ and $\mu^{-1}(\bar{z})=\bigcup_{i\in\mathcal{I}_z}B_i$, and
		%\item $\mu(B_i)=\bar{z}$ for any $i\in \mathcal{I}_z$, 
		\item for any $\mathcal{J}\subseteq \mathcal{I}$ and any connected component $C$ of $\cap_{i\in \mathcal{J}} B_i$, either $\bar{z}\subseteq \mu(C)$ or $\bar{z}\cap \mu(C)=\emptyset$.
	    %\item $\mld(X/Z\ni z,B)=\min_{i\in \mathcal{I}_z}\{a_i\}$.
	\end{itemize}   

	Let $\mathcal{S}$ be the set of connected components $C$ of $\cap_{i\in \mathcal{J}} B_i$, such that $\mu(C)=\bar{z}$ for some $\mathcal{J}\subseteq \mathcal{I}$. Since the flat locus of $\mu$ is an open subset, and any flat morphism is open, there exists an open subset $U$ of $Z$, such that the dimension of any fiber of $\mu|_{C\cap \mu^{-1}U}:C\cap \mu^{-1}U\to \bar{z}\cap U$ is equal to $\dim C-\dim {z}$ for any $C\in \mathcal{S}$.	
	\medskip
	
	We claim that $U$ has the required properties. We first show that 
    $$\mld(X/Z\ni z_{cp},B)\ge \dim{z}+\mld(X/Z\ni z,B)$$ for any closed point $z_{cp}\in \bar{z}\cap U.$
	For any point $y\in\mu^{-1}(z_{cp})$, let $\mathcal{I}(y):=\{i\in\mathcal{I}\mid y\in B_i\}$, and $C_{y}$ the unique connected component of $\cap_{i\in \mathcal{I}(y)} B_i$ containing $y$. 
	
	%$j\in \mathcal{I}_{z}$ such that $y\in B_j$,$C_{y}\subseteq B_j$, we have If either $\mu^{-1}(\bar{z})=\bigcup_{i\in\mathcal{I}_z}B_i$ or $\dim z=\dim Z$, $\mathcal{I}=\mathcal{I}_z$ and $|\mathcal{I}(y)|>0$ %, and $y\in C_{y}\cap \mu^{-1}(z_{cp})$. 
	If $\dim Z=\dim z$ and $|\mathcal{I}(y)|=0$, then $\mld(X/Z\ni z,B)\le1\le\dim X-\dim z$ and $\mld(Y\ni y,B_Y)=\dim X\ge\dim z+\mld(X/Z\ni z,B)$. Otherwise, $|\mathcal{I}(y)|>0$, and there exist $j\in \mathcal{I}_{z}$ and $C_y\in \mathcal{S}$, such that $y\in C_{y}\subseteq B_j$. Thus
	$$ \dim y\le \dim(C_{y}\cap \mu^{-1}(z_{cp}))=\dim C_{y}-\dim {z}.$$
	By Lemma \ref{lem:logsmoothmld}, we have
	\begin{align*}
	   \mld(Y\ni y,B_Y)&=\dim X-\dim y-\sum_{i\in \mathcal{I}(y)} (1-a_i)\\&\ge \dim X-\dim C_{y}+\dim\ {z}-|\mathcal{I}(y)|+\sum_{i\in \mathcal{I}(y)} a_i\\
	                        &=\dim{z}+\sum_{i\in \mathcal{I}(y)} a_i\ge \dim{z}+a_j\\&\ge \dim{z}+\mld(X/Z\ni z,B).
	\end{align*}
	Thus $\mld(X/Z\ni z_{cp},B)\ge \dim{z}+\mld(X/Z\ni z,B).$
	
	It suffices to show that
	$$\mld(X/Z\ni z_{cp},B)\le \dim{z}+\mld(X/Z\ni z,B)$$ 
	for any closed point $z_{cp}\in \bar{z}\cap U.$ 
	
	Let $y$ be the generic point of an irreducible component of $B_1\cap \mu^{-1}(z_{cp})$ of maximal dimension, i.e., $\dim y=\dim B_1-\dim{z}=\dim X-1-\dim{z}$. Since
	$$\dim X-1-\dim{z}=\dim y\le \dim C_{y}-\dim {z}\le \dim X-1-\dim{z},$$
	and the equalities hold if and only if $\mathcal{I}(y)=\{1\}$. By Lemma \ref{lem:logsmoothmld}, we have
	\begin{align*}
	 	\mld(X/Z\ni z_{cp},B)&\le\mld(Y\ni y,K_Y+B_Y)\\
	 	                                       &=\dim X-\dim y-(1-a_1)\\
	 	                                       &=\dim{z}+\mld(X/Z\ni z,B).
	\end{align*} 
\end{proof}

\begin{conj}[{\cite[Conjecture 2.4]{Ambro99}}, LSC for MLDs]\label{conj: LSC for mlds}
	For any pair $(X,B)$ and non-negative integer $d$, the function
	$$x\mapsto \mld(X\ni x,B)$$
	is lower semi-continuous for $d$-dimensional points $x\in X$.
\end{conj}
Conjecture \ref{conj: LSC for mlds} is known for $\dim X\le 3$ \cite[Proposition 2.5,Theorem 0.1]{Ambro99}.

\begin{conj}[LSC for relative MLDs]\label{conj: LSC for relative mlds}
	For any pair $(X/Z\ni z,B)$ and non-negative integer $d$, the function
	$$z\mapsto \mld(X/Z\ni z,B)$$
	is lower semi-continuous for $d$-dimensional points $z\in Z$.
\end{conj}

\begin{lem}\label{lem: limitpointrellsc}
Let $f:X\to Z$ be a proper morphism, $z_i$ a sequence of $d$-dimensional points, $x_i\in f^{-1}(z_i)$ a sequence of $d_0$-dimensional points. Suppose that $z\in \overline{\{z_i\}}$ is a $d$-dimensional point, then there exists a $d_0$-dimensional point $x\in\overline{\{x_i\}}\cap f^{-1}(z)$.
\end{lem}
\begin{proof}
We prove the lemma by induction on the dimension of $X$. By the upper semi-continuity of fiber dimensions, $\dim f^{-1}(z)\ge d_0$. Let $X_0:=\overline{\{x_i\}}$. If $X_0=X$, then for any $d_0$-dimensional point $x\in f^{-1}(z)$, $x\in \overline{\{x_i\}}$. In particular, the lemma holds when $\dim X=1$. If $\dim X_0<\dim X$, then replacing $X$ by $X_0$, we are done by the induction.
\end{proof}

\begin{prop}\label{prop:equivrelmld}
Conjecture \ref{conj: LSC for mlds} implies Conjecture \ref{conj: LSC for relative mlds}. In particular, Conjecture \ref{conj: LSC for relative mlds} holds when $\dim X\le 3$.
\end{prop}
\begin{proof} Suppose that there exists a sequence of $d$-dimensional points $z_i\in Z$, such that $z\in\overline{\{z_i\}}$, $\dim z=d$, and $\lim_{i\to+\infty}\mld(X/Z\ni z_i,B)<\mld (X/Z\ni z,B)$. 
Let $x_i\in X$, such that $\mld(X\ni x_i,B)=\mld(X/Z\ni z_i,B)$. Possibly passing to a subsequnce, we may assume that $x_i$ is a sequence of $d_0$-dimensional points. By Lemma \ref{lem: limitpointrellsc}, there exists a $d_0$-dimensional point $x\in\overline{\{x_i\}}\cap f^{-1}(z)$. By Conjecture \ref{conj: LSC for mlds}, 
\begin{align*}
\lim_{i\to+\infty}\mld(X/Z\ni z_i,B)&=\lim_{i\to+\infty}\mld(X\ni x_i ,B)\\
&\ge\mld (X\ni x,B)\ge\mld(X/Z\ni z,B),
\end{align*}
a contradiction.
\end{proof}

\begin{prop}\label{prop:dec eps com for closed point imply gene}
Problem \ref{problem: decompsable epsilon complements} for the case when $\dim z=0$ implies Problem \ref{problem: decompsable epsilon complements}.
\end{prop}

\begin{proof}
By Lemma \ref{lemma Amb99rel}, there exists an open subset $U_1$ of $Z$, such that $U_1\cap \bar{z}\neq\emptyset$ and    
$$\mld(X/Z\ni z_{cp},B)=\mld(X/Z\ni z,B)+\dim {z}$$
for any closed point $z_{cp}\in \bar{z}\cap U_1.$
Since we assume that Problem \ref{problem: decompsable epsilon complements} holds for the case when $\dim z=0$, there exist a finite set $\Ii_1$ of positive real numbers, two finite sets $\Ii_2,\mathcal{M}$ of non-negative rational numbers such that
\begin{itemize}
    \item $\sum a_i=1,$
    \item $\sum a_i\epsilon_i\ge\epsilon,$
    \item $\sum a_i(K_X+B^i)\ge K_X+B,$ and
    \item $(X/Z\ni z_{cp},B^i)$ is $(\epsilon_i+\dim z,\Rr)$-complementary for any $i,$
\end{itemize}
for some $a_i\in\Ii_1,B^i\in\Ii_2$ and $\epsilon_i\in\mathcal{M}.$

Let $(X/Z\ni z_{cp},B^i+G^i)$ be an $(\epsilon_i+\dim z,\Rr)$-complement of $(X/Z\ni z_{cp},B^i)$. By Lemma \ref{lemma Amb99rel} again, there exists an open subset $U_2$ of $Z$, such that $U_2\cap \bar{z}\neq\emptyset$ and    
$$\mld(X/Z\ni z_{cp}',B^i+G^i)=\mld(X/Z\ni z,B^i+G^i)+\dim {z}$$
    for any closed point $z_{cp}'\in \bar{z}\cap U_2$ and any $i.$

By Proposition \ref{prop:equivrelmld}, there exists a closed point $z_{cp}'\in \bar{z}\cap U_2$, such that $\mld(X/Z\ni z_{cp}',B^i+G^i)\ge \mld(X/Z\ni z_{cp},B^i+G^i)\ge \dim z+\epsilon_i$ for any $i$. Hence $\mld(X/Z\ni z,B^i+G^i)\ge \epsilon_i$ and $(X/Z\ni z,B^i)$ is $(\epsilon,\Rr)$-complementary for any $i$
\end{proof}

\begin{proof}[Proof of Theorem \ref{thm: decompsable epsilon complement for surfaces}]
Theorem \ref{thm: decompsable epsilon complement for surfaces} follows from Theorem \ref{thm:decomp comp for closed points} and Proposition \ref{prop:dec eps com for closed point imply gene}.
\end{proof}

\section{Proofs of main results}
\subsection{Diophantine approximation}

\begin{lem}\label{lem: linearindependentdiufantu}
	Let $p_0,l,c$ be positive integers, $\epsilon_1$ a positive real number, $\bm{r}_0=(r_1,\ldots,r_{c})\in\Rr^{c}$ a point, such that $r_0=1,r_1,\ldots,r_{c}$ are linearly independent over $\Qq$, and $\bm{e}=(e_1,\ldots,e_{c})\in\Rr^{c}$ a nonzero vector. Then there exist a positive integer $n_0$, and a point $\bm{r}_0'=(r_1',\ldots,r_{c}')\in \Rr^{c}$, such that 
	\begin{enumerate}
		\item $p_0|n_0$,
		\item $n_0\bm{r}'_0\in l\Zz^{c}$,
		\item $||\bm{r}_0-\bm{r}_0'||_{\infty}<\frac{\epsilon_1}{n_0},$
		 and
		\item $||\frac{\bm{r}_0-\bm{r}_0'}{||\bm{r}_0-\bm{r}_0'||_{\infty}}-\frac{\bm{e}}{||\bm{e}||_{\infty}}||_{\infty}<\epsilon_1.$
	\end{enumerate}	
\end{lem}
\begin{proof}
	Without loss of generality, we may assume that $\bm{e}\in\Rr_{\ge0}^{c}$ and $e_1=||\bm{e}||_{\infty}>0$. Let $\alpha_1=\frac{1}{2l}\epsilon_1$. Since $0\le e_i\le e_1$ for any $2\le i\le c$, we may pick positive real numbers $\alpha_i$, such that $\alpha_i<\alpha_1$, and
    $$|\frac{\alpha_i}{\alpha_1}-\frac{e_i}{e_1}|<\epsilon_1$$
    for any $2\le i\le c$.  

    By the continuity of functions $\frac{\alpha_i+x}{\alpha_1-x}$ and $\frac{\alpha_i-x}{\alpha_1+x}$, there exists a positive real number $\epsilon_2$, such that $\epsilon_2<\alpha_i<\alpha_1-2\epsilon_2$ for any $2\le i\le c$, and $$\max_{1\le i\le c}\{|\frac{\alpha_i+\epsilon_2}{\alpha_1-\epsilon_2}-\frac{e_i}{e_1}|,|\frac{\alpha_i-\epsilon_2}{\alpha_1+\epsilon_2}-\frac{e_i}{e_1}|\}<\epsilon_1.$$
    Let $\bm{\alpha}=(\alpha_1,\ldots,\alpha_{c})$. By Kronecker's theorem or Weyl's equidistribution theorem, there exist a positive integer $n_1$, and a point $\bm{\beta}=(\beta_1,\ldots,\beta_{c})\in \Zz^{c}$, such that
    $$||\frac{n_1p_0\bm{r}_0}{l}-\bm{\beta}-\bm{\alpha}||_{\infty}<\epsilon_2.$$
    We will show that $n_0:=n_1p_0$ and $\bm{r}_0':=\frac{l}{n_0}\bm{\beta}$ have the required properties. 

    It is clear that $p_0|n_0$, $n_0\bm{r}_0'\in l\Zz^{c}$, and
        $$0<\frac{l}{n_0}(\alpha_i-\epsilon_2)<r_i-r_i'=\frac{l}{n_0}(\frac{n_1p_0r_i}{l}-\beta_i)<\frac{l}{n_0}(\alpha_i+\epsilon_2)<\frac{\epsilon_1}{n_0}$$
    for any $1\le i\le c$. In particular, $r_i-r_i'<\frac{l}{n_0}(\alpha_i+\epsilon_2)<\frac{l}{n_0}(\alpha_1-\epsilon_2)<r_1-r_1'$ for any $2\le i\le c.$
    Hence $||\bm{r}_0-\bm{r}_0'||_{\infty}=r_1-r_1'<\frac{\epsilon_1}{n_0}$, and
    \begin{align*}
    ||\frac{\bm{r}_0-\bm{r}_0'}{||\bm{r}_0-\bm{r}_0'||_{\infty}}-\frac{\bm{e}}{||\bm{e}||_{\infty}}||_{\infty}&=\max_{1\le i\le c}\{|\frac{r_i-r_i'}{ r_1-r_1'}-\frac{e_i}{e_1}|\}\\&
    \le \max_{1\le i\le c}\{|\frac{\alpha_i-\epsilon_2}{\alpha_1+\epsilon_2}-\frac{e_i}{e_1}|,|\frac{\alpha_i+\epsilon_2}{\alpha_1-\epsilon_2}-\frac{e_i}{e_1}|\}<\epsilon_1.
    \end{align*}
    %where we use the fact that $\alpha_i-\epsilon_2>0$ for the first equality.
\end{proof}

\begin{lem}\label{lemma rational direction2}
     Let $\epsilon$ be a non-negative real number, $n_0$ a positive integer, $\Ii\subseteq[0,1],\Ii_1\subseteq(0,1],\Ii_2\subseteq[0,1]\cap\Qq,\mathcal{M}\subseteq\Qq_{\ge0}$ finite sets such that $n_0\Ii_2\subseteq\Zz$. Then there exists a positive integer $n$ depending only on $\epsilon,n_0,\Ii,\Ii_1,\Ii_2$ and $\mathcal{M}$ satisfying the following.

     Assume that $a_i\in \Ii_1,b_{ij}\in \Ii_2$ and $\epsilon_i\in\mathcal{M}$ $(1\le i\le k,1\le j\le s)$ such that 
     \begin{itemize}
        \item $\sum_{i=1}^k a_i=1$, 
        \item $\sum_{i=1}^k a_i\epsilon_i\ge\epsilon$, and
        \item $\sum_{i=1}^ka_ib_{ij}\in\Ii$ for any $1\le j\le s$.
     \end{itemize}
     Then there exists a point $\bm{a}'=(a_1',\ldots,a_k')\in\Rr_{>0}^k$ such that     
     \begin{enumerate}
        \item $n_0|n$,
        \item $\sum_{i=1}^k a_i'=1$,
        \item $n\bm{a}'\in n_0\Zz^k,$
        \item $\sum_{i=1}^k a_i'\epsilon_i\ge\epsilon,$ and 
        \item $n\sum_{i=1}^k a_i'b_{ij}= n\lfloor \sum_{i=1}^k a_ib_{ij} \rfloor+\lfloor (n+1)\{\sum_{i=1}^k a_ib_{ij}\}\rfloor$ for any $1\le j\le s$.
     \end{enumerate} 
     
     Moreover, if there exist real numbers $r_0=1,\dots,r_c$ which are linearly independent over $\Qq$ such that $\Ii\cup\{\epsilon\}\subseteq\Span_{\Qq_{\ge0}}(\{r_0,\dots,r_c\}),$ then we can additionally require that $\sum_{i=1}^k a_i'b_{ij}\ge \sum_{i=1}^k a_ib_{ij}$ for any $1\le j\le s.$
\end{lem}

\begin{proof}    
    There exist positive integers $c,l,M$, a point $\bm{r}_0=(r_1,\dots,r_{c})\in\Rr^{c}$, and $\Qq$-linear functions $a_i(\bm{r}):\Rr^{c}\to \Rr$ depending only on $\Ii_1$ such that $r_0=1,r_1,\dots,r_{c}$ are linearly independent over $\Qq,\Ii_1\subseteq\Span_{\Qq}(\{r_0,\dots,r_c\})$, $a_i(\bm{r}_0)=a_i,$ $la_i(\bm{r})$ is a $\Zz$-linear function and 
    $$|a_i(\bm{r})-a_i(\bm{0})|\le M||\bm{r}||_{\infty}$$ 
    for any $i,$ where $\bm{0}=(0,\dots,0)\in\Rr^{c}.$
    
    Since $\epsilon(\bm{r}):=\sum_{i=1}^k\epsilon_ia_i(\bm{r})$ is a $\Qq$-linear function, there exist a point $\bm{e}\in\Rr^{c}$ and a positive real number $\epsilon'$ such that $\epsilon(\bm{r})\ge \epsilon(\bm{r}_0)\ge\epsilon$ for any $\bm{r}\in\Rr^{c}$ satisfying
    $||\frac{\bm{r}_0-\bm{r}}{||\bm{r}_0-\bm{r}||_\infty} -\frac{\bm{e}}{||\bm{e}||_\infty}||_\infty<\epsilon'.$  Moreover, if $\Ii\cup\{\epsilon\}\subseteq\Span_{\Qq_{\ge0}}(\{r_0,\dots,r_c\}),$ then we can additionally require that
     $b_j(\bm{r})\ge b_j(\bm{r}_0)=\sum_{i=1}^ka_ib_{ij}$ 
     for any $1\le j\le s$ and $\bm{r}\in\Rr^c$ satisfying the same conditions, where $b_j(\bm{r}):=\sum_{i=1}^kb_{ij}a_i(\bm{r})$.
    
    Let $\epsilon''$ be a positive real number such that 
    $$\epsilon''<\min_{\gamma_1\in\Ii,\gamma_2\in\Ii_1}\{\gamma_1>0,1-\gamma_1>0,\gamma_2,1\},$$
    and $\epsilon'''$ a positive real number such that $\epsilon'''<\min\{\frac{\epsilon''^2}{M},\epsilon'\}.$ By Lemma \ref{lem: linearindependentdiufantu}, there exist an integer $n$ and a point $\bm{r}_0'\in\Rr^{c}$, such that
    \begin{itemize}
       \item $ln_0|n$,
       \item $n\bm{r}_0'\in ln_0\Zz^{c}$,
       \item $||\bm{r}_0-\bm{r}_0'||_{\infty}<\frac{\epsilon'''}{n}$, and
       \item $||\frac{\bm{r}_0-\bm{r}_0'}{||\bm{r}_0-\bm{r}_0'||_\infty}
             -\frac{\bm{e}}{||\bm{e}||_\infty}||_\infty<\epsilon'''.$
    \end{itemize}
    
    Let $a_i'=a_i(\bm{r}_0')$ for any $1\le i\le k$ and $\bm{a}'=(a_1',\dots,a_k')$. Since $\sum_{i=1}^k a_i=1$ and $1,r_1,\dots,r_{c}$ are linearly independent over $\Qq,\sum_{i=1}^k a_i(\bm{r})=1$ for any $\bm{r}\in\Rr^{c}$. In particular, $\sum_{i=1}^k a_i'=1.$ 
    Moreover, we have $ na_i'=n_0\frac{n}{ln_0}la_i(\bm{r}_0')\in n_0\Zz$ for any $1\le i\le k,$ and 
    $$||\bm{a}-\bm{a}'||_\infty\le M||\bm{r}_0-\bm{r}'_0||_\infty< M\frac{\epsilon'''}{n}<\frac{\epsilon''^2}{n},$$
    where $\bm{a}:=(a_1,\dots,a_k)$. In particular, $a_i'$ are positive real numbers, since 
    $$a_i'\ge a_i-|a_i-a_i'|\ge a_i-||\bm{a}-\bm{a}'||_\infty>a_i-\frac{\epsilon''^2}{n}>0.$$ 
    We have $\sum_{i=1}^k a_i'\epsilon=\epsilon(\bm{r}_0')\ge \epsilon$, since
    $||\frac{\bm{r}_0-\bm{r}_0'}{||\bm{r}_0-\bm{r}_0'||_\infty}-\frac{\bm{e}}{||\bm{e}||_\infty}||_\infty<\epsilon'''<\epsilon'.$

    It suffices to show $(5)$. If $\sum_{i=1}^k a_ib_{ij}=1$ for some $j$, then $b_{ij}=1$ for any $1\le i\le k$ as $\sum_{i=1}^k a_i=1$ and $1\ge b_{ij}\ge0.$ Thus $\sum_{i=1}^k a_i'b_{ij}=1$.   
    Hence we may assume that $1>\sum_{i=1}^k a_ib_{ij}>0$ and $\sum_{i=1}^k a_i'b_{ij}<1$. Since $n\sum_{i=1}^k a_i'b_{ij}=\sum_{i=1}^k \frac{n}{n_0}a_i'\cdot(n_0b_{ij})\in\Zz$, we only need to show that
    $$n\sum_{i=1}^k a_i'b_{ij}+1>(n+1)\sum_{i=1}^k a_ib_{ij}\ge n\sum_{i=1}^k a_i'b_{ij}.$$    
    The above inequalities hold since $\sum_{i=1}^k a_ib_{ij}\in\Ii$, $k\epsilon''<\sum_{i=1}^k a_i=1$, and
    $$ n\sum_{i=1}^k |a_i'-a_i|b_{ij}< nk\cdot\frac{\epsilon''^2}{n}<\epsilon''<\min\{\sum_{i=1}^k a_ib_{ij},1-\sum_{i=1}^k a_ib_{ij}\}.$$
     
     If $\Ii\cup\{\epsilon\}\subseteq\Span_{\Qq_{\ge0}}(\{r_0,\dots,r_c\}),$ then by our choice of $\bm{a}'$, $\sum_{i=1}^ka_i'b_{ij}=b_j(\bm{r}_0')\ge b_j(\bm{r}_0)=\sum_{i=1}^ka_ib_{ij}$ for any $1\le j\le s$.
\end{proof}

\begin{lem}\label{lem: convexset n vectors}
	Let $\mathcal{D}$ be a compact convex set in $\Rr^n$, then there exist $n+1$ points $\bm{v}_1,\ldots,\bm{v}_{n+1}$ in $\Rr^n$, such that $\mathcal{D}$ is contained in the interior of convex hull of $\bm{v}_1,\ldots,\bm{v}_{n+1}$.
	
	Moreover, there exists a positive real number $\epsilon$, such that for any $\bm{v}_1',\ldots,\bm{v}_{n+1}'\in\Rr^n$, if $||\bm{v}_i-\bm{v}_i'||<\epsilon$ for any $1\le i\le n+1$, then $\mathcal{D}$ in contained in the convex hull of $\bm{v}_1',\ldots,\bm{v}_{n+1}'$.
\end{lem}
\begin{proof}
	Since $\mathcal{D}$ is a bounded set, there exists a positive real number $M$, such that  
	$$\mathcal{D}\subseteq \{(x_1,\ldots,x_n)\in\Rr^n\mid n\sum_{i=1}^n|x_i|<M\}.$$
	Let $\bm{v}_{n+1}:=(-3M,\ldots,-3M)$ and $\bm{v}_i:=(0,\ldots,3M,\ldots,0)$ for $1\le i\le n$, where $3M$ is the $i$-th coordinate.
	
	For any point $\bm{b}=(b_1,\ldots,b_n)\in\mathcal{D}$, there exists a positive real number $0<a<\frac{1}{3n}$, such that $3aM+b_i\ge 0$ for any $1\le i\le n$. We have 
	$$\bm{b}=\sum_{i=1}^n\frac{3aM+b_i}{3M}\bm{v}_{i}+a\bm{v}_{n+1}+(1-a-\sum_{i=1}^n\frac{3aM+b_i}{3M})(\frac{1}{n+1}\sum _{i=1}^{n+1} \bm{v}_i).$$
	Since 
	$$a+\sum_{i=1}^n\frac{3aM+b_i}{3M}< \frac{1}{3}+\frac{2M}{3M}=1,$$
	$\mathcal{D}$ is contained in the interior of convex hull $\mathcal{V}$ of $\bm{v}_1,\ldots,\bm{v}_{n+1}$. Let $d=dist(\partial\mathcal{V},\mathcal{D})>0$. Since $d$ is a continuous function of $\bm{v}_1,\ldots,\bm{v}_{n+1}$, there exists a positive real number $\epsilon$, such that for any  $||\bm{v}_i-\bm{v}_i'||<\epsilon$, $d'=dist(\partial\mathcal{V}',\mathcal{D})>0$, where $\mathcal{V}'$ is the convex hull of $ \bm{v}_1',\ldots,\bm{v}_{n+1}'$. In particular, $\mathcal{D}$ in contained in the convex hull of $\bm{v}_1',\ldots,\bm{v}_{n+1}'$.
\end{proof}
\begin{lem}\label{lem:finitesetmonotonic}
Let $\Ii$ be a set of non-negative real numbers. Then the following are equivalent.
\begin{enumerate}
		\item For any finite set $\Ii_0$ of $\Ii$, there exist real numbers $r_0=1,r_1,\ldots,r_c$ which are linearly independent over $\Qq$, such that $$\Ii_0\subseteq \Span_{\Qq_{\ge0}}(\{r_0,r_1,\ldots,r_c\}),$$
	\item $\Span_{\Qq_{\ge0}}(\Ii\backslash\Qq)\cap (\Qq\backslash\{0\})=\emptyset$.
\end{enumerate}
Moreover, if $\Gamma$ satisfies one of the above two conditions and $\hat{\Gamma}\subseteq[0,1]\cap (\Qq+\Span_{\Qq_{\ge0}}(\Ii))$ for some finite set $\hat{\Gamma}$, then $\hat{\Gamma}$ also satisfies one of the above two conditions.
\end{lem}
\begin{proof}
Suppose that $\Ii$ satisfies (1). For any $a\in\Span_{\Qq_{\ge0}}(\Ii\backslash\Qq)\cap(\Qq\backslash\{0\})$, there exist positive rational numbers $\lambda_i$, and $\alpha_i\in\Ii\backslash\Qq$, such that $a=\sum_{i=1}^m\lambda_i\alpha_i$. By (1), there exist positive real numbers $r_0=1,r_1,\ldots,r_c$ which are linearly independent over $\Qq$, such that $\alpha_i\in \Span_{\Qq_{\ge0}}(\{r_0,r_1,\ldots,r_c\})$ for any $1\le i\le m$. Since $a\in\Qq$, $\alpha_i\in \Span_{\Qq_{\ge0}}(\{r_0,r_1,\ldots,r_c\})\cap (\Ii\backslash\Qq)$ for any $1\le i\le m$, which implies that $a=0$, a contradiction.

Suppose that $\Ii$ satisfies (2). For any finite set $\Ii_0=\{\alpha_1,\ldots,\alpha_m\}$ of $\Ii$, there exist positive real numbers $1,r_1',\ldots,r_c'$ which are linearly independent over $\Qq$, such that $\Ii_0\subseteq \Span_{\Qq}(\{1,r_1',\ldots,r_c'\})$. Let $\bm{r}'=(r_1',\ldots,r_c')$. Then there exist $\bm{a}_1',\ldots,\bm{a}_m'\in\Qq^c$, such that $\alpha_i-{\bm{a}_i}'\cdot{\bm{r}}'\in\Qq$ for any $1\le i\le m$. Possibly reordering the indices, we may assume that $\alpha_1,\ldots,\alpha_{m'}\notin\Qq$ and $\alpha_{m'+1},\ldots,\alpha_{m}\in\Qq$  for some $0\le m'\le m$. Then $\bm{a}'_1,\ldots,\bm{a}'_{m'}\neq \bm{0}$ and $\bm{a}'_{m'+1}=\ldots=\bm{a}'_m=\bm{0}$, where $\bm{0}=(0,0,\ldots,0)$. 

If $m'=0$, then $c=0$, $\Ii_0$ is a finite set of rational numbers, and $r_0=1,(1)$ holds. Thus we may assume $m'\ge 1$. 

Suppose that $\bm{0}$ belongs to the convex hull of $\{\bm{a}'_1,\ldots,\bm{a}'_{m'}\}$. There exist non-negative rational numbers $\lambda_i$, such that $\sum_{i=1}^{m'}\lambda_i=1$, and $\sum_{i=1}^{m'} \lambda_i\bm{a}_i'=\bm{0}$. Then $\sum_{i=1}^{m'}\lambda_i\bm{a}_i'\cdot \bm{r}'=\bm{0}\cdot\bm{r}'=0$ and $0\neq\sum_{i=1}^{m'}\lambda_i\alpha_i=\sum_{i=1}^{m'}\lambda_i(\alpha_i-{\bm{a}_i}'\cdot{\bm{r}}')\in\Qq$, which contradicts  $\Span_{\Qq_{\ge0}}(\Ii\backslash\Qq)\cap (\Qq\backslash\{0\})=\emptyset$. Hence $\bm{0}$ does not belong to the convex hull of $\{\bm{a}'_1,\ldots,\bm{a}'_{m'}\}$.

By Hahn-Banach theorem, there exist a rational point $\bm{t}\in\Qq^{c}$ and a rational number $b>0$, such that the hyperplane $H:=\{\bm{x}\in\Rr^{c}\mid \bm{t}\cdot\bm{x}=b\}$ intersects the segment $\overrightarrow{\bm{0}\bm{a}'_i}$ for any $1\le i\le m'$. Let $\mathcal{C}$ be the convex cone generated by $\overrightarrow{\bm{0}\bm{a}'_1},\ldots,\overrightarrow{\bm{0}\bm{a}'_{m'}}$. Then $\mathcal{C}\cap H$ is a compact convex set in $H$. By Lemma \ref{lem: convexset n vectors}, there exist rational points $\bm{v}_1,\ldots,\bm{v}_{c}\in H$, such that $\mathcal{C}\cap H$ is contained in the convex hull of $\bm{v}_1,\ldots,\bm{v}_{c}$. Hence $\mathcal{C}$ is contained in the cone generated by $\overrightarrow{\bm{0}\bm{v}_1},\ldots,\overrightarrow{\bm{0}\bm{v}_{c}}$. 
Since $\bm{v}_j$, $\bm{a}_i'$ are rational points, there exist non-negative rational numbers $a_{ij}$, such that $\bm{a}_i'=\sum_{j=1}^{c} a_{ij}\bm{v}_j$ for any $1\le i\le m'$. 

Let $r_j=\bm{v}_j\cdot \bm{r}'$ and $\bm{r}=(r_1,\ldots,r_c)$.
Then 
$$\alpha_i-{\bm{a}_i}'\cdot{\bm{r}}'=\alpha_i-\sum_{j=1}^{c} a_{ij}\bm{v}_j\cdot\bm{r}'=\alpha_i-\sum_{j=1}^{c}a_{ij}r_j\in\Qq$$
for any $1\le i\le m'$. Possibly replacing $(r_1,\dots,r_c)$ by $(r_1-t,\ldots,r_c-t)$ for some rational numbers $t\gg 0$, we may assume that $\alpha_i-\sum_{j=1}^{c}a_{ij}r_j\ge 0$ for any $1\le i\le m'$. Therefore $\alpha_i\in\Span_{\Qq_{\ge0}}(\{r_0,\ldots,r_c\})$ for any $1\le i\le m$, and we are done.

Finally, if $\Gamma$ satisfies (1) and $\hat{\Gamma}\subseteq[0,1]\cap (\Qq+\Span_{\Qq_{\ge0}}(\Ii))$ for some finite set $\hat{\Gamma}$, then there exist a finite set $\Gamma_0\subseteq \Ii$, and real numbers  $r_0'=1,r_1',\ldots,r_c'$ which are linearly independent over $\Qq$, such that 
\begin{align*}
\hat{\Gamma}&\subseteq[0,1]\cap (\Qq+\Span_{\Qq_{\ge0}}(\Ii_0))\\
&\subseteq[0,1]\cap(\Qq+\Span_{\Qq_{\ge0}}(\{r_0',\ldots,r_c'\}).
\end{align*}
Since $\hat{\Gamma}$ is a finite set, there exists a rational real numbers $t\gg 0$, such that 
$$\hat{\Gamma}\subseteq  \Span_{\Qq_{\ge0}}(\{r_0,\ldots,r_c\}),$$
where $r_i:=r_i'-t$ for any $0\le i\le c$.
\end{proof}

\subsection{Complements for surface germs}
\begin{lem}\label{lem: indexofbpf}
Let $\epsilon$ be a non-negative real number, $M$ and $I$ two positive integers. Then there exists a positive integer $n_0$ depending only on $\epsilon,I$ and $M$ satisfying the following.
     
     Assume that $(X/Z\ni z,B)$ is a surface germ such that
     \begin{itemize}
        \item $(X/Z\ni z,B)$ is $(\epsilon,\Rr)$-complementary,
        \item $-I(K_X+B)$ is Cartier and semiample over $Z$, and
         \item either $\epsilon=0$ or the multiplicity of any fiber of any minimal elliptic fibration of the minimal resolution of $X$ over $Z$ is bounded from above by $M$.
     \end{itemize}
     Then $(X/Z\ni z,B)$ has a monotonic $(\epsilon,n_0)$-complement.
\end{lem}

\begin{proof}

Possibly shrinking $Z$ near $z$, we may assume that $(X,B)$ is lc. Let $f:Y\to X$ be the minimal resolution, and we may write $K_Y+B_Y:=f^*(K_X+B)$. Note that $(Y/Z\ni z,B_Y)$ is $(\epsilon,\Rr)$-complementary and $-I(K_Y+B_Y)$ is Cartier and semiample over $Z$. Possibly replacing $(X,B)$ by $(Y,B_Y)$, we may assume that $X$ is smooth. Let $m$ be the smallest positive integer, such that $-m(K_X+B)$ is base point free over $Z$. Let $g:X\to W$ be the morphism defined by $|-m(K_X+B)|$, possibly replacing $g$ by its Stein factorization, we may assume that $g$ is a contraction. In the following, we either find a positive integer $m'$ depending only on $I$ and $M$ such that $-m'(K_X+B)$ is base point free over $Z$ and hence $(X/Z\ni z,B)$ has a monotonic $(\epsilon,m')$-complement by Lemma \ref{lem semiamplecomplement} or construct a monotonic complement directly.
\medskip

\noindent (1)
If $\dim W=2$, then $K_X+B=g^{*}(K_W+B_W)$, where $B_W$ is the strict transform of $B$. Note that $\mathcal{P}(W/Z\ni z,B_W)\subseteq(\frac{1}{I}\Zz_{\ge0}\cap[0,1])\cup\{+\infty\}$ is a finite set. By Lemma \ref{prop pld}, the Cartier index of $K_W+B_W$ is bounded from above by $I_1$. Since $-(K_W+B_W)$ is ample over $Z$, by \cite[Theorem 2.2.4]{Fujino09}, there exists $m_1$ which only depends on $I_1$, such that $-m_1(K_W+B_W)$ is base point free over $Z$. Hence $-m_1(K_X+B)$ is base point free over $Z$.
\medskip

\noindent (2)
If $\dim W=1$, then by the canonical bundle formula \cite{Kodaria60,Kodaira63}, either $W=Z$ or $\dim Z=0$ and $W$ is a curve of genus $\le 1$.
Note that any general fiber of $X\to W$ is $\mathbb{P}^1$ or an elliptic curve. In the former case, we may run a $K_X$-MMP over $W$ and terminates with a Mori fiber space $\tilde{X}\to W$. Since the MMP is $(K_X+B)$-trivial, by the cone theorem,  $I(K_{\tilde{X}}+\tilde{B})$ is Cartier, where $\tilde{B}$ is the strict transform of $B$. By \cite[Theorem 2.2.4]{Fujino09}, there exists a positive integer $m_2$ which only depends on $I$ such that $-m_2(K_{\tilde{X}}+\tilde{B})$ is base point free over $W$. Hence $-3m_2(K_X+B)$ is base point free over $Z$ (c.f. \cite[IV.3.2]{GTM52}). 

 In the latter case, we may contract $(-1)$ curves, and get a relatively minimal elliptic fibration $\tilde{X}\to W$. Replacing $X$ by $\tilde{X}$, we may assume that $g:X\to W$ is a minimal elliptic fibration over $Z$. The canonical bundle formula implies that
 $$K_X+B\sim g^{*}(K_W+L+\sum_{Q_i\in\mathcal{Q}}\frac{m_{Q_i}-1+b_i}{m_{Q_i}}Q_i),$$
where $b_i\in\frac{1}{I}\Zz_{\ge0}\cap[0,1]$, $L$ is a nef Cartier divisor on $W$ and $\mathcal{Q}$ is the set of closed points $Q$ such that either $g^*Q$ is a multiple fiber with multiplicity $m_Q>1$ or $\Supp g^*Q\subseteq\Supp B$. If $m_{Q}\le IM$ for any closed point $Q\in W$, then there exists a positive integer $m_2':=I(IM)!$, such that $-m_2'(K_X+B)$ is base point free over $W$. Hence $-3m_2'(K_X+B)$ is base point free over $Z$. If $\epsilon=0,$ then we may assume that $X\to W$ has a multiple fiber with multiplicity $m_{Q_1}>I$. Let 
$G:=\sum_{Q_i\in\mathcal{Q},m_{Q_i}>1}\frac{1-b_i}{m_{Q_i}}g^*Q_i$. By the classification of minimal elliptic fibrations, $(X,B+G)$ is lc. We will show that $(X/Z\ni z,B+G)$ is a monotonic $I$-complement of $(X/Z\ni z,B)$.
If $W=Z$, we have 
$$I(K_X+B+G)\sim Ig^{*}(K_W+L+\sum_{Q_i\in\mathcal{Q},m_{Q_i}>1} Q_i+
\sum_{Q_i\in\mathcal{Q},m_{Q_i}=1}b_iQ_i).$$
Suppose that $\dim Z=0$, then $W$ is $\Pp^1$. By \cite[Proposition 1.1]{Fujimoto90}, $L$ is a Cartier divisor of degree 1, and $X\to W$ has only one multiple fiber. Since $m_{Q_1}>I$ and $(X,B)$ is $\Rr$-complementary, $$\deg(K_W+L+\sum_{Q_i\in\mathcal{Q}}\frac{m_{Q_i}-1+b_i}{m_{Q_i}}Q_i)=-1+\frac{m_{Q_1}-1+b_1}{m_{Q_1}}+\sum_{Q_i\in\mathcal{Q},m_{Q_i}=1} b_i\le 0,$$
and $\sum_{m_{Q_i}=1,Q_i\in\mathcal{Q}} b_i$=0. Hence
$$K_X+ B+G\sim g^*(K_W+L+Q_1)\sim 0.$$

\medskip

\noindent (3)
If $\dim W=0$, then there exists a positive integer $m_3$ which only depends on $I$ such that $m_3(K_X+B)\sim 0$ (c.f. \cite[Corollary 1.11]{PS09}). %In particular, $(X/Z\ni z,B)$ has a monotonic $(\epsilon,m_3)$-complement.
\medskip

Let $n_0:=3m_1m_2m_2'm_3I$, and we are done.
\end{proof}

\begin{thm}\label{thm: epsilonlccompclosedpoint}
     Theorem \ref{thm epsilonlccompforsurfaces} holds for the case when $\dim z=0.$
\end{thm}

\begin{proof}
     According to Theorem \ref{thm reducetofiniteandfixpld}, we may assume that $\Ii$ is a finite set, $-(K_X+B)$ is semiample over a neighborhood of $z,X$ is smooth and either $\epsilon=0$ or the multiplicity of any fiber of any minimal elliptic fibration of $X$ over $Z$ is bounded from above by $M$. Possibly shrinking $Z$ near $z$, we may assume that $-(K_X+B)$ is semiample over $Z$.
     
     By Lemma \ref{lemma antinefpolytope2}, there exist a finite set $\Ii_1\subseteq(0,1]$, two finite sets $\Ii_2\subseteq[0,1]$ and $\mathcal{M}$ of non-negative rational numbers and a positive integer $I$ depending only on $\Ii$ and $\epsilon$, such that 
     \begin{itemize}
       %\item $\sum a_i=1,$
       \item $\sum a_i\epsilon_i     \ge\epsilon$,
       \item  $K_X+B=\sum a_i(K_X+B^i),$
       \item $(X/Z\ni z,B^i)$ is $(\epsilon_i,\Rr)$-complementary for any $i$, and
       \item $-I(K_X+B^i)$ is Cartier and semiample over $Z$ for any $i$,
       
       %\item If $\epsilon\in\Qq$, then  $\epsilon=\epsilon_i$ for any $i$, and if $\epsilon\neq 1$, then $\epsilon_i\neq1$ for any $i$.
     \end{itemize}
     for some $a_i\in\Ii_1,B^i\in\Ii_2$ and $\epsilon_i\in\mathcal{M}$. By Lemma \ref{lem: indexofbpf}, there exists a positive integer $n_0$ which only depends on $\epsilon_i,I$ and $M$ such that $({X}/Z\ni z,{B}^i)$ has an $(\epsilon_i,n_0)$-complement $({X}/Z\ni z,{B}^i+{G}^i)$ for some $\Qq$-Cartier divisor ${G}^i\ge0$ and any $i$.
     Let ${G}:=\sum a_i{G}^i$.
     
     By Lemma \ref{lemma rational direction2}, there exists a positive integer $n$ depending only on $\epsilon,p,n_0,\Ii$, $\Ii_1,\Ii_2,\mathcal{M}$ such that there exist positive rational numbers $a_i'$ with the following properties:
     \begin{itemize}
         \item $pn_0|n$,
         \item $\sum a_i'=1$,
         \item $\sum a_i'\epsilon_i\ge\epsilon$,
         \item $na_i'\in n_0\Zz$ for any $i$, and
         \item $nB'\ge n\lfloor B\rfloor+\lfloor (n+1)\{B\}\rfloor$, where $B':=\sum a_i'{B}^i$.
     \end{itemize}
     Let ${G}':=\sum a_i'{G}^i$, then
     $$
     n(K_{{X}}+B'+G')=n\sum a_i'(K_{X}+{B}^i+{G}^i)
     =\sum \frac{a_i'n}{n_0}\cdot n_0(K_{X}+{B}^i+{G}^i)\sim_Z0
     $$
     and
     $$a(E,X,B'+G')=\sum a_i'(E,X,B^i+G^i)\ge a_i'\epsilon_i\ge\epsilon$$
     for any $E\in e(z)$.
     Hence $(X/Z\ni z,B'+G')$ is an $(\epsilon,n)$-complement of $(X/Z\ni z,B)$.

     Moreover, if $\Span_{\Qq_{\ge0}}(\bar{\Ii}\cup\{\epsilon\}\backslash\Qq)\cap (\Qq\backslash\{0\})=\emptyset$, then ${B}'\ge B$ by Lemma \ref{lemma rational direction2} and Lemma \ref{lem:finitesetmonotonic}.
\end{proof}

\subsection{Proofs of Theorem \ref{thm epsilonlccompforsurfaces}, Corollary \ref{cor main} and Theorem \ref{thm: complementdim1}}

\begin{prop}\label{Prop: a complement conj closed point imply general}Let $n$ be a positive integer, $\epsilon$ a non-negative real number and $(X/Z\ni z,B)$ an $(\epsilon,\Rr)$-complementary pair. Suppose that for any closed point $z_{cp}\in\bar{z}$, if $(X/Z\ni z_{cp},B)$ is $(\epsilon+\dim z,\Rr)$-complementary, then $(X/Z\ni z_{cp},B)$ has an $(\epsilon+\dim z,n)$-complement $(X/Z\ni z_{cp},B_{z_{cp}}^{+})$. Then $(X/Z\ni z,B)$ has an $(\epsilon,n)$-complement $(X/Z\ni z,B^{+})$ for some $B^{+}=B_{z_{cp}}^{+}$. 

In particular, Conjecture \ref{conj a complement} for the case when $\dim z=0$ implies Conjecture \ref{conj a complement}.
\end{prop}
\begin{proof}
Possibly shrinking $Z$ near $z$, we may assume that there exists an $\Rr$-Cartier divisor $G\ge0$, such that $K_X+B+G\sim_{\Rr,Z}0$ and $(X/Z\ni z,B+G)$ is $\epsilon$-lc over $z$. By Lemma \ref{lemma Amb99rel}, there exists an open subset $U_1$ of $Z$, such that $U_1\cap \bar{z}\neq\emptyset$, $$\mld(X/Z\ni z_{cp},B+G)=\mld(X/Z\ni z,B+G)+\dim {z},$$
and $(X/Z\ni z_{cp},B)$ is $(\epsilon+\dim z,\Rr)$-complementary for any closed point $z_{cp}\in \bar{z}\cap U_1.$

By assumption, there exists an $(\epsilon+\dim {z},n)$-complement $(X/Z\ni z_{cp},B^{+}_{z_{cp}})$ of $(X/Z\ni z_{cp},B)$. By Lemma \ref{lemma Amb99rel} again, there exists an open subset $U_2$ of $Z$, such that $U_2\cap \bar{z}\neq\emptyset$ and    
$$\mld(X/Z\ni z_{cp}',B^{+}_{z_{cp}})=\mld(X/Z\ni z,B^{+}_{z_{cp}})+\dim {z}$$
for any closed point $z_{cp}'\in \bar{z}\cap U_2.$

By Proposition \ref{prop:equivrelmld}, there exists a closed point $z_{cp}'\in \bar{z}\cap U_2$, such that $\mld(X/Z\ni z_{cp}',B^{+}_{z_{cp}})\ge \mld(X/Z\ni z_{cp},B^{+}_{z_{cp}})\ge \dim z+\epsilon$. Hence $\mld(X/Z\ni z,B^{+}_{z_{cp}})\ge \epsilon$ and $(X/Z\ni z,B^{+}_{z_{cp}})$ is an $(\epsilon,n)$-complement of $(X/Z\ni z,B)$.
\end{proof}

\begin{proof}[Proof of Theorem \ref{thm epsilonlccompforsurfaces}] 
Theorem \ref{thm epsilonlccompforsurfaces} follows from Theorem \ref{thm: epsilonlccompclosedpoint} and Proposition \ref{Prop: a complement conj closed point imply general}.
\end{proof}

\begin{ex}
     Let $b\in(0,\frac{1}{2})$ be an irrational number, $\Ii:=\{0,b\}$, $\epsilon_1:=b$ and $\epsilon_2:=1-b.$ We have $\Span_{\Qq_{\ge0}}(\bar{\Ii}\cup\{\epsilon_1\}\backslash\Qq)\cap (\Qq\backslash\{0\})=\emptyset$, and $1\in \Span_{\Qq_{\ge0}}(\bar{\Ii}\cup\{\epsilon_2\}\backslash\Qq)\cap (\Qq\backslash\{0\})$. By Theorem \ref{thm epsilonlccompforsurfaces}, $(\Pp^2,b\Pp^1)$ has
     a monotonic $(\epsilon_1,n_1)$-complement, and $(\Pp^2,b\Pp^1)$ is $(\epsilon_2,n_2)$-complementary
     for some positive integers $n_1,n_2$. Since $\mld(\Pp^2,b\Pp^1)=\epsilon_2$, $(\Pp^2,b\Pp^1)$ does not have any monotonic $(\epsilon_2,n)$-complement for any positive integer $n$.
\end{ex}

\begin{ex}
     Let $b\in(0,1)$ be an irrational number and $\epsilon:=\frac{1-b}{2}$. Let $(X\ni x,B:=bB_1)$ be a surface germ where $B_1$ is a prime divisor, $f:Y\to X$ the minimal resolution of $X$ with exceptional divisors $E_1,\dots,E_{m+1}$ such that the dual graph of $f$ is a chain, $E_i$ intersects $E_{i+1}$ for any $1\le i\le m$, $E_1\cdot E_1=-3,B_{Y,1}\cdot E_{1}=1$ and $E_i\cdot E_i=-2,B_{Y,1}\cdot E_i=0$ for any $2\le i\le m+1$, where $B_{Y,1}$ is the strict transform of $B_1$ on $Y$. By Lemma \ref{lemma 3.3}, we have
     $$\pld(X\ni x,tB_1)=\frac{(1-t)(m+1)+1}{2m+3}$$
     for any $0<t<1$. In particular,
     $$\pld(X\ni x,B)=\frac{(1-b)(m+1)+1}{2m+3}>\epsilon.$$
     
     By Lemma \ref{lemma rational direction2}, there exist a positive integer $n$ and a positive real number $b'<b$ such that $2|n,nb'\in2\Zz$ and $nb'=\lfloor (n+1)b \rfloor$. Let $b'':=\frac{1+b'}{2}$, then $nb''\in\Zz$ and $1-b''=\frac{1-b'}{2}$. Let $D$ be a prime divisor on $X$ and $D_{Y}$ the strict transform of $D$ on $Y$ such that $D_Y\cdot E_{m+1}=1$ and $D_Y\cdot E_i=0$ for any $1\le i\le m$. Let $B^+:=b'B_1+b''D$, we have
     $$a(E_i,X,B^+)=\frac{1-b'}{2}>\epsilon$$
     for any $1\le i\le m+1$. 
     %In particular, $\pld(X\ni x,B^+)=\frac{1-b'}{2}>\epsilon.$
     When $m>2\lfloor\frac{1}{\epsilon}\rfloor$, $\mld(X\ni x,B^+)=\pld(X\ni x,B^+)>\epsilon$ by Lemma \ref{lemma pld=mld}. Thus $(X\ni x,B^+)$ is an $(\epsilon,n)$-complement of $(X\ni x,B)$.
\end{ex}

\begin{proof}[Proof of Corollary \ref{cor main}]
     Suppose on the contrary that there exist a sequence of surface pairs $(X_i/Z_i\ni z_i,B_{(i)}:=\sum_{j=1}^sb_{ij}B_{ij})$ satisfying the conditions and a strictly increasing sequence of positive integers $n_i,$ where 
     $$n_i:=\min\{n\mid (X_i/Z_i\ni z_i,B_{(i)})\text{ is }(\epsilon,n)\text{-complementary},p|n\}.$$ 
     Since $b_{ij}\in[0,1],$ possibly passing to a subsequence, we may assume that $\lim_{i\to\infty}b_{ij}=b_j$ for some real number $b_j\in[0,1]$ and either $\{b_{ij}\}_{i=1}^{\infty}$ is a decreasing sequence or an increasing sequence for any $1\le j\le s.$ 
     
     Let $b_{ij}':=\min\{b_{ij},b_j\}$ for any $i,j$, and $B_{(i)}':=\sum_{j=1}^sb_{ij}'B_{ij}$ for any $i.$ Since the set $\{b_{ij}'\}_{i,j}$ satisfies the DCC and $(X_i/Z_i\ni z_i, B_{(i)}')$ is $(\epsilon,\Rr)$-complementary, by Theorem \ref{thm epsilonlccompforsurfaces}, there exists a positive integer $p|n'$ depending only on $p,M,\epsilon$ and $\{b_{ij}'\}_{i,j}$ such that $(X_i/Z_i\ni z_i,B_{(i)}')$ is $(\epsilon,n')$-complementary for any $i$. There exists a positive integer $N_0$ such that $b_{ij}-b_j<\frac{1-\{(n'+1)b_j\}}{n'+1}$ for any $i\ge N_0$ and any $1\le j\le s$ as $\lim_{i\to \infty}b_{ij}=b_j$ for any $1\le j\le s$. Since either $b_{ij}'=b_{ij}$ or $1\ge b_{ij}>b_j=b_{ij}'$ and 
     $(n'+1)b_{ij}'<(n'+1)b_{ij}<1+\lfloor (n'+1)b_{ij}'\rfloor$, we have
     $$\lfloor(n'+1)\{b_{ij}\}\rfloor=\lfloor(n'+1)\{b_{ij}'\}\rfloor,\, n'\lfloor b_{ij}\rfloor=n'\lfloor b_{ij}'\rfloor$$ for any $1\le j\le s$ and any $i\ge N_0$, and $(X_i/Z_i\ni z_i,B_{(i)})$ is $(\epsilon,n')$-complementary for any $i\ge N_0$, a contradiction.
\end{proof}

\begin{proof}[Proof of Theorem \ref{thm: complementdim1}] 
It suffices to show the case when $\dim Z=\dim z=0$. When
$X$ is an elliptic curve, then $B=0$, $K_X\sim 0$, and we may let $n=p$. Thus we may assume that $X=\Pp^1$ and $\epsilon<1$. Let $B:=\sum_{k=1}^s b_kB_k$, where $B_k$ are distinct closed points. We show that possibly replacing $b_iB_i+b_jB_j$ by $bB_i:=(b_i+b_j)B_i$, we may assume that $b_i+b_j> 1-\epsilon$ for any $i\neq j.$ In particular, we have $\frac{2-\frac{2}{s}}{s-1}+\frac{2}{s}> 1-\epsilon$ and $s\le\lfloor \frac{4}{1-\epsilon}\rfloor$. Indeed, suppose that $b_i+b_j\le1-\epsilon$ for some $i\neq j$. If $(X,bB_i+\sum_{k\neq i,j}b_kB_k)$ has an $(\epsilon,n)$-complement $(X,b^{+}B_i+\sum_{k\neq i,j}b_k^+B_k)$ for some positive integer $n$. Then
$$nb^{+}\ge \lfloor(n+1)(b_i+b_j)\rfloor\ge \lfloor(n+1)b_i\rfloor+ \lfloor(n+1)b_j\rfloor.$$ 
Let $b_i^+:=\frac{\lfloor(n+1)b_i\rfloor}{n}$ and $b_j^+:=\frac{nb^{+}-\lfloor(n+1)b_i\rfloor}{n}$. Then $(X,\sum_{k=1}^sb_k^+B_k)$ is an $(\epsilon,n)$-complement of $(X,B)$.

Let $C$ be an elliptic curve. By Corollary \ref{cor main}, there exists a positive integer $N$ depending only on $\epsilon$ and $p$, such that $(\Pp^1\times C,B\times C)$ has an $(\epsilon,n)$-complement $(\Pp^1\times C,B_{C}^{+})$ for some positive integer $n$ satisfying $n\le N$ and $p|n$. Restricting to a general fiber $\Pp^1$, $(\Pp^1,B^{+}:=B_{C}^{+}|_{\Pp^1})$ is an $(\epsilon,n)$-complement of $(\Pp^1,B)$.
\end{proof}
\section{Applications of Conjecture \ref{conj a complement} and open problems}

\subsection{Applications of Conjecture \ref{conj a complement}} In this subsection, we will show that Conjecture \ref{conj a complement} implies Birkar-Borisov-Alexeev-Borisov Theorem, the global index conjecture of Calabi-Yau varieties, Shokurov's index conjecture and Shokurov-M\textsuperscript{c}Kernan conjecture. %Theses applications also provide positive evidences for Conjecture \ref{conj a complement}.

Recall that Birkar-Borisov-Alexeev-Borisov Theorem \cite{Bir16a, Bir16b} states that $\epsilon$-lc Fano varieties of dimension $d$ form a bounded family for positive real number $\epsilon$ and positive integer $d$. %, and the conjecture is recently proved by Birkar \cite{Bir16a, Bir16b}.

\begin{prop}\label{Prop: comp conj implies bbab}
	Conjecture \ref{conj a complement} in dimension $d$ implies Birkar-Borisov-Alexeev-Borisov Theorem in dimension $d$.
\end{prop}

\begin{proof}
   Let $X$ be a projective variety of dimension $d$ such that $(X,B)$ is $\epsilon$-lc for some big $\Rr$-divisor $B\ge0$ and positive real number $\epsilon$, and $K_X+B\sim_{\Rr}0$. Let $\tilde{X}\to X$ be a small $\Qq$-factorization. Since $(\tilde{X},0)$ is $(\epsilon,\Rr)$-complementary, by Conjecture \ref{conj a complement}, there exists a positive integer $n$ which only depends on $\epsilon$ and $d$, such that $(\tilde{X},0)$ is $(\epsilon,n)$-complementary. Thus $(X,B^{+})$ is $\epsilon$-lc and $n(K_X+B^{+})\sim 0$ for some boundary $B^{+}$. Now the boundedness of $X$ follows from \cite[Theorem 1.3]{HX15}.
\end{proof}
It is also expected that Conjecture \ref{conj a complement} should imply a conjecture due to M\textsuperscript{c}Kernan and Prokhorov \cite[Conjecture 3.9]{MP04}, which is a generalization of Birkar-Borisov-Alexeev-Borisov Theorem.

\medskip

Let $\Ii\subseteq[0,1]\cap \Qq$ be a finite set, $T$ the set of log Calabi-Yau varieties of dimension $d$,
$$T:=\{(X,B)\mid \dim X=d, B\in\Ii, (X,B) \text{ is lc, }K_X+B\equiv 0\}.$$
The global index conjecture of Calabi-Yau varieties predicts that there is an integer $n$ which only depends on $\Ii$ and $d$, such that $n(K_X+B)\sim 0$ for any $(X,B)\in T$. 

If the global index conjecture of Calabi-Yau varieties holds, it is possible to show the birational boundedness of Calabi-Yau varieties. When $\dim X=3$, it is illustrated in \cite{CDHJS18}. Recently, by assuming the global index conjecture, Birkar, Di Cerbo and Savaldi announced that they can show the birational boundedness of Calabi-Yau varieties when $X$ is projective klt, there is an elliptic fibration $X\to Y$ admitting a rational section, and $Y$ is rationally connected. 

\begin{prop}\label{prop:global index}
	Conjecture \ref{conj a complement} in dimension $d$ implies the global index conjecture of Calabi-Yau varieties in dimension $d$. In particular, the global index conjecture of Calabi-Yau varieties holds in dimension 2.
\end{prop}

\begin{proof}
	Let $\epsilon=0$, $Z$ a closed point, and $p$ a integer such that $p\Ii\subseteq\Zz$. By Conjecture \ref{conj a complement}, there exists a positive integer $p|n$, such that $n(K_X+B^{+})\sim 0$ for some lc pair $(X,B^{+})$, where $B^{+}\ge B$. Since $K_X+B\equiv0$, we conclude that $B^{+}=B$ and $n(K_X+B)\sim 0$. The last claim follows from Theorem \ref{thm epsilonlccompforsurfaces}.
\end{proof}

\medskip

The following index conjecture is due to Shokurov.

\begin{conj}[Shokurov's index conjecture]\label{conj:index}
	Let $d$ be a positive integer, $\epsilon$ a non-negative real number and $\Ii\subseteq[0,1]\cap \Qq$ a finite set. Then there exists a positive integer $I$ depending only on $\epsilon$ and $\Ii$ satisfying the following.
	
	Assume that $(X\ni x,B)$ is a pair such that
	\begin{enumerate}
	    \item $\dim X=d,$
	    \item $B\in\Ii$, and
	    \item $\mld(X\ni x,B)=\epsilon.$
	\end{enumerate}
	Then $I(K_X+B)$ is Cartier near $x$.
\end{conj}

\begin{prop}\label{prop: epsiloncomp conj implies index}
	Conjecture \ref{conj a complement} in dimension $d$ implies Conjecture \ref{conj:index} in dimension $d$. In particular, Conjecture \ref{conj:index} holds in dimension 2.
\end{prop}

\begin{proof}
	Let $p$ be a positive integer such that $p\Ii\subseteq\Zz$. By Conjecture \ref{conj a complement}, there exists a positive integer $p|n$, such that $n(K_X+B^{+})$ is Cartier near $x$, and $\mld(X\ni x,B^{+})\ge \epsilon$, where $nB^{+}\ge n\lfloor B\rfloor+\lfloor (n+1)\{B\}\rfloor\ge nB.$ Since $\mld(X\ni x,B)=\epsilon$, we conclude that $B^{+}=B$ near $x$, and we may choose $I=n$. The last claim follows from Theorem \ref{thm epsilonlccompforsurfaces}.
\end{proof}

The following conjecture on Fano type fibrations has a close relation with Conjecture \ref{conj a complement}. Conjecture \ref{conj:ShoMc} is equivalent to \cite[Conjecture 1.2]{Bir16crelle} and \cite[Conjecture 1.2]{AlexeevBorisov14}, but in a slightly different form.

\begin{conj}[Shokurov-M\textsuperscript{c}Kernan]\label{conj:ShoMc}
	Let $d$ be a positive integer and $\epsilon$ a positive real number. Then there exists a positive real number $\delta$ depending only on $d$ and $\epsilon$ satisfying the following. 
	
	Assume that $(X,B)$ is a pair of dimension $d$ and $\pi:X\to Z$ is a contraction such that 
	\begin{enumerate}
	   \item $(X,B)$ is $\epsilon$-lc,
	   \item $X$ is of Fano type over $Z$, and
	   \item $K_X+B\sim_{\Rr}0$ over $Z$.	
	   %\item $K_X+B\sim_{\Rr,Z} 0$,	
	   %\item the general fiber of $f$ are of Fano type.
	\end{enumerate}
	Then $B_Z\le 1-\delta$, where $B_Z$ is the discriminant part of the canonical bundle formula of $K_X+B$ on $Z$. %	we can choose $M_Z\ge0$ representing the moduli part such that $(Z,B_Z+M_Z)$ is $\delta$-lc. 
\end{conj}

\begin{prop}\label{prop: epsilon comp conj implies MS conj}
Conjecture \ref{conj a complement} in dimension $d$ implies Conjecture \ref{conj:ShoMc} in dimension $d$.	
\end{prop}

\begin{proof}
Possibly replacing $X$ by its small $\Qq$-factorization, we may assume that $X$ is $\Qq$-factorial. Let $z\in Z$ be a codimension one point. Possibly shrinking $Z$ near $z$, we may assume that $z$ is smooth. Fix a positive real number $0<a<\epsilon$, let $t:=\alct(X/Z\ni z,B;\pi^{*}z)$. 

Let $E\in e(z)$, such that $a(E,X,B+t\pi^{*}z)=a$. Let $g:Y\to X$ be a morphism which extracts $E$, and we may write $K_Y+B_Y:=g^*(K_X+B)$. Note that if $E$ is on $X$, then $g=\id$. Since $(Y/Z\ni z,0)$ is $(a,\Rr)$-complementary, by Conjecture \ref{conj a complement}, there exists a positive integer $n$ which only depends on $a,d$, such that $(Y/Z\ni z,0)$ has an $(a,n)$-complement $(Y/Z\ni z,B_{a})$ for some boundary $B_{a}$. By Birkar-Borisov-Alexeev-Borisov Theorem \cite{Bir16a, Bir16b}, $(F,B_{a}|_{F})$ is log bounded, where $F$ is a general fiber of $Y\to Z$. By \cite[Proposition 3.1]{Bir16crelle}, there exists a positive real number $\delta_{a}$, such that $t_a:=\lct(Y/Z\ni z,B_{a};g^{*}(\pi^{*}z))\ge \delta_{a}$. Thus $\delta_{a} m_iF_i\le F_i$ for any irreducible component $F_i$ of $g^{*}(\pi^{*}z)$, where $g^{*}(\pi^{*}z):=\sum m_iF_i$. Since $(X,B)$ is $\epsilon$-lc, by the construction of $Y$, 
$$1-a=tm_i+\mult_{F_i}B_Y\le tm_i+1-\epsilon$$
for some $i$. Therefore,
$$\lct(X/Z\ni z,B;\pi^{*}z)\ge t\ge \frac{\epsilon-a}{m_i}\ge (\epsilon-a)\delta_{a},$$
and $\delta:=(\epsilon-a)\delta_{a}$ has the required properties.
\end{proof}

\subsection{Open problems} In this subsection, we ask several questions which are related to our paper. Some questions arise from previous joint works of the second author with his collaborators.  

\medskip

Our first question is about the existence of $(\epsilon,\Rr)$-decomposed complements, that is whether Theorem \ref{thm: decompsable epsilon complement for surfaces} holds for higher dimensional varieties or not.

\begin{problem}[Existence of $(\epsilon,\Rr)$-decomposed complements]\label{problem: decompsable epsilon complements}
     Let $d$ be a positive integer, $\epsilon$ a non-negative real number and $\Ii\subseteq[0,1]$ a DCC set. Then there exist finite sets $\Ii_1\subseteq(0,1],\Ii_2\subseteq[0,1]\cap\Qq$ and $\mathcal{M}\subseteq\Qq_{\ge0}$ depending only on $d,\Ii$ and $\epsilon$ satisfying the following.
     
     Assume that $(X/Z\ni z,B)$ is a pair such that
     \begin{itemize}
       \item $\dim X=d$,
       \item $B\in\Ii$, and
       %\item $z\in Z$ is a closed point, 
       \item $(X/Z\ni z,B)$ is $(\epsilon,\Rr)$-complementary.
     \end{itemize}
     Then there exist $a_i\in\Ii_1$, $B^i\in\Ii_2$ and $\epsilon_i\in\mathcal{M}$ with the following properties:
     \begin{enumerate}
       \item $\sum a_i=1,$
       \item $\sum a_i\epsilon_i\ge\epsilon$,
       \item $\sum a_i(K_X+B^i)\ge K_X+B$,
       \item $(X/Z\ni z,B^i)$ is $(\epsilon_i,\Rr)$-complementary for any $i$, and
       \item If $\epsilon\in\Qq$, then $\epsilon=\epsilon_i$ for any $i$, and if $\epsilon\neq 1$, then $\epsilon_i\neq1$ for any $i$.
     \end{enumerate}
In particular, there exists $B':=\sum a_iB^{i}\ge B$, such that $(X/Z\ni z,B')$ is $(\epsilon,\Rr)$-complementary, and the coefficients of $B'$ belong to a finite set $\Ii'$ depending only on $d,\Ii$ and $\epsilon$.
\end{problem}

When $\epsilon=0$ and $X$ is of Fano type over $Z$, Problem \ref{problem: decompsable epsilon complements} was proved by Han-Liu-Shokurov \cite[Theorem 1.13]{HLS19}. 

When $\epsilon=0$ and $\dim Z=0$, Problem \ref{problem: decompsable epsilon complements} implies the Global ACC \cite[Theorem 1.5]{HMX14} directly. We will show that Problem \ref{problem: decompsable epsilon complements} implies the ACC for ($\epsilon,\Rr$)-complementary thresholds and the ACC for a-log canonical thresholds (Conjecture \ref{conj: ACC for aLCTs}). Thus Problem \ref{problem: decompsable epsilon complements} could be regarded as a relative version of the ACC properties of singularities. 

\begin{defn}[$(\epsilon,\Rr)$-complementary thresholds]\label{defn (epsilon,R)-complementary thresholds} 
     Suppose that $(X/Z\ni z, B)$ is $(\epsilon,\Rr)$-complementary for some non-negative real number $\epsilon$. The \emph{$(\epsilon,\Rr)$-complementary threshold} of $(X/Z\ni z,B)$ with respect to an $\Rr$-Cartier divisor $M\ge0$ is defined as
     \begin{align*}
     \eRcomp(X/Z\ni z,B;M):=\sup\{&t\ge0\mid (X/Z\ni z,B+tM)\\
     &\text{ is $(\epsilon,\Rr)$-complementary}\}.
     \end{align*}
     In particular, if $X=Z$, $\pi:X\to Z$ is the identity map and $\epsilon=a$, then we obtain the \emph{$a$-lc threshold} (Definition \ref{defn alct}).
\end{defn}

\begin{problem}[ACC for ($\epsilon,\Rr$)-complementary thresholds]\label{problem: ACC for (epsilon,R)-complementary thresholds}
       Fix a positive integer $d$ and a non-negative real number $\epsilon$. Let $\Ii_1$ and $\Ii_2$ be two DCC sets, then 
     \begin{align*}
     \epsilon\text{-}\operatorname{RCOMP}(d,\Ii_1,\Ii_2)&=\{\eRcomp(X/Z\ni z,B;M)\mid (X/Z\ni z,B)\\
     &\text{ is $(\epsilon,\Rr)$-complementary}, \dim X=d,B\in\Ii_1,M\in\Ii_2\}
     \end{align*}
     satisfies the ACC. 
\end{problem}

\begin{prop}\label{prop: decomposed epsilon comp imply acc for alct}
Problem \ref{problem: decompsable epsilon complements} in dimension $d$ implies the ACC for $(\epsilon,\Rr)$-complementary thresholds (Problem \ref{problem: ACC for (epsilon,R)-complementary thresholds}) and the ACC for a-log canonical thresholds (Conjecture \ref{conj: ACC for aLCTs}) in dimension $d$. In particular, Problem \ref{problem: ACC for (epsilon,R)-complementary thresholds} holds in dimension $2$.
\end{prop}
\begin{proof}
Since Conjecture \ref{conj: ACC for aLCTs} is a speical case of Problem \ref{problem: ACC for (epsilon,R)-complementary thresholds}, it suffices to show Problem \ref{problem: ACC for (epsilon,R)-complementary thresholds}.

Suppose on the contrary, there exist two DCC sets $\Ii_1,\Ii_2$ and a sequence pairs $(X_i/Z_i\ni z_i,B_i)$ of dimension $d$ which is $(\epsilon,\Rr)$-complementary, and $\Rr$-Cartier divisors $M_i$, such that 
$B_i\in \Ii_1,M_i\in\Ii_2$, and $t_i:=\eRcomp(X_i/Z_i\ni z_i,B_i;M_i)$ is strictly increasing. In particular, the coefficients of $B_i+t_iM_i$ belong to a DCC set $\Ii$. Since we assume that Problem \ref{problem: decompsable epsilon complements} is true, there exists a divisor $B'_i\ge B_i+t_iM_i$, such that the coefficients of $B_i'$ belong to a finite set $\Ii'$, and $(X_i/Z_i\ni z_i,B_i')$ is $(\epsilon,\Rr)$-complementary. By the construction of $t_i$, there exists a prime divisor $D_i$, such that $\Supp D_i\subseteq\Supp M_i$, and $\mult_{D_i} (B_i+t_iM_i)=\mult_{D_i}B_i'\in \Ii'$ for any $i$. Thus $t_i$ belongs to an ACC set, a contradiction.

The last claim follows from Theorem \ref{thm: decompsable epsilon complement for surfaces}.
\end{proof}

It is expected that Problem \ref{problem: ACC for (epsilon,R)-complementary thresholds} will imply the ACC for minimal log discrepancies. For recent works towards the relation between the ACC for a-log canonical thresholds and the ACC for minimal log discrepancies, we refer readers to \cite{Kawatika18,LJH18}.

Recall that the ACC for log canonical thresholds polytopes holds by \cite{HLQ17}. As a generalization, we may ask whether
$(\epsilon,\Rr)$-complementary thresholds polytopes are indeed polytopes or not and the ACC for $(\epsilon,\Rr)$-complementary thresholds polytopes holds or not, where the $(\epsilon,\Rr)$-complementary thresholds polytope $P(X/Z\ni z, \Delta;D_1,\ldots,D_s)$ of $D_1,\ldots,D_s$ with respect to $(X/Z\ni z,\Delta)$ is defined by
\begin{align*}
P(X, \Delta; &D_1,\ldots,D_s)\coloneqq\{(t_1,\ldots,t_s) \in \mathbb{R}_{\geq 0}^s\mid \\
&(X/Z\ni z,\Delta+t_1 D_1+\ldots+t_s D_s)\text{ is $(\epsilon,\Rr)$-complementary}\}.
\end{align*}
A more ambitious conjecture predicts that the set of volumes of $P(X/Z\ni z, \Delta;D_1,\ldots,D_s)$ should satisfy the ACC. This conjecture is unknown even in dimension 1.

\medskip

As we mentioned in the introduction, it is expected that Conjecture \ref{conj a complement} should hold for surface pairs when $\Gamma$ is not necessarily a DCC set. 

\begin{problem}\label{problem: [0,1]surface complement}    
     Let $\epsilon$ be a non-negative real number and $p$ a positive integer. Then there exists a positive integer $N$ depending only on $\epsilon$ and $p$ satisfying the following.

     Assume that $(X/Z\ni z,B)$ is a surface pair such that
     \begin{enumerate}     
     \item  either $\epsilon=0$ or $-K_X$ is big over $Z$, and
       \item $(X/Z\ni z,B)$ is $(\epsilon,\Rr)$-complementary.
      \end{enumerate}
      Then $(X/Z\ni z,B)$ is $(\epsilon,n)$-complementary for some positive integer $n$ satisfying $n\le N$ and $p|n$.
\end{problem} 

 For a fixed point $\bm{v}=(v_1^{0},\ldots,v_m^{0})\in \Rr^m$ and a positive real number $\epsilon$, Theorem \ref{thm:uniformpolytopeforsurfacemlds} says that for any surface germ $(X\ni x,\sum_{i=1}^m v^0_iB_i)$ which is $\epsilon$-lc at $x$, $\mld(X\ni x,\sum_{i=1}^m v_iB_i)$ is a linear function in a uniform rational polytope containing $\bm{v}$. It would be interesting to ask if Theorem \ref{thm:uniformpolytopeforsurfacemlds} holds in higher dimensions.

\begin{problem}\label{conj: linearmld}
	Let $\epsilon$ be a positive real number, $d,m$ two positive integers and $\bm{v}=(v_1^0,\ldots,v_m^0)\in\Rr^m$ a point. Then there exist a rational polytope $\bm{v}\in P\subseteq\Rr^m$ with vertices $\bm{v}_j=(v_1^j,\ldots,v_m^j)$, positive real numbers $a_j$ and positive real numbers $\epsilon_j$ depending only on $d,\epsilon,m$ and $\bm{v}$ satisfying the following. 
	\begin{enumerate}
	 \item $\sum a_j=1,\sum a_j\bm{v}_j=\bm{v}$ and $\sum a_j\epsilon_j\ge\epsilon$.
	 \item Assume that $(X\ni x,B:=\sum_{i=1}^m v_i^0B_{i})$ is a pair such that
	 \begin{itemize}
	     \item $\dim X=d$,
	     \item $B_1,\ldots,B_m\ge0$ are Weil divisors on $X$, and
	     \item $(X\ni x,B)$ is $\epsilon$-lc at $x$.
	 \end{itemize}
	 Then the function $P\to \Rr$ defined by
   	 $$(v_1,\ldots,v_m)\mapsto \mld(X\ni x,\sum_{i=1}^m v_iB_{i})$$
  	 is a linear function, and
  	 $$\mld(X\ni x,\sum_{i=1}^m v_i^jB_{i})\ge \epsilon_j$$
   	 for any $j$.	
	\end{enumerate}
	Moreover, if $\epsilon\in \Qq$, then we may pick $\epsilon_j=\epsilon$ for any $j$. 
\end{problem}

Problem \ref{conj: linearmld} implies the following. 
\begin{problem}\label{conj: epsilonlcrationalpolytope}
	Let $\epsilon$ be a rational non-negative real number, $d,m$ two positive integers and $\bm{v}=(v_1^0,\ldots,v_m^0)\in\Rr^m$ a point. Then there exist a uniform rational polytope $P\subseteq\Rr^m$ containing $\bm{v}$ depending only on $\epsilon,d,m$ and $\bm{v}$ satisfying the following. 
	
	Assume that $(X\ni x,B:=\sum_{i=1}^m v_i^0B_{i})$ is a pair such that
	\begin{enumerate}
	    \item $\dim X=d$,
	     \item $B_1,\ldots,B_m\ge0$ are Weil divisors on $X$, and
	     \item $(X\ni x,B)$ is $\epsilon$-lc at $x$.
	\end{enumerate}
	 Then $(X\ni x,\sum_{i=1}^m v_iB_{i})$ is $\epsilon$-lc at $x$ for any point $(v_1,\ldots,v_m)\in P$.	
\end{problem}
When $\epsilon=0$, Problem \ref{conj: epsilonlcrationalpolytope} was proved by Nakamura \cite[Theorem 1.6]{nakamura2016} and Han-Liu-Shokurov \cite[Theorem 5.8]{HLS19}. This result is related to the ACC for log canonical thresholds and its accumulation points. It is expected that Problem \ref{conj: epsilonlcrationalpolytope} is related to the ACC for a-log canonical thresholds and the ACC for minimal log discrepancies. 

   More generally, it seems that if the ACC properties of singularities holds for some invariant, then there are uniform rational polytopes for this invariant. Problem \ref{conj: linearmld} could be regarded as the question about the correspondence between the ACC for minimal log discrepancies and uniform rational polytopes for minimal log discrepancies under this general principle. We may also ask whether pseudo-effective thresholds, anti-pseudo-effective thresholds, $(\epsilon,\Rr)$-complementary thresholds satisfy this principle or not. 
   
   \medskip

    In the following, we will give an affirmative answer to Problem \ref{conj: linearmld} on a fixed $X/Z\ni z$ such that $(X/Z\ni z,\Delta)$ is klt over a neighborhood of $z$ for some boundary $\Delta$, and we do not assume that $\epsilon$ is positive.

\begin{lem}\label{lemma kltpairhasbddci}
     Suppose that $(X,B)$ is a klt pair. Then there is a positive integer $I$ depending only on $(X,B)$ such that if $D$ is a $\Qq$-Cartier Weil divisor on $X,$ then $ID$ is Cartier.
\end{lem}

\begin{proof}
     Let $(X', B')$ be a small $\Qq$-factorization of $(X, B).$ Let $D'$ be the pullback of $D$. There is a $\Qq$-divisor $L'\ge 0$ on $X'$ which is anti-ample over $X.$ Possibly rescaling $L'$, we may assume $(X', B'+L')$ is klt. In particular, $X' \to X$ is a $(K_{X'}+B'+L')$-negative contraction of an extremal face of the Mori-Kleiman cone of $X'.$ By the cone theorem, the Cartier index of $D'$ and $D$ are the same. Replacing $(X,B)$ by $(X',B')$, we may assume that $X$ is $\Qq$-factorial and $B=0.$

     Let $f: W\to X$ be a resolution of $X,$ and we may write
     $$K_{W}+(1-\epsilon)E=f^{*} K_{X}+F,$$
     where $F\ge 0,E$ is the sum of the $f$-exceptional divisors and $0<\epsilon<\mld(X,0)$ is a rational number. Since $X$ is klt, $(W, (1-\epsilon)E)$ is a klt pair. Let $M:=K_{W}+(1-\epsilon)E .$ By \cite{BCHM10}, we may run an $M$-MMP over $X$ with scaling of an ample divisor,
     $$W_{0}:=W\dashrightarrow W_{1}\dashrightarrow \cdots\dashrightarrow W_{k}$$
     and the MMP terminates with the minimal model $W_{k}=X$ on which $M_{k}$ is nef, where $M_{k}$ is the strict transform of $M$ on $W_{k}$. Let $l_j$ be a positive integer such that $l_jM_{j}$ is Cartier for any $1\le j\le k$. 

     We show by induction on $j$ that there exists a positive integer $I_{j},$ such that $I_{j} G_{j}$ is Cartier for any Weil divisor $G_{j}$ on $W_{j},$ and we get $I=I_{k}$ as desired. When $j=0,$ the claim is trivial, since $W_{0}$ is smooth. Suppose that the claim holds for $\le j .$ Let $G_{j+1}$ be any Weil divisor on $W_{j+1}, G_{j}$ the strict transform of $G_{j+1}$ on $W_{j}.$ 

     Let $W_{j}\to Z_{j}$ be the $M_j$-negative contraction in the MMP, $R_j$ the corresponding extremal ray. (If $W_{j} \to W_{j+1}$ is a divisorial contraction, then $Z_{j}=W_{j+1}$, otherwise, it is a flipping contraction.) If $G_{j}$ is numerically trivial over $Z_{j},$ then by the cone theorem, $I_{j} G_{j+1}$ is Cartier, and we are done. We may assume that $G_{j}$ is not numerically trivial over $Z_{j}.$ If $G_{j}$ is ample over $Z_{j}$, let $I_{j}'=I_{j},$ otherwise, let $I_{j}'=-I_{j}.$ Let $H_{j}$ be the pullback of an ample Cartier divisor on $Z_{j}$. Possibly replacing $H_{j}$ by some multiple, we may assume that $H_{j}+I_{j}'G_{j}$ is an ample Cartier divisor. For any positive integer $n$, let
     $$r_j(n):=\max\{t\in\Rr\mid (n+1)H_{j}+I'_{j}G_{j}+\frac{t}{m_j}M_{j}\text{ is nef}\},$$
     where $m_{j}=(l_{j}(\dim X+1))!.$ By the rationality theorem, $r_j(n)$ is a positive integer. Since $H_{j}$ is nef, $r_j(n)$ is a non-decreasing function of $n$. Since $r_j(n)\le m_{j}\frac{I'_{j}(G_{j}\cdot R_j)}{-(M_{j}\cdot R_j)}$, there exist two positive integers $r_j,n_j$, such that $r_j(n)=r_j$ for any $n>n_{j}$. Let 
     $$D_j=(n+1)H_{j}+I'_{j}G_{j}+\frac{r_j}{m_{j}}M_{j},$$
     for some $n>\max\{\frac{2r_j\dim X}{m_{j}},n_j\}.$ By the length of extremal rays and the rationality theorem, $D_j\cdot R_j=0$. By the cone theorem, $I_j(m_j(n+1)H_{j+1}+m_jI_{j}'G_{j+1}+r_jM_{j+1})$ is Cartier on $W_{j+1},$ and hence $m_jl_{j+1}I_{j}^2 G_{j+1}$ is Cartier. Let $I_{j+1}=m_jl_{j+1}I_{j}^2,$ we finish the induction.
\end{proof}

\begin{thm}
Let $\epsilon$ be a non-negative real number, $m$ a positive integer and $\bm{v}=(v_1^0,\ldots,v_m^0)\in\Rr^m$ a point. Let $(X/Z\ni z,\Delta)$ be a pair which is klt over a neighborhood of $z$. Then there exist a rational polytope $\bm{v}\in P\subseteq\Rr^m$ with vertices $\bm{v}_j=(v_1^j,\ldots,v_m^j)$, positive real numbers $a_j$ and positive rational numbers $\epsilon_j$ depending only on $\epsilon,m,\bm{v}$ and $(X/Z\ni z,\Delta)$ satisfying the following.  
	\begin{enumerate}
	   \item $\sum a_j=1,\sum a_j\bm{v}_j=\bm{v}$ and $\sum a_j\epsilon_j\ge\epsilon$.
	   \item Assume that $(X/Z\ni z,\sum_{i=1}^m v_i^0B_{i})$ is $\epsilon$-lc over $z$, where $B_1,\ldots,B_m\ge0$ are Weil divisors. Then the function $P\to \Rr$ defined by
  	   $$(v_1,\ldots,v_m)\mapsto\mld(X/Z\ni z,\sum_{i=1}^m v_iB_{i})$$
	   is a linear function, and
	   $$\mld(X/Z\ni z,\sum_{i=1}^m v_i^jB_{i})\ge \epsilon_j$$
	   for any $j$.	
	\end{enumerate}
	Moreover, if $\epsilon\in \Qq$, then we may pick $\epsilon_j=\epsilon$ for any $j$.
\end{thm}
The proof is very similar to the proof of Theorem \ref{thm:uniformpolytopeforsurfacemlds}.
\begin{proof}
      Possibly shrinking $Z$ near $z$ and replacing $X$ by a small $\Qq$-factorization of $(X,\Delta)$, we may assume that $X$ is $\Qq$-factorial and klt. In particular, $M_0:=\mld(X/Z\ni z,0)>0$ and $\mld(X/Z\ni z,C)\le M_0$ for any boundary $C$. By Lemma \ref{lemma kltpairhasbddci}, there exists a positive integer $I_0$ such that $I_0D$ is Cartier for any Weil divisor $D$ on $X$.
      
      There exist $\Qq$-linearly independent real numbers $r_0=1,r_1,\ldots,r_c$ for some $0\le c\le m$, and $\Qq$-linear functions $s_i(\bm{r}):\Rr^{c}\to \Rr$ such that $s_i(\bm{r}_0)=v_i^0$ for any $1\le i\le m$, where $\bm{r}_0=(r_1,\dots,r_c)$. Note that the map $\Rr^c\to V$ defined by
     $$\bm{r}\mapsto (s_1(\bm{r}),\dots,s_m(\bm{r}))$$    
     is one-to-one, where $V\subseteq\Rr^m$ is the rational envelope of $\bm{v}$.
      
      If $c=0$, $P=V=\{\bm{v}\}$ and there is nothing to prove. Suppose that $c\ge1$. Let $B(\bm{r}):=\sum_{i=1}^m s_i(\bm{r})B_i$, then $B(\bm{r}_0)=\sum_{i=1}^m v_i^0B_{i}.$ By assumption, $(X/Z\ni z,B(\bm{r}_0))$ is $\epsilon$-lc over $z$. By Lemma \ref{lemma bddcimldlinear}, there exist a positive real number $\delta$ and a $\Qq$-linear function $f(\bm{r})$ depending only on $\epsilon,I_0,M_0,c,m,\bm{r}_0,$ $s_i(\bm{r})$ such that $f(\bm{r}_0)\ge\epsilon$, and
      $$\mld(X/Z\ni z,B(\bm{r}))=a(E,X,B(\bm{r}))\ge f(\bm{r})$$
      for some prime divisor $E\in e(z)$ and any $\bm{r}\in\Rr^c$ satisfying $||\bm{r}-\bm{r}_0||_\infty\le\delta$. Note that if $\epsilon\in\Qq$, then $f(\bm{r})=\epsilon$ for any $\bm{r}\in\Rr^c$.
      We may find $2^c$ positive rational numbers $r_{i,1},r_{i,2}$ such that $r_{i,1}<r_i<r_{i,2}$ and $\max\{r_i-r_{i,1},r_{i,2}-r_i\}\le\delta$ for any $1\le i\le c$. By our choice of $\delta$, the function $\Rr^c\to \Rr$ defined by
      $$\bm{r}\mapsto \mld(X/Z\ni z,B(\bm{r}))$$
      is a linear function on $\bm{r}\in U_c:=[r_{1,1},r_{1,2}]\times\cdots\times[r_{c,1},r_{c,2}]$. 
      
      Let $\bm{r}_j$ be the vertices of $U_c,$ and $\epsilon_j:=f(\bm{r}_j),v_i^j:=s_i(\bm{r}_j)$ for any $i,j$. Note that if $\epsilon\in\Qq$, then $\epsilon_j=f(\bm{r}_j)=\epsilon$ for any $j$.
      Let $P:=\{(s_1(\bm{r}),\dots,s_m(\bm{r}))\mid \bm{r}\in U_c\}\subseteq V$. Then the function $P\to \Rr$ defined in the theorem is a linear function, $(v_1^j,\dots,v_m^j)$ are vertices of $P,$ and
      $$\mld(X/Z\ni z,\sum_{i=1}^m v_i^jB_i)=\mld(X/Z\ni z,B(\bm{r}_j))\ge f(\bm{r}_j)\ge\epsilon_j$$ 
      for any $j$. 
      
      Finally, we may find positive real numbers $a_j$ such that $\sum a_j=1$ and $\sum a_j\bm{r}_j=\bm{r}_0$. Then $\sum a_j\bm{v}_j=\bm{v}$ and $\sum a_j\epsilon_j\ge\epsilon$ as $\sum_j a_j v_i^j=\sum_j a_js_i(\bm{r}_j)=s_i(\sum_j a_j\bm{r}_j)=s_i(\bm{r}_0)=v_i^0$ for any $1\le i\le m,$ and $\sum a_j\epsilon_j=\sum a_jf(\bm{r}_j)=f(\sum a_j\bm{r}_j)=f(\bm{r}_0)\ge\epsilon.$
\end{proof}

Roughly speaking, we show in Lemma \ref{lemma uniformnefpolytope} that there exists a uniform rational anti-nef polytope by assuming that the Cartier index of the $\Rr$-Cartier nef divisor is bounded. Problem \ref{conj: uniformnefness} is about whether the assumption on the Cartier index is necessary or not.

\begin{problem}\label{conj: uniformnefness}
	Let $d,m$ be two positive integers and $\bm{v}=(v_1^0,\ldots,v_m^0)\in\Rr^m$ a point. Then there exists a uniform rational polytope $P\subseteq\Rr^m$ containing $\bm{v}$ depending only on $d,m$ and $\bm{v}$ satisfying the following.  
	
	Assume that $(X/Z\ni z,B:=\sum_{i=1}^m v_i^0B_{i})$ is a pair of dimension $d$ such that
	\begin{itemize}
	    \item $X$ is of Fano type over $Z$,
	    \item $B_1,\ldots,B_m\ge0$ are Weil divisors on $X$,
	    \item $(X/Z\ni z,B)$ is lc over $z$, and
	    \item $-(K_X+B)$ is nef over $Z$.
	\end{itemize}
	Then $(X/Z\ni z,B':=\sum_{i=1}^m v_iB_{i})$ is lc over $z$ and $-(K_X+B')$ is nef over $Z$ for any $(v_1,\ldots,v_m)\in P$.
\end{problem}

We may show Problem \ref{conj: uniformnefness} when one of the following conditions holds:
\begin{enumerate}
	\item the Picard number of $X$ is $1$ (\cite[Theorem 5.17]{HLS19}), or
	\item $(X,B)$ is $\epsilon$-lc (Birkar-Borisov-Alexeev-Borisov Theorem), or 
	\item the Cartier indices of $K_X,B_i$ are bounded from above (Lemma \ref{lemma uniformnefpolytope}).
\end{enumerate}

\medskip

Our next question is about whether Nakamura's conjecture \cite[Conjecture 1.1]{mustata-nakamura18} holds for DCC coefficients or not.

\begin{problem}\label{problem: DCCmustatanakamura}
   Let $\Ii\subseteq[0,1]$ be a DCC set and $X\ni x$ a klt germ. Then there exists an integer $N$ depending only on $X\ni x$ and $\Gamma$ satisfying the following.
     
   Assume that $(X,B:=\sum_{i=1}^mb_iB_i)$ is a pair such that
     \begin{enumerate}
         %\item $X$ is smooth,
         \item $B_1,\dots,B_m$ are distinct prime divisors on $X$,
         \item $(X,B)$ is lc, and
         \item $b_i\in \Ii$ for any $1\le i\le m$.
     \end{enumerate}
     Then there is a prime divisor $E$ over $X$ such that $\mld(X\ni x,B)=a(E,X,B)$ and $a(E,X,0)\le N.$
\end{problem}

\medskip

Finally, we expect the following slightly stronger version of Conjecture \ref{conj:ShoMc} holds.

\begin{problem}\label{problem: MSconj}
	Let $d$ be a positive integer and $a<\epsilon$ two non-negative real numbers. Then there exists a positive real number $\delta$ depending only on $d$, $a$ and $\epsilon$ satisfying the following. 
	
	Assume that $(X/Z\ni z,B)$ is a pair of dimension $d$ such that 
	\begin{enumerate}
	   \item ${\rm codim}\ z=1$,
	   \item $X$ is of Fano type over $Z$, and
	   \item $(X/Z\ni z,B)$ is an $(\epsilon,\Rr)$-complement of itself.
	   %\item the general fiber of $f$ are of Fano type.
	\end{enumerate}
	Then possibly shrinking $Z$ near $z$, $\alct(X/Z\ni z,B;f^{*}z)\ge \delta$.
\end{problem}
Let $d=3$, $\epsilon=1$, $a=0$ and $\dim Z=2$ in Problem \ref{problem: MSconj}, a conjecture of Shokurov predicts that the optimal value of $\delta$ is $\frac{1}{2}$ (c.f. \cite[10.6.2. Conjecture]{Prokhorov18}). This conjecture is related to Iskovskikh's conjecture on $\Qq$-conic bundles \cite{Iskovskikh96}, which was proved by Mori and Prokhorov \cite[Theorem 1.2.7]{MoriProkhorov08}.
\begin{appendices}
\section{Structure of surface singularities}

Let $\mathcal{DG}$ be an extended dual graph with vertices $v_i$ $(1\le i\le n)$ which corresponds to the curves $E_i$ $(1\le i\le n)$. We define $\Delta(\mathcal{DG})$ as the absolute value of the determinant of the matrix $(E_i\cdot E_j)_{1\le i,j\le n}$, and set $\Delta(\emptyset)=1$. 

Assume that $\mathcal{DG}_1$ and $\mathcal{DG}_2$ are subgraphs of $\mathcal{DG}$, we define $\mathcal{DG}_1\sqcup\mathcal{DG}_2$ as the disjoint union of $\mathcal{DG}_1$ and $\mathcal{DG}_2$. We define $\mathcal{DG}-\mathcal{DG}_1$ as the subgraph of $\mathcal{DG}$ obtained by deleting from $\mathcal{DG}$ the vertices of $\mathcal{DG}_1$ and all the edges incident to them.

\begin{comment}
\begin{lem}[{\cite[Lemma 3.1.8 and Lemma 3.1.9]{Fli92}}]\label{lemma cofactor}

     Let $\mathcal{DG}$ be a tree with simple edges and $n$ vertices, $v_{j_1}, v_{j_2}$ two vertices of $\mathcal{DG}.$ Assume that the matrix $(E_i\cdot E_j)$ is negative definite. Then the $(j_1,j_2)$ cofactor of $(E_i\cdot E_j)$ is
     $$A_{j_{1} j_{2}}=-(-1)^{n} \Delta (\mathcal{DG}-( v_{j_{1}}\mapsto v_{j_{2}})).$$         
     
     Moreover, if $v$ is a vertex of $\mathcal{DG}$ of weight $w$, and $v_1,\dots, v_s$ are vertices that are adjacent to $v.$ Then
     $$\Delta(\mathcal{DG})=w\Delta(\mathcal{DG}-v)-\sum_{i=1}^s \Delta(\mathcal{DG}-v-v_i).$$
     
\end{lem}

\han{DG-vivj,DG-graph}

\begin{lem}\label{lemma solution}
     
     Let $(X\ni x,B)$ be an lc surface germ and $f:Y\to X$ the minimal resolution of $X$ with exceptional divisors $E_1,\dots,E_n.$ Assume that the dual graph $\mathcal{DG}$ of $f$ is a tree with simple edges. Then
     $$a_{j}=\frac{\sum_{k=1}^n \Delta(\mathcal{DG}-(v_{j}\mapsto v_{k}))c_{k}}{\Delta(\mathcal{DG})},$$ 
     where $c_{k}:=2-2p_a(E_k)-(B_Y+\sum_{i\neq k} E_{i})\cdot E_{k},B_Y$ is the strict transform of $B$ and $a_j:=a(E_j,X,B)$ for any $1\le j\le n.$  
     
\end{lem}

\begin{proof}
By the adjunction formula, we have
       \begin{align*} 
           (\sum_{i=1}^n a_iE_i)\cdot E_k &=(K_{Y}+B_Y+\sum_{i=1}^n E_i)\cdot E_k
           \\ &=2p_a(E_k)-2+(B_Y+\sum_{i\neq k}E_i)\cdot E_k
           \\ &=-c_k.
       \end{align*}
     Then the result follows from Lemma \ref{lemma cofactor} and Cramer's Rule.
\end{proof}

\end{comment}

\begin{figure}[ht]
     
     \begin{tikzpicture}
         
         %%\\case (a)
         \draw (-1.5,0) ellipse (2 and 0.25);
         \draw (0.75,0) circle (0.1);
         \draw (-1.4,0) ellipse (2.5 and 0.45);
         \draw [<-](-2,0.24)--(-3,1);          %%\\q
         \draw [->](-1.4,1)--(-1.4,0.45);         %%\\m
         \node [left] at (-2.9,1) {\footnotesize$q$};
         \node [above] at (-1.4,0.9) {\footnotesize$m$};
         \node [below] at (-1,-0.8) {\footnotesize$$};
 
     \end{tikzpicture}
     \caption{}
     \label{fig:1}    
\end{figure}
     We need the following well-known description of the ground chains {\cite[3.1.11]{Fli92}}. Every chain with positive integer weights $w_1, \ldots, w_n \geq 2$ $(n\ge 2)$ corresponds in a unique way to the pair $(m, q),$ where $m=\Delta(\mathcal{DG})$ and $1 \leq q<m$ is an integer 
coprime to $m$ defined by:
     $$\frac{m}{q}=w_n-\frac{1}{w_{n-1}-\frac{1}{\cdots w_{1}}}.$$
     If $n=1,$ we set $m=\Delta(\mathcal{DG}),q=1$. If $n=0,$ set $m=1,q=0$.

\medskip

We introduce the following notation.

\noindent{\bf Notation.}
    Let $\mathcal{DG}$ be an extended dual graph with simple edges and rational vertices (curves). 
    \begin{enumerate}
      \item By $\mathcal{DG}=\mathcal{DG}_{ch}^{(m,q),u}$ we mean that $\mathcal{DG}$ is a chain whose ground graph corresponds to the pair $(m,q)$ and $u$ is determined by the external part. In particular, by $\mathcal{DG}=\mathcal{DG}_{ch}^{2^A}$ we mean that $\mathcal{DG}$ is a chain with $A$ vertices of weight $2$ and it has no external part.
          
      \item By $\mathcal{DG}=\mathcal{DG}_{tr}^{2^A}$ we mean that $\mathcal{DG}$ is a tree with $A+2$ vertices of weight $2$ and it has no external part. Moreover, it has exactly one fork which has three branches and two of the branches have only one vertex of weight $2$.
    \end{enumerate} 

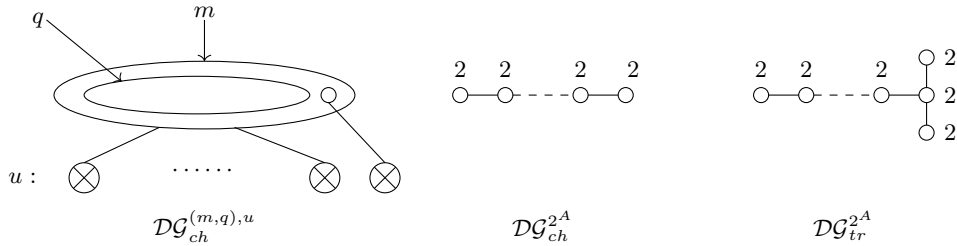
\begin{figure}[ht]  
          \begin{tikzpicture}
         \draw (-1,0) ellipse (1.5 and 0.25);   %%\\short oval
         \draw (0.75,0) circle (0.1);
         \draw (-0.9,0) ellipse (2 and 0.45);   %% long oval
         \node [below] at (-0.9,-1.4) {\footnotesize$\mathcal{DG}_{ch}^{(m,q),u}$};
         \draw [<-](-2,0.18)--(-3,1);             %%\\q
         \draw [<-](-0.9,0.45)--(-0.9,1);         %%\\m
         \node [left] at (-2.9,1) {\footnotesize$q$};
         \node [above] at (-0.9,0.9) {\footnotesize$m$};
         %\draw [<-](0.75,0.1)--(1.5,0.9);
         %\node [above] at (1.5,0.9) {\footnotesize$v_p{\rm (min)}$};         
         
         %%%%%%\\external vertice 
         \node [left] at (-3,-1.1) {\footnotesize$u:$};
         \draw (1.5,-1.1) circle (0.2);       %% conn min          
         \draw (0.75,-0.1)--(1.5,-0.9);          
         \draw (1.359,-0.959)--(1.641,-1.241);
         \draw (1.359,-1.241)--(1.641,-0.959);
         
         \draw (0.7,-1.1) circle (0.2);       %% second
         \draw (0.7,-0.9)--(-0.5,-0.43);
         \draw (0.559,-0.959)--(0.841,-1.241);
         \draw (0.559,-1.241)--(0.841,-0.959);
         
         \draw (-2.5,-1.1) circle (0.2);      %% first
         \draw (-2.5,-0.9)--(-1.5,-0.43);
         \draw (-2.641,-0.959)--(-2.359,-1.241);
         \draw (-2.641,-1.241)--(-2.359,-0.959);
         \node at (-0.9,-1) {\footnotesize$\cdots \cdots$};  
         
         %%%%%%%%%%%%%%%%%%%%

         \draw (2.5,0) circle (0.1);
         \node [above] at (2.5,0.1) {\footnotesize$2$}; 
         \draw (2.6,0)--(3,0);
         \draw (3.1,0) circle (0.1);
         \node [above] at (3.1,0.1) {\footnotesize$2$};
         \draw [dashed](3.2,0)--(4,0);
         \draw (4.1,0) circle (0.1);
         \node [above] at (4.1,0.1) {\footnotesize$2$};
         \draw (4.2,0)--(4.6,0);
         \draw (4.7,0) circle (0.1);
         \node [above] at (4.8,0.1) {\footnotesize$2$};
         \node [below] at (3.6,-1.4) {\footnotesize$\mathcal{DG}_{ch}^{2^A}$};  
         
         %%%%%%%%%%%%%%%%%%%%%
         \draw (6.5,0) circle (0.1);
         \node [above] at (6.5,0.1) {\footnotesize$2$}; 
         \draw (6.6,0)--(7,0);
         \draw (7.1,0) circle (0.1);
         \node [above] at (7.1,0.1) {\footnotesize$2$};
         \draw [dashed](7.2,0)--(8,0);
         \draw (8.1,0) circle (0.1);
         \node [above] at (8.1,0.1) {\footnotesize$2$};
         \draw (8.2,0)--(8.6,0);
         \draw (8.7,0) circle (0.1);
         \node [right] at (8.8,0) {\footnotesize$2$};
         \node [below] at (7.6,-1.4) {\footnotesize$\mathcal{DG}_{tr}^{2^A}$};  
         \draw (8.7,0.5) circle (0.1);      
         \draw (8.7,-0.5) circle (0.1);  
         \draw (8.7,0.1)--(8.7,0.4);                    
         \draw (8.7,-0.1)--(8.7,-0.4);  
         \node [right] at (8.8,0.5) {\footnotesize$2$};          
         \node [right] at (8.8,-0.5) {\footnotesize$2$};             
                         
   \end{tikzpicture}          
\end{figure}

     Let $(X\ni x,B:=\sum_{i=1}^{s} b_iB_i)$ be an lc surface germ, where $B_i$ are distinct prime divisors. Let $f:Y\to X$ be the minimal resolution of $X$ with exceptional divisors $E_1,\dots,E_n$. Suppose that the dual graph $\mathcal{DG}$ of $f$ is a chain, and $v_i$ is adjacent to $v_{i+1}$ for any $1\le i\le n-1$. We define
     $$\alpha(\mathcal{DG},\bm{b}):=1-\sum_{k=1}^n (B_Y\cdot E_k)\Delta(\mathcal{DG}-\{v_k,\dots ,v_n\}),$$
     where $B_Y$ is the strict transform of $B$ on $Y$ and $\bm{b}:=(b_1,\dots,b_s)$. 
     
\medskip

\begin{lem}[{\cite[lemma 3.3]{Ale93}}]\label{lemma 3.3}

     Let $\epsilon\le1$ be a positive real number and $\Ii\subseteq[0,1]$ a $($not necessarily DCC$)$ set such that $\inf_{\gamma\in\Ii}\{\gamma>0\}>0$. Then there exists a positive integer $N_0$ depending only on $\epsilon$ and $\Ii$ satisfying the following. 
     
     Assume that $(X\ni x,B:=\sum_{i=1}^{s} {b_iB_i})$ is an lc surface germ such that
     \begin{itemize}      
        \item $B_1,\dots,B_s$ are distinct prime divisors,  
        \item $b_i\in\Ii$ for any $1\le i\le s,$ and
        \item $\pld(X,B)\ge \epsilon.$         
     \end{itemize}     
     Then the dual graph of the minimal resolution of $X$ belongs to one of the following sets.     
     \begin{enumerate}     
        \item $\mathfrak{F}_{\epsilon,\Ii}$: consists of finitely many graphs $\mathcal{DG}$ with at most $N_0-1$ vertices. In particular, there exists a positive integer $N(\epsilon,\Ii)$ depending only on $\epsilon$ and $\Ii$ such that $\Delta(\mathcal{DG})\le N(\epsilon,\Ii)$ for any $\mathcal{DG}\in \mathfrak{F}_{\epsilon,\Ii}$.    
               
        \item $\mathfrak{C}_{\epsilon,\Ii}$: consists of chains $\mathcal{DG}=\mathcal{DG}\{\mathcal{DG}_{ch}^{(m_1,q_1),u_1},\mathcal{DG}_{ch}^{2^A},\mathcal{DG}_{ch}^{(m_2,q_2),u_2}\}$, where $A\in\Zz_{\ge1},\mathcal{DG}_{ch}^{(m_i,q_i),u_i}\in \mathfrak{F}_{\epsilon,\Ii}$ $(i=1,2)$, $\mathcal{DG}-\mathcal{DG}_{ch}^{2^A}=\mathcal{DG}_{ch}^{(m_1,q_1),u_1}$ $ \sqcup\mathcal{DG}_{ch}^{(m_2,q_2),u_2}$ such that $m_i\le\lfloor\frac{2}{\min\{\epsilon,\epsilon_0\}}\rfloor^{\lfloor\frac{2}{\epsilon}\rfloor}$ $(i=1,2)$. Moreover, we have
              \begin{align*}
                 \pld(X,B)=\min\{&\frac{(A+\frac{m_2}{m_2-q_2})\frac{\alpha_1}{m_1-q_1}+\frac{q_1}{m_1-q_1}\frac{\alpha_2}{m_2-q_2}}{A+\frac{q_1}{m_1-q_1}+\frac{m_2}{m_2-q_2}},\\&
                  \frac{(A+\frac{m_1}{m_1-q_1})\frac{\alpha_2}{m_2-q_2}+\frac{q_2}{m_2-q_2}\frac{\alpha_1}{m_1-q_1}}{A+\frac{q_2}{m_2-q_2}+\frac{m_1}{m_1-q_1}}
                  \},
              \end{align*}
              where $\alpha_i:=\alpha(\mathcal{DG}_{ch}^{(m_i,q_i),u_i},\bm{b}:=(b_1,\dots,b_s))$ $(i=1,2)$ which are independent of $A$. Moreover, if $\Ii$ satisfies the DCC, then we may additionally require that
    $$\min\{\frac{\alpha_1}{m_1-q_1},\frac{\alpha_2}{m_2-q_2}\}\ge\epsilon.$$
              
        \item $\mathfrak{T}_{\epsilon,\Ii}$: consists of trees $\mathcal{DG}=\mathcal{DG}\{\mathcal{DG}_{ch}^{(m_1,q_1),u_1},\mathcal{DG}_{tr}^{2^A}\}$, where $\mathcal{DG}_{ch}^{(m_1,q_1),u_1}$ $\in\mathfrak{F}_{\epsilon,\Ii}, \mathcal{DG}-\mathcal{DG}_{ch}^{(m_1,q_1),u_1}=\mathcal{DG}_{tr}^{2^A}$ and $m_1\le\lfloor\frac{2}{\min\{\epsilon,\epsilon_0\}}\rfloor^{\lfloor\frac{2}{\epsilon}\rfloor}$. Moreover, $\Delta(\mathcal{DG})\le N(\epsilon,\Ii)$ and
        $$\pld(X,B)=\frac{\alpha_1}{m_1-q_1},$$
        where $\alpha_1:=\alpha(\mathcal{DG}_{ch}^{(m_1,q_1),u_1},\bm{b}:=(b_1,\dots,b_s))$ which is independent of $A$.                    
    \end{enumerate} 
   
\end{lem}      
     %%%%%%%%%% graphs  
     \begin{figure}[ht]     
          %%%Type A_1
     \begin{tikzpicture}        
                                                %%\\oval at left
         \draw (-1,0) ellipse (1.5 and 0.25);   %%\\short oval
         \draw (0.75,0) circle (0.1);
         \draw (-0.9,0) ellipse (2 and 0.45);   %% long oval
         \draw (1.1,0)--(1.4,0);
         \draw (1.5,0) circle (0.1);
         \node [above] at (1.5,0.1) {\footnotesize$2$}; 
         \draw (1.6,0)--(2,0);
         \draw (2.1,0) circle (0.1);
         \node [above] at (2.1,0.1) {\footnotesize$2$};
         \draw (2.2,0)--(2.6,0);
         \draw (2.7,0) circle (0.1);
         \node [above] at (2.7,0.1) {\footnotesize$2$};
         \node [below] at (2.7,-1.4) {\footnotesize$\mathcal{DG}\{\mathcal{DG}_{ch}^{(m_1,q_1),u_1},
                                     \mathcal{DG}_{ch}^{2^A},\mathcal{DG}_{ch}^{(m_2,q_2),u_2}\}$};
         \draw [dashed](2.8,0)--(3.7,0);
         \draw [<-](-2,0.18)--(-3,1);             %%\\q
         \draw [<-](-0.9,0.45)--(-0.9,1);         %%\\m
         \node [left] at (-2.9,1) {\footnotesize$q_1$};
         \node [above] at (-0.9,0.9) {\footnotesize$m_1$};
         \draw [<-](0.75,0.1)--(1.5,0.9);
         \node [above] at (1.5,0.9) {\footnotesize$v_p{\rm (min)}$};         
         %\\external vertice 
         \draw (1.5,-1.1) circle (0.2);       %% conn min          
         \draw (0.75,-0.1)--(1.5,-0.9);          
         \draw (1.359,-0.959)--(1.641,-1.241);
         \draw (1.359,-1.241)--(1.641,-0.959);
         
         \draw (0.7,-1.1) circle (0.2);       %% second
         \draw (0.7,-0.9)--(-0.5,-0.43);
         \draw (0.559,-0.959)--(0.841,-1.241);
         \draw (0.559,-1.241)--(0.841,-0.959);
         
         \draw (-2.5,-1.1) circle (0.2);      %% first
         \draw (-2.5,-0.9)--(-1.5,-0.43);
         \draw (-2.641,-0.959)--(-2.359,-1.241);
         \draw (-2.641,-1.241)--(-2.359,-0.959);
         \node at (-0.9,-1) {\footnotesize$\cdots \cdots$};       
          %%\\oval at right
         \draw (6.8,0) ellipse (2 and 0.45);         %%\\long oval
         \draw (6.9,0) ellipse (1.5 and 0.25);       %%\\short oval
         \draw (5.1,0) circle (0.1);                 %\\min point
         \draw (4.3,0) circle (0.1);
         \draw (3.8,0) circle (0.1);
         \draw (4.4,0)--(4.8,0);
         \draw (3.9,0)--(4.2,0);
         \node [above] at (4.3,0.1) {\footnotesize$2$}; 
         \node [above] at (3.8,0.1) {\footnotesize$2$}; 
         \node [right] at (8.5,1) {\footnotesize$q_2$};
         \node [above] at (6.8,0.9) {\footnotesize$m_2$};         
         \draw [<-](7.7,0.2)--(8.6,1);    %%q
         \draw [<-](6.8,0.45)--(6.8,1);    %% m
         %\draw [<-](5.1,0.1)--(4.3,0.9);
         %\node [above] at (4.3,0.9) {\footnotesize${\rm min}$};                   
         %\\external vertice
         \draw (4.3,-1.1) circle (0.2);
         \draw (5.1,-0.1)--(4.3,-0.9);       %% conn min               
         \draw (4.16,-0.959)--(4.44,-1.241);
         \draw (4.16,-1.241)--(4.44,-0.959);  
         
         \draw (8.4,-1.1) circle (0.2);     %% second
         \draw (8.4,-0.9)--(7.5,-0.43);
         \draw (8.259,-0.959)--(8.541,-1.241);
         \draw (8.259,-1.241)--(8.541,-0.959);
         
         \draw (5.1,-1.1) circle (0.2);     %% first
         \draw (5.1,-0.9)--(6.1,-0.43);
         \draw (4.96,-0.959)--(5.24,-1.241);
         \draw (4.96,-1.241)--(5.24,-0.959);         
         \node at (6.8,-1) {\footnotesize$\cdots \cdots$};       
    \end{tikzpicture}
        %%%type A_2
    \begin{tikzpicture}
                                                    
         \draw (-1,0) ellipse (1.5 and 0.25);   %%\\short oval
         \draw (0.75,0) circle (0.1);
         \node [below] at (0.75,-1.4) {\footnotesize$\mathcal{DG}\{\mathcal{DG}_{ch}^{(m_1,q_1),u_1},      
                                      \mathcal{DG}_{tr}^{2^A}\}$};
         \draw (-0.9,0) ellipse (2 and 0.45);   %% long oval
         \draw (1.1,0)--(1.4,0);
         \draw (1.5,0) circle (0.1);
         \node [above] at (1.5,0.1) {\footnotesize$2$}; 
         \draw (1.6,0)--(2,0);
         \draw (2.1,0) circle (0.1);
         \node [above] at (2.1,0.1) {\footnotesize$2$};
         \draw (2.2,0)--(2.6,0);
         \draw (2.7,0) circle (0.1);
         \node [above] at (2.7,0.1) {\footnotesize$2$};
         \draw [dashed](2.8,0)--(3.7,0);
         \draw [<-](-2,0.18)--(-3,1);             %%\\q
         \draw [<-](-0.9,0.45)--(-0.9,1);         %%\\m
         \node [left] at (-2.9,1) {\footnotesize$q_1$};
         \node [above] at (-0.9,0.9) {\footnotesize$m_1$};
         \draw [<-](3.8,0.1)--(2.5,0.9);
         \draw [<-](0.8,0.1)--(2.1,0.9);
         \node [above] at (2.3,0.9) {\footnotesize${\rm (min)}$};        
         %\\external vertice 
         \draw (1.5,-1.1) circle (0.2);       %% conn min          
         \draw (0.75,-0.1)--(1.5,-0.9);          
         \draw (1.359,-0.959)--(1.641,-1.241);
         \draw (1.359,-1.241)--(1.641,-0.959);
         
         \draw (0.7,-1.1) circle (0.2);       %% second
         \draw (0.7,-0.9)--(-0.5,-0.43);
         \draw (0.559,-0.959)--(0.841,-1.241);
         \draw (0.559,-1.241)--(0.841,-0.959);
         
         \draw (-2.5,-1.1) circle (0.2);      %% first
         \draw (-2.5,-0.9)--(-1.5,-0.43);
         \draw (-2.641,-0.959)--(-2.359,-1.241);
         \draw (-2.641,-1.241)--(-2.359,-0.959);
         \node at (-0.9,-1) {\footnotesize$\cdots \cdots$};
                
         \draw (3.9,0) circle (0.1);
         \draw (3.9,0.6) circle (0.1);
         \draw (3.9,-0.6) circle (0.1);
         
         \draw (3.9,0.1)--(3.9,0.5);
         \draw (3.9,-0.1)--(3.9,-0.5);
         
         \node [right] at (4,0) {\footnotesize$2$};
         \node [right] at (4,0.6) {\footnotesize$2$};
         \node [right] at (4,-0.6) {\footnotesize$2$};
    \end{tikzpicture}
     
    \caption{}
    \label{fig:2}   
    \end{figure}

\begin{proof}
     By \cite[lemma 3.3]{Ale93}, we only need to show the last statement in $(2)$. Since $\Ii$ satisfies the DCC, there exists a positive real number $\epsilon'$ depending only on $\epsilon$ and $\Ii$ such that $\epsilon-\frac{1-b}{l}\ge\epsilon'$ for any positive real number $\frac{1-b}{l}<\epsilon$, where $b\in\Ii_+,l\in\Zz_{>0}$. We show that the statement holds if $\mathcal{DG}$ has more than $N_0':=2N_0+\lfloor\frac{2\lfloor\frac{2}{\min\{\epsilon,\epsilon_0\}}\rfloor^{\lfloor\frac{2}{\epsilon}\rfloor}}{\epsilon'}\rfloor+1$ vertices. Otherwise we may assume that $\frac{\alpha_1}{m_1-q_1}<\epsilon$, and hence $\epsilon-\frac{\alpha_1}{m_1-q_1}\ge\epsilon'$.
     It follows that
     $$A\le\frac{\frac{m_2}{m_2-q_2}\frac{\alpha_1}{m_1-q_1}+\frac{q_1}{m_1-q_1}\frac{\alpha_2}{m_2-q_2}}{\epsilon-\frac{\alpha_1}{m_1-q_1}}<\frac{2\lfloor\frac{2}{\min\{\epsilon,\epsilon_0\}}\rfloor^{\lfloor\frac{2}{\epsilon}\rfloor}}{\epsilon'}.$$
     Since $\mathcal{DG}_{ch}^{(m_i,q_i),u_i}\in \mathfrak{F}_{\epsilon,\Ii}$ $(i=1,2)$, $\mathcal{DG}$ has at most $N_0'-1$ vertices, a contradiction. Replacing $N_0$ by $N_0'$, we are done.
\end{proof}

\begin{rem}
In Lemma \ref{lemma 3.3}(2), if $\frac{\alpha_1}{m_1-q_1}\le \frac{\alpha_2}{m_2-q_2}$, then
$$\pld(X,B)=\frac{(A+\frac{m_2}{m_2-q_2})\frac{\alpha_1}{m_1-q_1}+\frac{q_1}{m_1-q_1}\frac{\alpha_2}{m_2-q_2}}{A+\frac{q_1}{m_1-q_1}+\frac{m_2}{m_2-q_2}},$$
since
\begin{align*}
&\frac{(A+\frac{m_2}{m_2-q_2})\frac{\alpha_1}{m_1-q_1}+\frac{q_1}{m_1-q_1}\frac{\alpha_2}{m_2-q_2}}{A+\frac{q_1}{m_1-q_1}+\frac{m_2}{m_2-q_2}}-\frac{(A+\frac{m_1}{m_1-q_1})\frac{\alpha_2}{m_2-q_2}+\frac{q_2}{m_2-q_2}\frac{\alpha_1}{m_1-q_1}}{A+\frac{q_2}{m_2-q_2}+\frac{m_1}{m_1-q_1}}\\
=&\frac{(A+1)(\frac{\alpha_1}{m_1-q_1}-\frac{\alpha_2}{m_2-q_2})}{A+\frac{q_1}{m_1-q_1}+\frac{m_2}{m_2-q_2}}\le0.
\end{align*}
\end{rem}

Theorem \ref{cor accforpld} is an easy consequence of Lemma \ref{lemma 3.3}.
\begin{thm}[{\cite[Theorem 3.2]{Ale93}}]\label{cor accforpld}
     Let  $\Gamma\subseteq[0,1]$ be a DCC set. Then $ \mathcal{P}ld(\Gamma)$ satisfies the ACC.
\end{thm}

\begin{lem}\label{lemma minimalreslwithbddvertices}
     Let $\epsilon \leq 1$ be a positive real number and $N$ a positive integer. Then there is a positive integer $N_{0}$ depending only on $\epsilon$ and $N$ satisfying the following.
     
     Assume that $\left(X \ni x, B:=\sum b_{i} B_{i}\right)$ is an lc surface germ such that
     \begin{enumerate}
         \item $B_{i}$ are distinct prime divisors on $X$, and
         \item $+\infty>\pld(X\ni x,B)\ge\epsilon.$
     \end{enumerate}
     Let $f:Y\to X$ be the minimal resolution of $X$ with exceptional divisors $E_1,\dots,E_{n}$, and $n\le N.$ Then
     $$a(E_{j}, X, B)=\frac{l_{j_0}-\sum_i l_{j_i} b_{i}}{l_j}$$
     for some positive integers $l_{j_0}\le N_{0},l_{j_i},l_j$, and any $1\le j\le n.$
\end{lem}
\begin{proof}
      Let $\mathcal{DG}$ be the dual graph of $f,v_{i}$ the vertex corresponds to $E_{i}, a_{i}:=a(E_{i}, X, B)$ and $w_{i}:=-(E_{i} \cdot E_{i})$ for any $1 \le i\le n.$ By Lemma \ref{lemma concave}$(1)$, $w_{i}\le\lfloor\frac{2}{\epsilon}\rfloor$ for any $1\le i\le n.$ By \cite[3.1.10]{Fli92},
      $$a_{j}=\frac{\sum_{k=1}^{n} \Delta(\mathcal{DG}-(v_{j} \mapsto v_{k})) c_{k}}{\Delta(\mathcal{DG})},$$
      where $(v_j\mapsto v_k)$ is the unique shortest path from $v_j$ to $v_k$, and $c_{k}:=2-2 p_{a}(E_{k})-(f_{*}^{-1} B+\sum_{i \ne k} E_{i}) \cdot E_{k}.$ Since $n \leq N$ and $w_{j} \le\lfloor\frac{2}{\epsilon}\rfloor$, it follows that there exists a positive integer $N_0$ with the required properties.
\end{proof}

     As an application of Lemma \ref{lemma 3.3}, we have the following result.

\begin{prop} \label{prop pld}
     Let $\epsilon\le1$ be a non-negative rational number and $\Ii\subseteq \Qq \cap[0,1]$ a finite set. Then there is a positive integer $I$ depending only on $\epsilon$ and $\Gamma$ satisfying the following. 
     
     Assume that $(X\ni x,B)$ is an lc surface germ such that 
     \begin{enumerate}
       \item $B\in \Ii,$ and
       \item $\pld(X\ni x,B)=\epsilon.$
     \end{enumerate}
     Then $I(K_X+B)$ is Cartier near $x$.
\end{prop}
\begin{proof}
      If $\epsilon=0$ and $B=0$, then by the classification of lc surface singularities, $IK_X$ is Cartier for some positive integer $I\in\{1,2,3,4,6\}$. If $\epsilon=0$ and $B\neq 0$, then $\mld(X\ni x,B)=0$ and $X$ is klt near $x$. The boundedness of Cartier index follows from \cite[Corollary 1.9]{Bir16a}. 

      It suffices to show the proposition for the case when $\epsilon>0$. Assume that  $\Ii\cup\{\epsilon\}\subseteq\frac{1}{I_0}\Zz$ for some positive integer $I_0.$ Let $f:Y\to X$ be the minimal resolution of $X,$ and we may write
      $$ K_{Y}+B_Y+\sum (1-a_k)E_k=f^{*}(K_{X}+B),$$
      where $B_Y$ is the strict transform of $B$, $E_k$ are the exceptional divisors and $a_k:=a(E_k,X,B)$ for any $k$. By Lemma \ref{lemma 3.3}, the dual graph $\mathcal{DG}$ of $f$ belongs one of $\mathfrak{F}_{\epsilon,\Ii},\mathfrak{C}_{\epsilon,\Ii}$ and $\mathfrak{T}_{\epsilon,\Ii}$. 
      %We may use the notation in the proof of Lemma \ref{lemma 3.3}. 
      
      Suppose that either $\mathcal{DG}\in\mathfrak{F}_{\epsilon,\Ii}$ or $\mathcal{DG}\in\mathfrak{T}_{\epsilon,\Ii}$. By Lemma \ref{lemma 3.3}, there exists a positive integer $I_1$ which only depends on $\epsilon$ and $\Ii$ such that $\Delta(\mathcal{DG})|I_1$. Let $I:=I_0I_1$. Then $Ia_k\in\Zz$ for any $k$.
      
      Now we consider the case when $\mathcal{DG}=\mathcal{DG}\{\mathcal{DG}_{ch}^{(m_1,q_1),u_1},\mathcal{DG}_{ch}^{2^A},\mathcal{DG}_{ch}^{(m_2,q_2),u_2}\}$ $\in\mathfrak{C}_{\epsilon,\Ii}$. Since $\Ii$ is a finite set, we may assume that $\epsilon\le\frac{\alpha_1}{m_{1}-q_{1}}\le\frac{\alpha_2}{m_{2}-q_{2}}$ and hence
      $$ \pld(X\ni x,B)=a_p=\frac{(A+\frac{m_2}{m_2-q_2})\frac{\alpha_1}{m_1-q_1}+\frac{q_1}{m_1-q_1}\frac{\alpha_2}{m_2-q_2}}{A+\frac{q_1}{m_1-q_1}+\frac{m_2}{m_2-q_2}}=\epsilon,$$
      %and $\epsilon\le\frac{\alpha_1}{m_{1}-q_{1}}\le\frac{\alpha_2}{m_{2}-q_{2}}.$ 
      where $p$ is the number of the vertices of $\mathcal{DG}_{ch}^{(m_1,q_1),u_1}$.
      It follows that $\epsilon=\frac{\alpha_1}{m_{1}-q_{1}}=\frac{\alpha_2}{m_{2}-q_{2}},a_p=a_{p+1}=\cdots=a_{p+A+1}=\epsilon$ and $I_0a_k\in\Zz$ for any $p\le k\le p+A+1.$ Since $a_{j+1}-w_ja_j+a_{j-1}=B_Y\cdot E_j,Ia_k\in \Zz$ for any $k.$

      It is enough to show that $I(K_X+B)$ is Cartier near $x$. Since $Ia_k\in\Zz$, $If^{*}(K_X+B)=I(K_Y+B_Y)+\sum I(1-a_k)E_k$ is Cartier. Let $H_Y\ge0$ be an ample $\Qq$-Cartier divisor on $Y$. Then $H:=f^*f_*H_Y-H_Y$ is ample and exceptional over $X$. Let $K_Y+B_Y':=f^{*}(K_X+B).$ Possibly rescaling $H,$ we may assume that every coefficient of $B_Y'+H$ is at most one. We may run a $(K_Y+B_Y'+H)$-MMP over $X$ and the MMP terminates with $X$ thanks to \cite[Theorem 1.1]{Fuj12}. Since the MMP is $(If^{*}(K_X+B))$-trivial, according to the cone theorem \cite[Theorem 3.2]{Fuj12}, $I(K_X+B)$ is Cartier near $x$.      
\end{proof}

\section{Nakamura's conjecture with DCC coefficients for surfaces}
%ACC for MLDs of smooth surface pairs
In this appendix, we show Theorem \ref{lemma dccsmoothmld}, that is, Nakamura's conjecture \cite[Conjecture 1.1]{mustata-nakamura18} holds for surface pairs with DCC coefficients. Musta\c{t}\u{a} and Nakamura proved Theorem \ref{lemma dccsmoothmld} for finite coefficients \cite[Theorem 1.3]{mustata-nakamura18}, and Alexeev proved it in full generality \cite[Lemma 3.7]{Ale93}. Our proof is based on Alexeev's idea and we use some different arguments. 

\begin{thm}\label{lemma dccsmoothmld}
     Let $\Ii\subseteq[0,1]$ be a DCC set. Then there exists a positive integer $N$ depending only on $\Gamma$ satisfying the following.
     
     Assume that $X\ni x$ is a smooth surface germ and $B_1,\dots,B_m$ are distinct prime divisors on $X$ such that
     \begin{itemize}
         %\item $X$ is smooth,
         \item $(X,B:=\sum_{i=1}^mb_iB_i)$ is lc,
         \item $b_i\in \Ii$ for any $1\le i\le m$, and
         \item $x\in\cap_{i=1}^m B_i$.
     \end{itemize}
     Then there is a prime divisor $E$ over $X$ such that $\mld(X\ni x,B)=a(E,X,B)$ and $a(E,X,0)\le N.$
     In particular, we have 
     $$\mld(X\ni x,B)=l_0-\sum_{i=1}^m l_ib_i$$
     for some non-negative integers $l_0,\dots,l_m$ and $l_0\le N$.
\end{thm}
     
\begin{proof}
     By Lemma \ref{lemmma extractmld}, there exists a sequence of blow-ups
     $$X_n\to X_{n-1}\to\cdots\to X_1\to X_0:=X$$
     with the following properties:
     \begin{itemize}
         \item $f_i:X_{i}\to X_{i-1}$ is the blow-up of $X_{i-1}$ at a point $x_{i-1}\in X_{i-1}$ with the exceptional divisor $E_i$ for any $1\le i\le n$, where $x_0:=x$,
         \item $x_i\in E_i$ for any $1\le i\le n-1$, and
         \item $a(E_i,X,B)>a(E_n,X,B)=\mld(X\ni x,B)$ for any $1\le i\le n-1$.
     \end{itemize}
     For convenience, we may denote the strict transform of $E_i$ on $X_k$ by $E_i$ for any $k\ge i.$ 
     
     Let $B_{X_i}$ be the strict transform of $B$ on $X_i$, $a_i:=a(E_i,X,B)$ and $g_i:=f_1\circ\cdots\circ f_i$. Let $v_i$ be the vertex corresponds to $E_i$, $w_i:=-(E_i\cdot E_i)$, $e_i:=1-a_i\le1$ for any $1\le i\le n,$
     and $c_i:=-1+\mult_{x_{i}}B_{X_i}$ for any $0\le i\le n-1.$ Since $\sum E_i$ is snc, there exists at most one $k<i$ such that $x_i\in E_{k}$. 
     
     If $\mld(X\ni x,B)\ge1,$ then
     $$\mld(X\ni x,B)=a(E_1,X,B)=2-\mult_xB,$$ 
     and $a(E_1,X,0)=2$ by Lemma \ref{lemmma extractmld}$(4)$. Let $N=2$ and we are done. In the following, we may assume that $\mld(X\ni x,B)<1.$     
     By Lemma \ref{lemmma extractmld}$(5)$, $1\ge e_{i}>0$ for any $1\le i\le n$. By Lemma \ref{lemma concave}$(4)$, the dual graph $\mathcal{DG}$ of $g_n$ is a chain. 
     \begin{claim}\label{claim blowuppoints}
If $E_i$ intersects $E_j$ on $X_j$ and $e_i\ge e_j$ for some $i<j$, then $x_{j}\in E_i.$
     \end{claim}
     \begin{proof}[Proof of {\rm Claim B.2}]
         Otherwise there exists a sequence of vertices $v_{k_1},\dots$, $v_{k_l}=:v_n$ such that $v_{k_s}$ is adjacent to $v_{k_{s+1}}$ for any $1\le s\le l-1,$ $v_{k_1}=v_{i}$ and $v_{k_2}=v_{j}$. By Lemma \ref{lemma concave}$(2)$, $e_i<e_{j}$, a contradiction.
     \end{proof}
     
     It is clear that $\mult_{x_{j}}\Supp B_{X_{j}}$ is decreasing, we will need the following claim.
     \begin{claim}\label{claim bddsnc}
     
        For any $j\ge 1$, there exists a positive integer $N_j$ depending only on $\Gamma$ and $j$ such that either $n-j\le N_j$ or
        $$\mult_{x_{m_j}}\Supp B_{X_{m_j}}\le\mult_{x_j}\Supp B_{X_j}-1$$
        for some positive integer $m_j$ satisfying $0<m_j-j\le N_j.$
     \end{claim}
    Suppose that Claim \ref{claim bddsnc} holds. Since $-1+\mult_xB\le1$, 
     $$\mult_x\Supp B\le\lfloor\frac{2}{b_0}\rfloor,$$
     where $b_0:=\min_{b\in \Ii}\{b>0\}>0.$ Thus $n$ and $a(E,X,0)$ are bounded from above by Claim \ref{claim bddsnc}.
    \end{proof}
     \begin{proof}[Proof of {\rm Claim B.3}]

        Let
        $$\Ii_j:=\{-s_0+\sum s_it_i>0\mid s_i\in\Zz_{\ge0},t_i\in \Ii,s_0\le j+1\}\cup\{0\}$$
        which satisfies the DCC set as $\Ii$ does.
        There exists a positive real number $\xi_j$ such that if $-1+\sum s_it_i>0$ for some $s_i\in\Zz_{\ge0}$, $t_i\in \Ii_j$, then $-1+\sum s_it_i\ge\xi_j.$
        
        We will show that $N_j:=\lfloor\frac{1}{\xi_j}\rfloor$ has the required properties. Suppose on the contrary that $n-j>N_j$ and $\mult_{x_{m_j}}\Supp B_{X_{m_j}}=\mult_{x_j}\Supp B_{X_j}$, where $m_j=j+N_j+1$. Thus
        $c_j=c_{j+1}=\cdots=c_{m_j}$.
        We first show that $c_j\le0$. Let $e_{i}'=e_{k}$ if $x_{i}\in E_{k}$ for some $k<i$, otherwise let $e_{i}'=0$.
        Now suppose that $c_j>0$ which implies that $c_j\in\Ii_j,c_j\ge\xi_j$ and $e_{j+1}=c_j+e_j+e_j'> c_j\ge\xi_j$. We show by induction on $k$ that $e_{j+k}> k\xi_j$ for any $1\le k\le N_j+1.$ Assume that for some $N_j\ge k_0\ge 1,$
        $e_{j+i}> i\xi_j$ for any $1\le i\le k_0$. Then by induction,
        $$e_{j+k_0+1}=c_j+e_{j+k_0}+e_{j+k_0}'\ge c_j+e_{j+k_0}> (k_0+1)\xi_j.$$
        We finish the induction. Hence $1\ge e_{m_j}>(m_j-j)\xi_j$ and $m_j-j=N_j+1\le N_j,$ a contradiction. Thus $c_j\le0$.
        
                     \begin{figure}[ht]
       \begin{tikzpicture}

       \draw (5,1.5)--(5,0);
       \node [left] at (5,0.75) {\tiny$E_{j'}$};
       \node at (5.4,0.3) {\tiny$x_{j+k}$};
       \draw (4.9,0.3)--(6.2,-1);
       \node [below] at (5.4,-0.3) {\tiny$E_{j+k}$};
       \draw (5.8,-0.8)--(7.5,-0.8);
       \node [below] at (7,-0.8) {\tiny$E_{j+k-1}$};
       \draw [dashed] (7.2,-0.9)--(8.3,-0.2);
       \draw(8.3,0.3)--(7.8,1.5);
       \node[right] at (8,1) {\tiny$E_{j''}$};
       \draw(8.22,0.7)--(8.22,-0.4);
       \node[right] at (8.22,0) {\tiny$E_{j+1}$};
       \node [below] at (7,-1.2) {\footnotesize{sub-dual graph in the proof Claim B.3}};
      \end{tikzpicture}
      %\capture{}
      \end{figure}     
        
        If $x_j\notin E_{j_0}$ for any $j_0<j$, then $e_{j+1}=e_{j}+c_j\le e_j$ which implies that $x_{j+1}\in E_j$ by Claim \ref{claim blowuppoints}. Since Claim \ref{claim bddsnc} fails for $j$, Claim \ref{claim bddsnc} also fails for $j+1$. Otherwise we can take $N_j:=N_{j+1}+1$ which has the required properties. Therefore possibly replacing $j,\Ii_j,\xi_j,N_j$ by $j+1,\Ii_{j+1},\xi_{j+1},N_{j+1}$ respectively, we may assume that $x_j\in E_{j_0}$ for some $j_0<j$.
                
        We show that either $x_{j+k-1}\in E_{j_0}$ for any $2\le k\le N_j+1$ or $x_{j+k-1}\in E_{j}$ for any $2\le k\le N_j+1$. Since $E_{j+1}$ intersects $E_{j_0},E_{j}$, and $\mathcal{DG}$ is a chain, either $x_{j+1}\in E_{j}$ or $x_{j+1}\in E_{j_0}$. We may assume that $x_{j+1}\in E_{j'}$, where $j'\in\{j,j_0\}$. Let $j''$ be the integer such that $\{j',j''\}=\{j,j_0\}$. By Claim \ref{claim blowuppoints},
        $$e_{j''}<e_{j+1}=c_j+e_j+e_{j_0}=c_j+e_{j'}+e_{j''},$$ 
        and hence $e_{j'}+c_j>0$. Thus $e_{j'}+c_j\ge \xi_j$ as $e_{j'}+c_j\in \Ii_j$.
        We then show by induction on $k$ that $x_{j+k-1}\in E_{j'}$ for any $2\le k\le  N_j+1$. Assume that for some $ N_j\ge k_0\ge 2,$
        $x_{j+i-1}\in E_{j'}$ for any $2\le i\le k_0$. It follows that
        $$e_{j+i}=e_{j+i-1}+(e_{j'}+c_j)=e_{j''}+i(e_{j'}+c_j)>i\xi_j$$ 
        for any $1\le i\le k_0$. If $x_{j+k_0}\notin E_{j'}$, then $x_{j+k_0}\in E_{j+k_0-1}$ as $\mathcal{DG}$ is a chain. It follows that
        $$e_{j+k_0+1}=c_{j}+e_{j+k_0}+e_{j+k_0-1},$$
        and $e_{j+k_0}>e_{j'}$ by Claim \ref{claim blowuppoints}. Let $\bar{e}_{j+k_0-2}=e_{j''}$ if $k_0=2$, otherwise let $\bar{e}_{j+k_0-2}=e_{j+k_0-2}$.
        %Thus, we have $$e_{j+k_0+1}+e_{j+k_0-2}-2e_{j+k_0-1}=e_{j+k_0}-e_{j_0}>0.$$
        Then on $X_{j+k_0+1}$, we have
        \begin{align*}
          0&=(K_{X_{j+k_0+1}}+B_{X_{j+k_0+1}}+\sum_{i=1}^{j+k_0+1} e_{i}E_{i})\cdot E_{j+k_0-1}
         \\&\ge-2+(e_{j+k_0-1}-1)E_{j+k_0-1}\cdot E_{j+k_0-1}+e_{j+k_0+1}+\bar{e}_{j+k_0-2}
         \\&\ge e_{j+k_0+1}+\bar{e}_{j+k_0-2}-2e_{j+k_0-1}
         \\&= e_{j+k_0}-e_{j'}
         > 0,
        \end{align*} 
        a contradiction. We finish the induction.
        Hence $x_{m_j-1}\in E_{j'}$ and 
        $1\ge e_{m_j}=e_{j''}+(m_j-j)(e_{j'}+c_j)> (m_j-j)\xi_j$ which implies that $m_j-j=N_j+1\le N_j,$ a contradiction.
     \end{proof}

Recall that Theorem \ref{lemma dccsmoothmld} implies Lemma \ref{lemma mld set}. The ACC for minimal log discrepancies in dimension 2 follows from Lemma \ref{lemma mld set} and Lemma \ref{cor accforpld}.

\begin{thm}[{\cite[Theorem 3.6]{Ale93}}]\label{thm accformld}
     Let $\Gamma\subseteq[0,1]$ be a DCC set. Then $\mathcal{M}ld(2,\Gamma)$ satisfies the ACC.   
\end{thm}
\begin{proof}
    It suffices to show that the set $\mathcal{M}ld(2,\Gamma)\cap[\epsilon,+\infty)$ satisfies the ACC for any $1\ge\epsilon>0$. The result follows from Lemma \ref{lemma mld set} and Lemma \ref{cor accforpld}.
\end{proof}

\end{appendices}

\bibliographystyle{alpha}

%\bibliography{bibfile}

\begin{thebibliography}{CDCH{\etalchar{+}}18}

\bibitem[AB14]{AlexeevBorisov14}
Valery Alexeev and Alexander Borisov.
\newblock On the log discrepancies in toric {M}ori contractions.
\newblock {\em Proc. Amer. Math. Soc.}, 142(11):3687--3694, 2014.

\bibitem[Ale93]{Ale93}
Valery Alexeev.
\newblock Two two-dimensional terminations.
\newblock {\em Duke Mathematical Journal}, 69(3):527--545, 1993.

\bibitem[Amb99]{Ambro99}
Florin Ambro.
\newblock On minimal log discrepancies.
\newblock {\em Math. Res. Lett.}, 6(5-6):573--580, 1999.

\bibitem[BCHM10]{BCHM10}
Caucher Birkar, Paolo Cascini, Christopher~D. Hacon, and James McKernan.
\newblock Existence of minimal models for varieties of log general type.
\newblock {\em J. Amer. Math. Soc.}, 23(2):405--468, 2010.

\bibitem[Bir04]{Bir04}
Caucher Birkar.
\newblock Boundedness of $\epsilon$-log canonical complements on surfaces,
  \url{https://www.dpmms.cam.ac.uk/~cb496/surfcomp.pdf}.
\newblock {\em preprint}, 2004.

\bibitem[Bir11]{Birkar11}
Caucher Birkar.
\newblock On existence of log minimal models {II}.
\newblock {\em J. Reine Angew. Math.}, 658:99--113, 2011.

\bibitem[Bir12]{Birkar12}
Caucher Birkar.
\newblock Existence of log canonical flips and a special {LMMP}.
\newblock {\em Publ. Math. Inst. Hautes \'Etudes Sci.}, 115:325--368, 2012.

\bibitem[Bir16a]{Bir16b}
Caucher Birkar.
\newblock Singularities of linear systems and boundedness of {F}ano varieties.
\newblock {\em arXiv:1609.05543}, 2016.

\bibitem[Bir16b]{Bir16crelle}
Caucher Birkar.
\newblock Singularities on the base of a {F}ano type fibration.
\newblock {\em J. Reine Angew. Math.}, 715:125--142, 2016.

\bibitem[Bir18]{Bir18}
Caucher Birkar.
\newblock Log {C}alabi-{Y}au fibrations.
\newblock {\em arXiv:1811.10709}, 2018.

\bibitem[Bir19]{Bir16a}
Caucher Birkar.
\newblock Anti-pluricanonical systems on {F}ano varieties.
\newblock {\em Ann. of Math. (2)}, 190(2):345--463, 2019.

\bibitem[BLX19]{BLX19}
Harold Blum, Yuchen Liu, and Chenyang Xu.
\newblock Openness of {K}-semistability for {F}ano varieties.
\newblock {\em arXiv:1907.02408}, 2019.

\bibitem[BS10]{BirkarSho10}
C.~Birkar and V.~V. Shokurov.
\newblock Mld's vs thresholds and flips.
\newblock {\em J. Reine Angew. Math.}, 638:209--234, 2010.

\bibitem[CDCH{\etalchar{+}}18]{CDHJS18}
Weichung Chen, Gabriele Di~Cerbo, Jingjun Han, Chen Jiang, and Roberto Svaldi.
\newblock Birational boundedness of rationally connected {C}alabi-{Y}au
  3-folds.
\newblock {\em arXiv:1804.09127}, 2018.

\bibitem[FM18]{FM18}
Stefano Filipazzi and Joaqu\'{i}n Moraga.
\newblock Strong $(\delta,n)$-complements for semi-stable morphisms.
\newblock {\em arXiv:1810.01990}, 2018.

\bibitem[FMX19]{FMX19}
Stefano Filipazzi, Joaqu\'{i}n Moraga, and Yanning Xu.
\newblock Log canonical 3-fold complements.
\newblock {\em arXiv:1909.10098}, 2019.

\bibitem[Fuj90]{Fujimoto90}
Yoshio Fujimoto.
\newblock On rational elliptic surfaces with multiple fibers.
\newblock {\em Publ. Res. Inst. Math. Sci.}, 26(1):1--13, 1990.

\bibitem[Fuj09]{Fujino09}
Osamu Fujino.
\newblock Effective base point free theorem for log canonical pairs--{K}oll\'ar
  type theorem.
\newblock {\em Tohoku Math. J. (2)}, 61(4):475--481, 2009.

\bibitem[Fuj12]{Fuj12}
Osamu Fujino.
\newblock Minimal model theory for log surfaces.
\newblock {\em Publ. Res. Inst. Math. Sci.}, 48:339--371, 06 2012.

\bibitem[Har77]{GTM52}
Robin Hartshorne.
\newblock {\em Algebraic geometry}.
\newblock Springer-Verlag, New York-Heidelberg, 1977.
\newblock Graduate Texts in Mathematics, No. 52.

\bibitem[HL18]{HanLiu18}
Jingjun Han and Wenfei Liu.
\newblock On numerical nonvanishing for generalized log canonical pairs.
\newblock {\em arXiv:1808.06361}, 2018.

\bibitem[HLQ17]{HLQ17}
Jingjun Han, Zhan Li, and Lu~Qi.
\newblock {ACC} for log canonical threshold polytopes.
\newblock {\em arXiv:1706.07628}, 2017.

\bibitem[HLS19]{HLS19}
Jingjun Han, Jihao Liu, and V.~V. Shokurov.
\newblock {ACC} for minimal log discrepancies of exceptional singularities.
\newblock {\em arXiv:1903.04338v1}, 2019.

\bibitem[HMX14]{HMX14}
Christopher~D. Hacon, James McKernan, and Chenyang Xu.
\newblock A{CC} for log canonical thresholds.
\newblock {\em Ann. of Math. (2)}, 180(2):523--571, 2014.

\bibitem[HX15]{HX15}
Christopher~D. Hacon and Chenyang Xu.
\newblock Boundedness of log {C}alabi-{Y}au pairs of {F}ano type.
\newblock {\em Math. Res. Lett.}, 22(6):1699--1716, 2015.

\bibitem[Isk96]{Iskovskikh96}
V.~A. Iskovskikh.
\newblock On a rationality criterion for conic bundles.
\newblock {\em Mat. Sb.}, 187(7):75--92, 1996.

\bibitem[Kaw15]{Kawakita15index}
Masayuki Kawakita.
\newblock The index of a threefold canonical singularity.
\newblock {\em Amer. J. Math.}, 137(1):271--280, 2015.

\bibitem[Kaw18]{Kawatika18}
Masayuki Kawakita.
\newblock On equivalent conjectures for minimal log discrepancies on smooth
  threefolds.
\newblock {\em arXiv:1803.02539, to appear in Journal of Algebraic Geometry},
  2018.

\bibitem[KM98]{KM98}
J\'anos Koll\'ar and Shigefumi Mori.
\newblock {\em Birational geometry of algebraic varieties}, volume 134 of {\em
  Cambridge Tracts in Mathematics}.
\newblock Cambridge University Press, Cambridge, 1998.
\newblock With the collaboration of C. H. Clemens and A. Corti, Translated from
  the 1998 Japanese original.

\bibitem[Kod60]{Kodaria60}
K.~Kodaira.
\newblock On compact complex analytic surfaces. {I}.
\newblock {\em Ann. of Math. (2)}, 71:111--152, 1960.

\bibitem[Kod63]{Kodaira63}
K.~Kodaira.
\newblock On compact analytic surfaces. {II}, {III}.
\newblock {\em Ann. of Math. (2) 77 (1963), 563--626; ibid.}, 78:1--40, 1963.

\bibitem[Kol92]{Fli92}
J\'anos Koll\'ar.
\newblock {\em Flips and abundance for algebraic threefolds}.
\newblock Soci\'et\'e Math\'ematique de France, Paris, 1992.
\newblock Papers from the Second Summer Seminar on Algebraic Geometry held at
  the University of Utah, Salt Lake City, Utah, August 1991, Ast\'erisque No.
  211 (1992).

\bibitem[Liu18]{LJH18}
Jihao Liu.
\newblock Toward the equivalence of the {ACC} for $a$-log canonical thresholds
  and the {ACC} for minimal log discrepancies.
\newblock {\em arXiv:1809.04839}, 2018.

\bibitem[MN18]{mustata-nakamura18}
Mircea Musta\c{t}\u{a} and Yusuke Nakamura.
\newblock A boundedness conjecture for minimal log discrepancies on a fixed
  germ.
\newblock In {\em Local and global methods in algebraic geometry}, volume 712
  of {\em Contemp. Math.}, pages 287--306. Amer. Math. Soc., Providence, RI,
  2018.

\bibitem[MP04]{MP04}
James McKernan and Yuri Prokhorov.
\newblock Threefold thresholds.
\newblock {\em Manuscripta Math.}, 114(3):281--304, 2004.

\bibitem[MP08]{MoriProkhorov08}
Shigefumi Mori and Yuri Prokhorov.
\newblock On {$\mathbb{Q}$}-conic bundles.
\newblock {\em Publ. Res. Inst. Math. Sci.}, 44(2):315--369, 2008.

\bibitem[Nak16]{nakamura2016}
Yusuke Nakamura.
\newblock On minimal log discrepancies on varieties with fixed gorenstein
  index.
\newblock {\em Michigan Math. J.}, 65(1):165--187, 03 2016.

\bibitem[Pro18]{Prokhorov18}
Yu.~G. Prokhorov.
\newblock The rationality problem for conic bundles.
\newblock {\em Uspekhi Mat. Nauk}, 73(3(441)):3--88, 2018.

\bibitem[PS01]{PS01}
Yu.~G. Prokhorov and V.~V. Shokurov.
\newblock The first fundamental theorem on complements: from global to local.
\newblock {\em Izv. Ross. Akad. Nauk Ser. Mat.}, 65(6):99--128, 2001.

\bibitem[PS09]{PS09}
Yu.~G. Prokhorov and V.~V. Shokurov.
\newblock Towards the second main theorem on complements.
\newblock {\em J. Algebraic Geom.}, 18(1):151--199, 2009.

\bibitem[Sho79]{Sho79}
Vyacheslav~V. Shokurov.
\newblock Smoothness of a general anticanonical divisor on a {F}ano variety.
\newblock {\em Izv. Akad. Nauk SSSR Ser. Mat.}, 43(2):430--441, 1979.

\bibitem[Sho92]{Sho92}
Vyacheslav~V. Shokurov.
\newblock Three-dimensional log perestroikas.
\newblock {\em Izv. Ross. Akad. Nauk Ser. Mat.}, 56(1):105--203, 1992.

\bibitem[Sho96]{Sho96lcflip}
V.~V. Shokurov.
\newblock {$3$}-fold log models.
\newblock {\em J. Math. Sci.}, 81(3):2667--2699, 1996.
\newblock Algebraic geometry, 4.

\bibitem[Sho00]{Sho00}
V.~V. Shokurov.
\newblock Complements on surfaces.
\newblock {\em J. Math. Sci. (New York)}, 102(2):3876--3932, 2000.
\newblock Algebraic geometry, 10.

\bibitem[Sho04a]{Sho04comp}
Vyacheslav Shokurov.
\newblock Bounded complements: conjecture, examples and applications.
\newblock In {\em Oberwolfach {R}eports. {V}ol. 1, no. 3}, volume~1, pages
  1674--1676. EMS Publishing House, Berlin, 2004.

\bibitem[Sho04b]{Sho04}
Vyacheslav~V. Shokurov.
\newblock Letters of a bi-rationalist. {V}. {M}inimal log discrepancies and
  termination of log flips.
\newblock {\em Tr. Mat. Inst. Steklova}, 246(Algebr. Geom. Metody, Svyazi i
  Prilozh.):328--351, 2004.

\bibitem[Sho09]{Sho09}
V.~V. Shokurov.
\newblock Letters of a bi-rationalist. {VII}. {O}rdered termination.
\newblock {\em Tr. Mat. Inst. Steklova}, 264(Mnogomernaya Algebraicheskaya
  Geometriya):184--208, 2009.

\bibitem[Sho20]{Sho19}
Vyacheslav~V. Shokurov.
\newblock Existence and boundedness of $n$-complements.
\newblock {\em preprint}, 2020.

\bibitem[Xu19a]{Xu19}
Chenyang Xu.
\newblock A minimizing valuation is quasi-monomial.
\newblock {\em arXiv:1907.01114, to appear in Annals of Mathematics}, 2019.

\bibitem[Xu19b]{Xuyanning19-1}
Yanning Xu.
\newblock Complements on log canonical {F}ano varieties.
\newblock {\em arXiv:1901.03891}, 2019.

\bibitem[Xu19c]{Xuyanning19-2}
Yanning Xu.
\newblock Some results about the index conjecture for log {C}alabi-{Y}au pairs.
\newblock {\em arXiv:1905.00297}, 2019.

\end{thebibliography}
\newcommand{\etalchar}[1]{$^{#1}$}

\end{document}